\documentclass[11pt, a4paper,oneside, reqno]{amsart}
\usepackage[english]{babel}
\usepackage{amsmath, amsthm, amsfonts, mathrsfs, amssymb, amscd}
\usepackage{bm}
\usepackage{mathtools}
\usepackage{xcolor}
\definecolor{grey}{rgb}{0.5,0.5,0.5}
\mathtoolsset{centercolon}
\usepackage{accents}
\usepackage{cancel}

\usepackage[toc]{appendix}

\usepackage[shortlabels]{enumitem}
\usepackage[backgroundcolor=yellow, colorinlistoftodos,prependcaption,textsize=small,textwidth=25mm]{todonotes}
\usepackage[colorlinks, citecolor = blue, urlcolor={red}]{hyperref}
\usepackage{geometry}
\allowdisplaybreaks

\newtheorem{theorem}{Theorem}
\newtheorem{corollary}[theorem]{Corollary}
\newtheorem{lemma}[theorem]{Lemma}

\newtheorem{proposition}[theorem]{Proposition}

\theoremstyle{remark} 
\newtheorem{remark}[theorem]{Remark}

\theoremstyle{definition} 
\newtheorem{definition}[theorem]{Definition}

\numberwithin{theorem}{section}
\numberwithin{equation}{section}

\def\N{{\mathbb N}}

\def\R{{\mathbb R}}
\def\C{{\mathbb C}}

\def\F{{\mathbb F}}
\def\G{{\mathbb G}}

\def\U{{\mathbb U}}

\newcommand{\om}{\omega}
\newcommand{\Om}{\Omega}

\renewcommand{\vec}[1]{\bm{#1}}

\newcommand{\Dom}{\mathcal{O}}
\newcommand{\BDom}{\partial\Dom}

\renewcommand{\Re}{\operatorname{Re}}

\newcommand{\ii}{{\rm i}} 

\newcommand{\mc}{\mathcal}

\newcommand{\mrm}{\mathrm}

\newcommand{\Hinf}{H^{\infty}}
\newcommand{\RR}{\mathbb{R}}

\newcommand{\CC}{\mathbb{C}}
\newcommand{\NN}{\mathbb{N}}
\newcommand{\ZZ}{\mathbb{Z}}

\newcommand{\OO}{\mathcal{O}}

\renewcommand{\SS}{\mathcal{S}}
\newcommand{\PP}{\mathcal{P}}
\newcommand{\II}{\mathcal{I}}
\newcommand{\FF}{\mathcal{F}}
\newcommand{\half}{\frac{1}{2}}
\newcommand{\RRdh}{\RR^d_+}
\newcommand{\RRd}{\RR^d}
\newcommand{\Cc}{C_{\mathrm{c}}}
\renewcommand{\d}{\partial}
\newcommand{\del}{\Delta}
\newcommand{\grad}{\nabla}

\newcommand{\eps}{\varepsilon}
\newcommand{\ph}{\varphi}

\newcommand{\gam}{\gamma}

\newcommand{\Tr}{\operatorname{Tr}}
\newcommand{\ext}{\operatorname{ext}}
\newcommand{\loc}{{\rm loc}}
\renewcommand{\tilde}[1]{\widetilde{#1}}
\renewcommand{\hat}[1]{\widehat{#1}}
\newcommand{\dist}{\operatorname{dist}}
\newcommand{\XX}{\mathbb{X}}
\newcommand{\DD}{\mathbb{D}}

\newcommand{\cir}[1]{\accentset{\circ}{#1}}
\renewcommand{\b}{{\rm b}}

\DeclareMathAlphabet{\mathpzc}{OT1}{pzc}{m}{it}

\newcommand\op{\mathop{\circ}\nolimits}

\DeclarePairedDelimiter\abs{\lvert}{\rvert}

\DeclarePairedDelimiter\cbrace\{\}

\DeclarePairedDelimiter{\nrm}\lVert\rVert

\newcommand{\has}[1]{\Bigl(#1\Bigr)}
\newcommand{\cbraces}[1]{\Bigl\{#1\Bigr\}}

\newcommand{\dd}{\hspace{2pt}\mathrm{d}}

\DeclareMathOperator{\UMD}{UMD}
\DeclareMathOperator{\BIP}{BIP}
\DeclareMathOperator{\id}{id}

\author[R. Denk]{Robert Denk}
\address[Robert Denk]{Universit\"at Konstanz, Fachbereich Mathematik und Statistik, 78357 Konstanz, Germany}\email{robert.denk@uni-konstanz.de}

\author[F.B. Roodenburg]{Floris B. Roodenburg}
\address[Floris Roodenburg]{Delft Institute of Applied Mathematics\\
Delft University of Technology \\ P.O. Box 5031\\ 2600 GA Delft\\The
Netherlands} \email{f.b.roodenburg@tudelft.nl}

\begin{document}
\title[]{Higher-order regularity for a structurally damped plate equation on rough domains}

\makeatletter
\@namedef{subjclassname@2020}{
  \textup{2020} Mathematics Subject Classification}
\makeatother


\thanks{FBR is supported by the VICI grant VI.C.212.027 of the Dutch Research Council (NWO)}

\begin{abstract}
We prove well-posedness and higher-order regularity for a linear structurally damped plate equation with inhomogeneous Dirichlet--Neumann boundary conditions on the half-space and on bounded domains. To this end, we study maximal regularity properties of the related first-order system on weighted Sobolev spaces of arbitrarily high smoothness. In particular, we consider Sobolev spaces with power weights that measure the distance to the boundary. This allows us to avoid unnatural compatibility conditions for the data and treat the plate equation with rough inhomogeneous boundary conditions on bounded $C^{1,\kappa}$-domains, where $\kappa\in (0,1)$ depends on the exponent of the spatial power weight, but is independent of the smoothness of the data. Our methods can serve as an example to treat more complicated mixed-order systems as well.
\end{abstract}

\maketitle

\setcounter{tocdepth}{1}
\tableofcontents

\section{Introduction}
In this paper, we consider the linear structurally damped plate equation with inhomogeneous Dirichlet--Neumann boundary conditions given by
\begin{equation}\label{eq:plate_eq_intro}
\left\{
  \begin{aligned}
  &\d_t^2 u +\del^2 u-\rho\del \d_t u  = f\qquad &(t,x)&\in (0,T)\times \OO,\\
   &u=g_0\qquad &(t,x)&\in (0,T)\times \d\OO,\\
   &\d_{\vec{n}}u=g_1\qquad &(t,x)&\in (0,T)\times \d\OO,\\
   &u|_{t=0}=(\d_t u)|_{t=0}=0\qquad &x&\in \OO.
\end{aligned}\right.
\end{equation}
Here, $\rho>0$ is a fixed parameter and $\d_{\vec{n}}$ denotes the inner normal derivative. Furthermore, $T\in (0,\infty]$ is a time horizon and we consider the case of the half-space $\OO=\RRdh$ and the case of a bounded domain $\OO\subseteq \RR^{d}$.

The plate equation  occurs as a linear model for vibrating stiff objects where the potential energy involves curvature-like terms which lead to the bi-Laplacian as the main `elastic' operator, see, e.g., \cite[Chapter 12]{Leis86} and \cite{Russell84}. The equation \eqref{eq:plate_eq_intro} relates to the Kirchhoff model of thin plates; for mathematical modelling of thin elastic plates, we also refer to \cite[Chapter 2]{Lagnese89}. 
In the undamped situation (corresponding to $\rho=0$ in \eqref{eq:plate_eq_intro}), there is no energy dissipation,  and the equation has no smoothing effect. Apart from  thermoelastic damping effects (as discussed in a semigroup context in, e.g., 
\cite{Lasiecka-Triggiani98}), a standard damping model is given
by structural damping as above. This model  describes a situation where higher frequencies are more strongly damped than lower frequencies. 
In \eqref{eq:plate_eq_intro}, the damping term has `half of the order'
of the leading elastic term. For mathematical models of internal damping mechanisms leading to structural damping of second-order, 
we also refer to \cite[Section~4.2]{Russell92}. The generation of an analytic semigroup for 
the plate equation with structural damping was investigated, e.g., in \cite{Chen-Triggiani89} and \cite{DS15}.

The Dirichlet--Neumann boundary conditions in \eqref{eq:plate_eq_intro} are also called generalised 
Dirichlet conditions and are known in the context of the plate equation as clamped boundary conditions. 
We remark that these boundary conditions are mathematically more involved than the boundary conditions $u=\Delta u =0$, which can be treated by operator-theoretic methods, see, e.g., \cite{CCD08, CS05, SV10}. \\

In this paper, we study well-posedness and regularity of  \eqref{eq:plate_eq_intro} in a framework with weighted $L^q$-$L^p$-Sobolev spaces. Specifically, we consider spatial power weights of the form $w_{\gam}^{\d\OO}:=\dist(\cdot, \d\OO)^{\gam}$ for some suitable $\gam>-1$. These weights are well suited to control boundary singularities of solutions and their derivatives, see \cite{LLRV24,LLRV25, LV18}. Our main result for \eqref{eq:plate_eq_intro} on bounded domains reads as follows. This is a special case of Theorem \ref{thm:main_inhomgenous}, where for simplicity we omit the temporal weights. We recall that $F_{q,p}^s$ and $B^{s}_{p,p}$ denote the Triebel--Lizorkin and Besov spaces, respectively.
\begin{theorem}\label{thm:intro_plateeq}
    Let $p,q\in (1,\infty)$, $\kappa\in (0,1)$, $\gam\in ((1-\kappa)p-1, p-1)$ and let $k\geq 2$ be an integer. Suppose that $1-\frac{1}{q}\neq \frac{\gam+1}{2p}$ and $\half-\frac{1}{q}\neq \frac{\gam+1}{2p}$. Furthermore, let $\OO$ be a bounded $C^{1,\kappa}$-domain. Then for all
    \begin{align*}
        f&\in L^q(\RR_+; W^{k-2,p}(\OO, w_{\gam+kp}^{\d\OO})), \\
        g_0&\in F^{1-\frac{\gam+1}{2p}}_{q,p}(\RR_+; L^p(\d\OO))\cap L^q(\RR_+, B^{2-\frac{\gam+1}{p}}_{p,p}(\d\OO)) \text{ with }g_0|_{t=0}=0\text{ if }1-\tfrac{1}{q}> \tfrac{\gam+1}{2p},\\
        g_1&\in F^{\half-\frac{\gam+1}{2p}}_{q,p}(\RR_+; L^p(\d\OO))\cap L^q(\RR_+, B^{1-\frac{\gam+1}{p}}_{p,p}(\d\OO))\text{ with }g_1|_{t=0}=0\text{ if }\tfrac{1}{2}-\tfrac{1}{q}> \tfrac{\gam+1}{2p},
    \end{align*}
   there exists a unique solution $$u\in W^{2,q}(\RR_+; W^{k-2,p}(\OO, w_{\gam+kp}^{\d\OO}))\cap L^q(\RR_+; W^{k+2,p}(\OO, w_{\gam+kp}^{\d\OO}))$$ to \eqref{eq:plate_eq_intro}. Moreover, this solution $u$ depends continuously on the data $f$, $g_0$ and $g_1$. 
\end{theorem}

To prove Theorem \ref{thm:intro_plateeq}, we rewrite \eqref{eq:plate_eq_intro} as a first-order system  for $v= (u,\partial_t u)^\top$, i.e.,
\begin{equation}\label{eq:Cauchy-Intro}
    \d_t v + A v = \binom 0 f,\quad \text{ with }A=\begin{pmatrix}
        0& -\id\\ \del^2& -\rho \del
    \end{pmatrix}.
\end{equation}
We establish results on sectoriality and $R$-sectoriality of the operator matrix $A$. With this operator-theoretic approach, in the first step, we obtain unique solvability and maximal regularity results for \eqref{eq:plate_eq_intro} with homogeneous boundary conditions  $u=\partial_{\vec n}u=0$. In a second step, we include inhomogeneous boundary data, using trace results for weighted spaces obtained in \cite{DR25}. The theory for temporal traces in the weighted setting is delicate and therefore we only consider zero initial values in \eqref{eq:plate_eq_intro}.

The main novelties of Theorem \ref{thm:intro_plateeq} can be summarised as follows.
\begin{enumerate}[(i)]
\item \emph{Fewer compatibility conditions for the data}. Due to the introduction of a spatial power weight, there is no need to include additional conditions of compatibility type on the right-hand side of \eqref{eq:plate_eq_intro}. We discuss this in more detail below.
\item \emph{Higher-order regularity}. We establish maximal regularity results in weighted Sobolev spaces of arbitrarily high smoothness.
\item \emph{Rough domains}. The theory applies to bounded $C^{1,\kappa}$-domains, independent of the regularity order $k$. In this sense, one can speak of a rough domain. 
\item \emph{Rough boundary data}. We obtain higher-order regularity of the solution in weighted spaces, even if the boundary data is in the standard low-regularity trace spaces. In this sense, we can also handle rough boundary data.
\end{enumerate}

It is remarkable that, due to the introduction of a spatial power weight, we can give a fairly complete well-posedness and regularity theory for \eqref{eq:plate_eq_intro} that avoids additional unnatural compatibility conditions as encountered in \cite{DS15}. It was shown in \cite{DS15} that the operator $A$ in \eqref{eq:Cauchy-Intro} is $R$-sectorial in the basic space $W^{2,p}_0(\RRdh)\times L^p(\RRdh)$, but \emph{not} in the space $W^{2,p}(\RRdh)\times L^p(\RRdh)$. Therefore, one has to include the additional boundary conditions $h_1=\partial_{\vec n} h_1 =0$ on $\partial\RRdh$ for the first component of the  right-hand side of the equation $(\lambda-A)v=h$. It was shown in \cite{DD11} that such compatibility conditions are in general necessary for the sectoriality in unweighted spaces. For higher-order regularity results, even more compatibility conditions on the data have to be assumed, see \cite[Theorem~3.6]{DD11}. We can avoid any compatibility conditions on the data, independent of the smoothness order of the solution, by introducing weighted Sobolev spaces. We obtain $R$-sectoriality of $A$ in the basic space \[ \XX^{k,p}_\gam(\RRdh) := W^{k,p}(\RRdh,w_{\gam+kp})\times W^{k-2,p}(\RRdh, w_{\gam+kp}),\]
where $p\in (1,\infty)$, $\gam\in (-1,p-1)$, $k\in\N_0$ and $w_{\gam+kp}=\dist(\cdot, \d\RRdh)^{\gam+kp}$, see Theorem~\ref{thm:Rsect_dom}.  In particular, setting $k=2$, we can replace  the basic space $W^{2,p}_0(\RRdh)\times L^p(\RRdh)$ from \cite{DS15} by the weighted space $W^{2,p}(\RRdh,w_{\gam+2p})\times L^p(\RRdh, w_{\gam+2p})$. The advantage of the weighted space $W^{2,p}(\RRdh,w_{\gam+2p})$ is that the Dirichlet and Neumann trace operators do not exist, in this way omitting the compatibility conditions.

We obtain $R$-sectoriality of $A$ on $\XX^{k,p}_\gam(\RRdh)$ by the following method. We start with $R$-sectoriality of $A$ on the space $W^{2,p}_0(\RRdh, w_\gam)\times L^p(\RRdh, w_\gam)$ for $\gam\in (-1,p-1)$, which can be obtained by extending the results in \cite{DS15} to the weighted setting. Then we lower the smoothness of the spaces by extrapolating the 
operator in the sense of Amann's interpolation-extrapolation scales (see Section~\ref{subsec:int_ext} for details on these spaces). For the extrapolated 
operator on $L^p(\RRdh, w_\gam) \times W^{-2,p}(\RRdh, w_\gam)$, due to the low regularity, the boundary conditions are not seen, so there are no compatibility conditions needed. This proves $R$-sectoriality on $\XX^{0,p}_\gam(\RRdh)$.

Afterwards, with an induction argument on $k$, we raise the smoothness and the weight with a similar technique as developed in \cite{LLRV24} which essentially relies on Hardy's inequality (Lemma \ref{lem:Hardy}). In this way, we can get $R$-sectoriality in weighted spaces $\XX^{k,p}_\gam(\RRdh)$ of arbitrary smoothness $k\in \NN_0$ without artificial compatibility conditions.\\ 

Once the $R$-sectoriality is obtained on $\RRdh$, a localisation argument allows us to treat the case of bounded domains. For the localisation, we make use of a special diffeomorphism known as the \emph{Dahlberg--Kenig--Stein pullback}, which straightens the boundary and preserves the distance to the boundary. The blow-up of higher-order derivatives of this pullback near the boundary are compensated by the weights in the space. Therefore, we can establish the results for domains with $C^{1,\kappa}$-boundary, even for higher-order regularity. Note that even for the standard regularity $k=2$, the classical localisation would require a domain of class $C^4$, as can be seen from \cite[Section~5]{DS15}. The Dahlberg--Kenig--Stein pullback was already used in \cite{LLRV25} to obtain higher-order regularity for the Laplacian on rough domains. \\

The operator $A$ in \eqref{eq:Cauchy-Intro} is an example of a mixed-order system, also called a Douglis--Nirenberg system as considered in \cite{ADN59, ADN64}. In general, semigroup theory for mixed-order boundary value problems is not yet available. For a general mixed-order system in the whole space, the Newton polygon theory developed in, e.g.,  \cite{DK13,Denk-Saal-Seiler08,Gindikin-Volevich92} leads to equivalent conditions for maximal regularity. For mixed-order boundary value problems, the Newton polygon approach is known only in special cases, see, e.g., \cite{Denk-Volevich02} and for some recent results, see \cite{Neuttiens-Sauer25}. One of the difficulties for mixed-order systems lies in the fact that one cannot take a product of $L^p$-spaces as a basic space, due to the Douglis--Nirenberg order structure of the system. Instead, the basic space necessarily is a tuple of Sobolev spaces, where now the question of compatibility conditions arises, similarly to the discussion above. With the present paper, we hope to provide a first step for a maximal regularity theory for mixed-order systems in weighted spaces, omitting compatibility conditions.

\subsection*{Outline} The paper is organised as follows. In Section \ref{sec:prelim}, we introduce notation and recall the necessary background on \(R\)-sectoriality, maximal regularity, and interpolation-extrapolation scales. Section \ref{sec:spaces} is devoted to weighted Sobolev spaces with power weights on the half-space and on rough domains. In Sections \ref{sec:Rsect} and \ref{sec:Rsect_domains}, we establish \(R\)-sectoriality of the elliptic operator associated with the damped plate equation on the half-space and on bounded domains, respectively. Finally, in Section \ref{sec:MR}, we apply these results to prove well-posedness and higher-order maximal regularity for the structurally damped plate equation with inhomogeneous boundary conditions.

\section{Preliminaries}\label{sec:prelim}

\subsection{Notation}
We denote by $\NN_0$ and $\NN_1$ the sets of natural numbers starting at $0$ and $1$, respectively. For $a\in\RR$, we use the notation $(a)_+=a$ if $a\geq 0$ and $(a)_+=0$ otherwise.

For $d\in\NN_1$ the half-space is given by $\RRdh=\RR_+\times\RR^{d-1}$, where $\RR_+=(0,\infty)$ and for $x\in \RRdh$ we write $x=(x_1,\tilde{x})$ with $x_1\in \RR_+$ and $\tilde{x}\in \RR^{d-1}$. For $\gam\in\RR$, $\OO\subseteq \RR^d$ open and $x\in\OO$ we define the power weight $w_{\gam}^{\d\OO}(x):=\mrm{dist}(x,\d\OO)^\gam$. For notational convenience, we write $w_{\gam}:=w_{\gam}^{\smash{\d\RRdh}}$.

For two topological vector spaces $X$ and $Y$, the space of continuous linear operators is denoted by $\mc{L}(X,Y)$ and $\mc{L}(X):=\mc{L}(X,X)$. Unless specified otherwise, $X$ will always denote a Banach space with norm $\|\cdot\|_X$ and the dual space is $X':=\mc{L}(X,\CC)$.

For a linear operator $A:X\supseteq D(A)\to X$ on a Banach space $X$, we denote by $\sigma(A)$ and $\rho(A)$  the spectrum and resolvent set, respectively. For $\lambda\in\rho(A)$, the resolvent operator is given by $R(\lambda,A)=(\lambda-A)^{-1}\in \mc{L}(X)$.

We write $f\lesssim g$ (resp. $f\gtrsim g$) if there exists a constant $C>0$, possibly depending on parameters which will be clear from the context or will be specified in the text, such that $f\leq Cg$ (resp. $f\geq Cg$). Furthermore, $f\eqsim g$ means $f\lesssim g$ and $g\lesssim  f$.\\

For an open and non-empty set $\OO\subseteq \RR^d$ and $\ell\in\NN_0\cup\{\infty\}$, the space $C^\ell(\OO;X)$ denotes the space of $\ell$-times continuously differentiable functions from $\OO$ to some Banach space $X$. As usual, this space is equipped with the compact-open topology. In the case $\ell=0$ we write $C(\OO;X)$ for $C^0(\OO;X)$. Furthermore, we write $C^\ell_\b(\OO;X)$ for the space of all functions $f\in C^\ell(\OO;X)$ such that $\d^{\alpha} f$ is bounded on $\OO$ for all multi-indices $\alpha\in \NN_0^d$ with $|\alpha|\leq \ell$. 

Let $\Cc^{\infty}(\OO;X)$ be the space of compactly supported smooth functions on $\OO$ equipped with its usual inductive limit topology. The space of $X$-valued distributions is given by $\mc{D}'(\OO;X):=\mc{L}(\Cc^{\infty}(\OO);X)$. Moreover, $\Cc^{\infty}(\overline{\OO};X)$ is the space of smooth functions with their support in a compact set contained in $\overline{\OO}$.

We denote the Schwartz space by $\SS(\RRd;X)$, and $\SS'(\RRd;X):=\mc{L}(\SS(\RRd);X)$ is the space of $X$-valued tempered distributions. For $\OO\subseteq \RR^d$ we define $\SS(\OO;X):=\{u|_{\OO}: u\in\SS(\RRd;X)\}$. For $f\in \SS(\RRd;X)$ we define the $d$-dimensional Fourier transform $(\mc{F} f)(\xi):=\hat{f}(\xi):=\int_{\RRd} f(x)e^{-\ii x\cdot\xi}\dd x$ for $\xi\in \RR^d$, which extends to $\SS'(\RRd;X)$ by duality.

Finally, for $\theta\in(0,1)$ and a compatible couple $(X,Y)$ of Banach spaces, the complex interpolation space is denoted by $[X,Y]_{\theta}$.

\subsection{\texorpdfstring{$R$-}{R-}sectoriality and maximal regularity}\label{subsec:prelim-Rsect}
In this section, we recall the tools from maximal regularity theory required for proving our results about the damped plate equation. For a detailed study of $R$-sectoriality and maximal regularity, we refer to \cite{DHP03, HNVW24, KW04}.

\subsubsection{\texorpdfstring{$R$-}{R-}sectoriality}
Throughout this section, $X$ and $Y$ are Banach spaces. Let $(\Omega, \mathfrak{F}, \mathbb{P})$ be a probability space. Then  $\eps : \Omega \to \CC$ is called a \emph{Rademacher variable} if the random variable $\eps$ is uniformly distributed over $\{z\in \CC: |z|=1\}$. A \emph{Rademacher sequence $(\eps_n)_{n\geq 1}$} is a sequence that consists of independent Rademacher variables $\eps_n$. For $N\in\NN_1$ we define the \emph{random sequence space $\eps_N^2(X)$} as the Banach space of finite sequences $(x_n)_{n=1}^N$ in $X$, endowed with the norm
\begin{equation*}
  \big\|(x_n)_{n=1}^N\big\|_{\eps_N^2(X)}:= \Big\|\sum_{n=1}^N \eps_n x_n\Big\|_{L^2(\Omega; X)}.
\end{equation*}
A family of bounded operators $\mc{T}\subseteq \mc{L}(X,Y)$ is called \emph{$R$-bounded} if there exists a constant $C>0$ such that for all $N\in\NN_1$, $(T_n)_{n=1}^N\subseteq \mc{T}$ and $(x_n)_{n=1}^N\subseteq X$, we have
\begin{equation*}
  \Big\|\sum_{n=1}^N\eps_n T_n x_n\Big\|_{L^2(\Om; Y)}\leq C \Big\|\sum_{n=1}^N \eps_n x_n\Big\|_{L^2(\Om; X)}.
\end{equation*}
In particular, we note that $R$-boundedness implies uniform boundedness, while the converse only holds under more geometric restrictions on the Banach spaces $X$ and $Y$ (for instance if $X$ and $Y$ are Hilbert spaces). Moreover, any singleton operator $\{T\}\subseteq \mc{L}(X,Y)$ is $R$-bounded.\\

For $\om\in(0,\pi)$, let $\Sigma_\om=\{z\in \CC\setminus\{0\}: |\arg(z)|< \om\}$ be the sector in the complex plane.
\begin{definition}[($R$-)sectoriality]\hspace{2em}
\begin{enumerate}[(i)]
  \item An injective, closed linear operator $(A, D(A))$ with dense domain and dense range on a Banach space $X$ is called \emph{sectorial} if there exists an $\om\in(0,\pi)$ such that $\sigma(A)\subseteq \overline{\Sigma}_\om$ and 
  \begin{equation*}
    \{zR(z,A): z\in \CC\setminus \overline{\Sigma}_{\om}\}
  \end{equation*}
  is bounded. The angle of sectoriality $\om(A)$ is defined as the infimum over all possible $\om$.
  \item A sectorial operator $(A, D(A))$ is called \emph{$R$-sectorial} if for some $\sigma\in (\om(A),\pi)$ the family
  \begin{equation*}
    \{zR(z,A): z\in \CC\setminus \overline{\Sigma}_{\sigma}\}
  \end{equation*}
  is $R$-bounded. The angle of $R$-sectoriality $\om_R(A)$ is defined as the infimum over all possible $\sigma$. 
\end{enumerate}
\end{definition}

Occasionally, we will also use the notions of \emph{bounded $\Hinf$-calculus} and \emph{bounded imaginary powers} ($\BIP$), see \cite[Chapter 10]{HNVW17} and \cite[Chapter 15]{HNVW24}. In particular, if $X$ is a $\UMD$ Banach space, then a bounded $\Hinf$-calculus implies bounded imaginary powers, which again implies $R$-sectoriality. We recall that $X$ is a $\UMD$ (unconditional martingale differences) Banach space if and only if the Hilbert transform extends to a bounded operator on $L^p(\RR;X)$ for all $p\in(1,\infty)$. The $\UMD$ property is a necessary condition for many results on vector-valued spaces, see, e.g., \cite{HNVW16, HNVW24}. We list the following relevant properties of $\UMD$ spaces, see \cite[Chapter 4 \& 5]{HNVW16}.
\begin{enumerate}[(i)]
  \item Hilbert spaces are $\UMD$ spaces. In particular, $\CC$ is a $\UMD$ Banach space.
  \item For $p\in(1,\infty)$, $(S, \Sigma,\mu)$ a $\sigma$-finite measure space and $X$ a $\UMD$ Banach space, $L^p(S;X)$ is a $\UMD$ Banach space.
      \item $\UMD$ Banach spaces are reflexive.
\end{enumerate}

\subsubsection{Maximal regularity}
Let $T\in(0, \infty] $ and define the time interval $I:=(0,T)$. Let $A$ be a linear operator with domain $D(A)$ on a Banach space $X$. For $f\in L^1_\loc(\overline{I};X)$ consider the abstract Cauchy problem
\begin{equation}\label{eq:Cauchy}
  \begin{aligned}
  \d_t u(t)+A u(t)&=f(t),\qquad  t\in I, \\
    u(0)&=0.
  \end{aligned}
\end{equation}
We call a strongly measurable function $u:I\to X$ a \emph{strong solution to \eqref{eq:Cauchy}} if $u$ takes values in $D(A)$ almost everywhere, $Au\in L^1_\loc(\overline{I};X)$ and $u$ solves
\begin{equation*}
  u(t)+\int_0^t Au(s)\dd s=\int_0^t f(s)\dd s, \quad \text{ for almost all }t\in I.
\end{equation*}
Moreover, for $q\in (1,\infty)$, $v\in A_q(I)$ and $f\in L^q(I,v;X)$, a strong solution $u$ to \eqref{eq:Cauchy} is called an \emph{$L^q(v)$-solution} if $Au \in L^q(I,v;X)$. Here, $v$ is a time weight that belongs to the class of \emph{Muckenhoupt weights $A_q(I)$}, see Section \ref{sec:spaces}.

\begin{definition}[Maximal $L^q(v)$-regularity]
  A linear operator $A$ on a Banach space $X$ has \emph{maximal $L^q(v)$-regularity on $I$} if for all $f\in L^q(I,v;X)$ the Cauchy problem \eqref{eq:Cauchy} has a unique $L^q(v)$-solution $u$ on $I$ and
  \begin{equation*}
    \|Au\|_{L^q(I,v;X)}\leq C\|f\|_{L^q(I,v;X)},
  \end{equation*}
  where the constant $C$ is independent of $f$.
\end{definition}

The following sufficient condition for maximal regularity in terms of $R$-sectoriality holds: if $X$ is a $\UMD$ Banach space and $\lambda + A$ is $R$-sectorial with angle $\om_R(\lambda+ A)< \frac{\pi}{2}$ for some $\lambda$ large enough, then $A$ has maximal $L^q(v)$-regularity on $(0,T)$ for any $T<\infty$, $q\in(1,\infty)$ and $v\in A_q(I)$. If $\lambda=0$, then we obtain maximal $L^q(v)$-regularity on $\RR_+$ as well. This statement follows from combining the results in \cite[Theorems 17.3.1(2), 17.2.26(1) \& 17.2.39]{HNVW24}.

\subsection{Interpolation-extrapolation spaces}\label{subsec:int_ext} We recall the theory of interpolation-extra\-polation scales. For a detailed study, we refer to \cite[Chapter V]{Am95} or \cite[Section 6.3]{Ha06}.

Let $A$ be a sectorial operator on a Banach space $X$ that satisfies $0\in \rho(A)$. Hence, $(X, \|A^{-1}\cdot \|_X)$ is a normed space and we can define the extrapolated space
\begin{equation*}
  E_{-1,A}: = (X, \|A^{-1}\cdot \|_X)^{\sim},
\end{equation*}
where $\sim$ denotes the completion of the space. Since $X\hookrightarrow E_{-1,A}$ and $\overline{D(A)}=X$, we have that $D(A)\ni x\mapsto Ax\in X $ extends to an isometric isomorphism $T_{-1,A}$ from $X$ to $E_{-1,A}$. The operator $A$ induces a closed linear operator $A_{-1}:= T_{-1,A} A T_{-1,A}^{-1}$ which satisfies $A_{-1}|_{D(A)}=A$. Since ($R$-)sectoriality and boundedness of the $\Hinf$-calculus are preserved under similarity transforms, we obtain the following result.
\begin{proposition}\label{prop:int_ext_properties}
  Let $A$ be a sectorial operator on a Banach space $X$ that satisfies $0\in \rho(A)$. Then $A_{-1}$ with $D(A_{-1})=X$ is the closure of $A$ in $E_{-1,A}$. Moreover, the following statements hold:
  \begin{enumerate}[(i)]
    \item $A_{-1}$ on $E_{-1,A}$ is sectorial with angle $\om(A_{-1})=\om(A)$,
    \item if $A$ is $R$-sectorial on $X$, then so is $A_{-1}$ on $E_{-1,A}$ and $\om_{R}(A)=\om_{R}(A_{-1})$,
    \item if $A$ has a bounded $\Hinf$-calculus on $X$, then so does $A_{-1}$ on $E_{-1,A}$ and $\om_{\Hinf}(A)=\om_{\Hinf}(A_{-1})$.
  \end{enumerate}
\end{proposition}
Proposition \ref{prop:int_ext_properties} shows that for $\alpha>0$ the fractional powers $(A_{-1})^{\alpha}$ are well-defined closed operators on $E_{-1,A}$ (see, e.g., \cite[Definition~15.2.2]{HNVW24}). We denote the domain of $(A_{-1})^{\alpha}$ as
\begin{equation*}
  E_{-1+\alpha, A}:= \big(D((A_{-1})^\alpha), \|(A_{-1})^{\alpha}\cdot\|_{E_{-1,A}}\big),\qquad \alpha\geq 0,
\end{equation*}
and note that $E_{0,A}=X$ by Proposition \ref{prop:int_ext_properties}. 

Let $\alpha\geq -1$ and let $A_{\alpha}$ be the realisation of $A_{-1}$ on $E_{\alpha,A}$, i.e,
\begin{align*}
  D(A_{\alpha}) & := \{x\in E_{\alpha,A} : A_{-1}x\in E_{\alpha,A}\},\\
  A_{\alpha} x& := A_{-1} x \quad \text{ for }x\in D(A_{\alpha}),
\end{align*}
and note that $A_{0}=A$. The scale $(E_{\alpha,A}, A_{\alpha})_{\alpha\geq-1}$ is called the \emph{interpolation-extrapolation scale of $A$} and forms an interpolation scale with respect to complex interpolation.
\begin{proposition}[{\cite[Theorem 6.6.9]{Ha06}}]\label{prop:Ext_compl_int}
  Let $A$ be a sectorial operator on a Banach space $X$ that satisfies $0\in \rho(A)$. Furthermore, assume that $A$ has $\BIP$. Then 
  \begin{equation*}
    E_{\alpha(1-\theta)+\beta\theta,A} = \big[E_{\alpha,A}, E_{\beta,A}\big]_\theta, \qquad \alpha,\beta\geq -1, \, \theta\in(0,1).
  \end{equation*}
\end{proposition}
Furthermore, the following characterisation of the extrapolated spaces holds.
\begin{proposition}[{\cite[Theorem V.1.4.12]{Am95}}]\label{prop:ext_adjoint}
  Let $A$ be a sectorial operator on a reflexive Banach space $X$ that satisfies $0\in \rho(A)$,
  and let $A'$ be the adjoint operator in $X'$. Then, for all $\theta \in (0,1)$ we have
  $E_{-\theta,A} = (E_{\theta, A'})'$
  with respect to the duality induced by $(X, X')$, where
  \begin{equation*}
    E_{\theta, A'} =\big(D((A')^\theta), \|(A')^\theta\cdot\|_{X'}\big).
  \end{equation*}
\end{proposition}

\subsection{Domains and localisation}
Let $\kappa\in (0,1]$ and let $\OO\subseteq \RR^{d-1}$ be open. A function $h:\OO\to \RR$ is called \emph{uniformly $\kappa$-H\"older continuous on $\OO$} if 
\begin{equation*}
  [h]_{\kappa, \OO}:= \sup_{\substack{x, y\in \OO\\ x\neq y}}\frac{|h(x)-h(y)|}{|x-y|^\kappa}<\infty.
\end{equation*}
In addition, for $\ell\in \NN_0$ we define the space of $\kappa$-H\"older continuous functions by 
\begin{equation*}
  C_{\b}^{\ell,\kappa}(\OO):=\{h\in C_{\b}^{\ell}(\OO): [\d^\alpha h]_{\kappa, \OO}<\infty\text{ for all }|\alpha|= \ell\}.
\end{equation*}
For $\kappa=0$ we write $C_\b^{\ell,0}(\OO):=C^{\ell}_\b(\OO)$. By $\Cc^{\ell,\kappa}(\OO)$ we denote the subset of functions in $C^{\ell,\kappa}(\OO)$ with compact support in $\OO$. Moreover, on $C_\b^{\ell,\kappa}(\OO)$ we define the norm
\begin{equation*}
  \|h\|_{C^{\ell,\kappa}(\OO)}:= \sum_{|\alpha|\leq \ell}\sup_{x\in\OO}|\d^\alpha h(x)|+ \sum_{|\alpha|=\ell}[\d^\alpha h]_{\kappa, \OO}.
\end{equation*}

\begin{definition}\label{def:domains}
  Let $\Dom \subseteq \R^d$ be a domain, i.e., a connected open set. Let $\ell \in \N_0$ and $\kappa\in [0,1]$.
  \begin{enumerate}[(i)]
    \item We call $\mc{O}$ a \emph{special $\Cc^{\ell,\kappa}$-domain} if, after translation and rotation, it is of the form
        \begin{equation}\label{eq:specialdomainh}
          \mc{O} = \cbrace{(x_1,\tilde{x})\in \R^d: x_1>h(\tilde{x})}
        \end{equation}
        for some $h \in \Cc^{\ell,\kappa}(\R^{d-1};\RR)$.
    \item Given a special $\Cc^{\ell,\kappa}$-domain $\mc{O}$, we define
\begin{equation*}
[\mc{O}]_{C^{\ell,\kappa}}:= \nrm{h}_{C^{\ell,\kappa}(\R^{d-1})},
\end{equation*}
where $h\in \Cc^{\ell,\kappa}(\R^{d-1};\RR)$ is such that, after rotation and translation, \eqref{eq:specialdomainh} holds. Note that $[\mc{O}]_{C^{\ell,\kappa}}$ is uniquely defined due to the compact support of $h$.
    \item We call $\mc{O}$ a \emph{$C^{\ell,\kappa}$-domain} if every boundary point $x \in \BDom$ admits an open neighbourhood $V$ such that
        \begin{equation*}
        \mc{O}\cap V = W \cap V \qquad \text{and}\qquad \BDom \cap V = \partial W \cap V
        \end{equation*}
        for some special $\Cc^{\ell,\kappa}$-domain $W$.
  \end{enumerate}
  If $\kappa=0$, then we write $C^\ell$ for $C^{\ell,0}$ in the definitions above.
\end{definition}

\begin{remark}\label{rem:special_dom}
  Note that for any $\delta>0$ and $C^\ell$-domain $\mc{O}$, the special $\Cc^\ell$-domains $W$ can always be chosen such that $[W]_{C^\ell}<\delta$. If $\kappa\in (0,1]$, $\eps \in (0,\kappa)$ and $\OO$ is a $C^{\ell,\kappa}$-domain, then for every $\delta>0$ the special $\Cc^{\ell,\kappa}$-domains $W$ can always be chosen such that $[W]_{C^{\ell, \kappa-\eps}}<\delta$. Indeed, if $h\in \Cc^{\ell,\kappa}(\RR^{d-1})$ is associated with $W$, then for any $|\alpha|=\ell$, we have
\begin{equation*}
  [\d^\alpha h]_{\kappa-\eps,\RR^{d-1}} = \sup_{\substack{x, y\in \RR^{d-1}\\ x\neq y}}\frac{|\d^\alpha h(x)-\d^\alpha h(y)|}{|x-y|^\kappa}|x-y|^\eps< \delta,
\end{equation*}
whenever $|x-y|^\eps$ is small enough. Note that for $\eps=0$, the quantity $[\d^{\alpha} h]_{\kappa,\RR^{d-1}}$ cannot be made arbitrarily small.
\end{remark}

The following lemma provides a diffeomorphism between special domains and the half-space, which is known as the \emph{Dahlberg--Kenig--Stein pullback}. This diffeomorphism straightens the boundary and preserves the distance to the boundary, while higher-order derivatives may blow up near the boundary. The weights in the spaces compensate this blow-up, see \cite{LLRV25}, where this diffeomorphism is used to obtain boundedness of the $\Hinf$-calculus for the Laplacian on rough domains. The proof of Lemma \ref{lem:localisation_weighted_blow-up} can be found in \cite[Appendix A]{LLRV25} (see also \cite{KimD07}), which is based on \cite{KK04} and \cite{Li85}. 
\begin{lemma}[{\cite[Lemma A.4]{LLRV25}}]\label{lem:localisation_weighted_blow-up}
Let $\OO$ be a special $\Cc^{1}$-domain.
Then there exist continuous functions $h_1\colon\overline{\Dom} \to \R$ and $h_2 \colon \overline{\R^d_+} \to \R$ which satisfy the following properties.
\begin{enumerate}[(i)]
    \item\label{it:lem:localisation_weighted_blow-up;Psi_inverse} The map $\Phi:\overline{\Dom} \to \overline{\R^d_+}$
    given by
    $$
    \Phi(x)= (x_1-h_1(x),\tilde{x}), \qquad x=(x_1,\tilde{x}) \in \overline{\Dom},
    $$
    is a $C^{1}$-diffeomorphism with inverse $\Phi^{-1}:\overline{\R^d_+} \to \overline{\Dom}$ given by
    $$
    \Phi^{-1}(y) = (y_1+h_2(y),\tilde{y}),\qquad y=(y_1,\tilde{y}) \in \overline{\R^d_+}.
    $$
    \item\label{it:lem:localisation_weighted_blow-up;dist_preserving} We have
    \begin{equation*}
    \begin{aligned}
    \mrm{dist}(\Phi(x),\partial\R^d_+) &\eqsim \mrm{dist}(x,\BDom), \qquad&& x \in \Dom,\\
    \mrm{dist}(\Phi^{-1}(y),\BDom)
    &\eqsim \mrm{dist}(y,\partial\R^d_+),\qquad &&y \in \R^d_+,
    \end{aligned}
    \end{equation*}
    where the implicit constants depend on $[\OO]_{C^1}$.
    \item\label{it:lem:localisation_weighted_blow-up;smoothness} We have $h_1 \in C^\infty(\Dom)$ and $h_2 \in C^\infty(\R^d_+)$.
    \end{enumerate} 
 In addition, let $\ell\in\NN_1$, $\kappa\in[0,1]$ and let $\OO$ be a special $\Cc^{\ell,\kappa}$-domain with $[\OO]_{C^{\ell,\kappa}}\leq 1$.
    \begin{enumerate}[resume*]
    \item\label{it:lem:localisation_weighted_blow-up;est} The map $\Phi$ in \ref{it:lem:localisation_weighted_blow-up;Psi_inverse} is a $C^{\ell,\kappa}$-diffeomorphism and for all $\alpha\in \N_0^d$, $\ell_0\in \{0, \dots, \ell\}$ and $\kappa_0\in[0,\kappa]$, we have
    \begin{equation*}
    \begin{aligned}
    |\d^{\alpha} h_1(x)| &\leq  C\cdot[\OO]_{C^{\ell,\kappa}}\cdot\mrm{dist}(x,\BDom)^{-(|\alpha|-\ell_0-\kappa_0)_+},\qquad  &&x \in \Dom,\\
    |\d^{\alpha} h_2(y)| &\leq C\cdot[\OO]_{C^{\ell,\kappa}}\cdot\mrm{dist}(y,\partial\R^d_+)^{-(|\alpha|-\ell_0-\kappa_0)_+},\qquad &&y \in \R^d_+,
    \end{aligned}
    \end{equation*}
    where the constant $C>0$ only depends on $\ell, \kappa, \alpha$ and $d$.
\end{enumerate}
\end{lemma}

\section{Weighted Sobolev spaces}\label{sec:spaces}
Let $\Dom$ be a domain in $\R^{d}$ with non-empty boundary $\BDom$. We call a locally integrable function $w:\OO\to (0,\infty)$ a \emph{weight}. In particular, we will deal with the following classes of weights.
\begin{enumerate}[(i)]
\item \emph{Muckenhoupt $A_p$ weights:} let $p\in(1,\infty)$ and let $\OO$ be $\RRd$, $\RRdh$ or an interval $(0,a)$ for $a>0$. A weight $w$ on $\OO$ belongs to the Muckenhoupt class $A_p(\OO)$ if 
\begin{equation*}
    [w]_{A_p(\OO)}:=\sup_{B}\Big(\frac{1}{|B|}\int_Bw(x)\dd x\Big)\Big(\frac{1}{|B|}\int_B w(x)^{-\frac{1}{p-1}}\dd x\Big)^{p-1}<\infty,
\end{equation*}
where the supremum is taken over all balls $B\subseteq \OO$. Moreover, $A_\infty(\OO):=\bigcup_{p>1}A_p(\OO)$. For a detailed study of Muckenhoupt weights and their properties, we refer to \cite[Chapter 7]{Gr14_classical_3rd}.
 \item \emph{Power weights:} for $\gam \in\RR$ define $w^{\BDom}_{\gam}$ on $\Dom$ by
\begin{equation*}
w^{\BDom}_{\gam}(x) := \mrm{dist}(x,\BDom)^{\gam}, \qquad x \in \Dom.
\end{equation*}
Furthermore, we write $w_{\gam}(x) := w_\gam^{\smash{\d\RRdh}}(x)=|x_1|^\gam$ for $x=(x_1, \tilde{x})\in \RRdh$. We note that the power weight $w_\gam(x)=|x_1|^\gam$ is in $A_p(\OO)$ if and only if $\gam\in (-1,p-1)$. \\
\end{enumerate}

For $p \in [1,\infty)$, $w$ a weight and $X$ a Banach space, we define the weighted Lebesgue space $L^p(\Dom,w;X)$ as the Bochner space consisting of all equivalence classes of strongly measurable $f\colon \mc{O}\to X$ such that
\begin{equation*}
\nrm{f}_{L^p(\Dom,w;X)} := \has{\int_{\Dom}\|f(x)\|_X^p\:w(x)\dd x }^{1/p}<\infty.
\end{equation*}
Let $w$ be such that $w^{-\frac{1}{p-1}}\in L^1_{\loc}(\OO)$. The $k$-th order weighted Sobolev space for $k \in \N_0$ is defined as
\begin{equation*}
W^{k,p}(\Dom,w;X) := \left\{ f \in \mc{D}'(\Dom;X) : \forall |\alpha| \leq k, \partial^{\alpha}f \in L^p(\Dom,w;X) \right\}
\end{equation*}
equipped with the canonical norm. If $w\equiv 1$, then we simply write $W^{k,p}(\OO;X)$. Furthermore, if $X=\CC$, then we omit this from the notation as well.
\begin{remark}\label{rem:L1loc}
  The local $L^1$ condition for $w^{-\frac{1}{p-1}}$ ensures that all the derivatives $\d^{\alpha}f$ are locally integrable in $\OO$. If $\OO$ is the half-space $\RRdh$ or a bounded domain, then this condition for power weights $w_{\gam}^{\d\OO}$ holds for all $\gam\in\RR$. For $\OO=\RR^d$ the local $L^1$ condition holds only for weights $w_\gam(x)= |x_1|^\gam$ with $\gam\in(-\infty,p-1)$ since functions might not be locally integrable near $x_1=0$, see also \cite[Example~1.7]{KO1984}. This explains why, for example, we cannot employ classical reflection arguments from $\RRdh$ to $\RR^d$ if $\gam>p-1$. 
\end{remark}

\subsection{Weighted Sobolev spaces with boundary conditions}\label{subsec:Sob_BC}
We continue with defining weighted Sobolev spaces with vanishing boundary conditions.
Let $p\in(1,\infty)$, $k\in\NN_0$, $\gam>-1$ and let $X$ be a Banach space. We define
\begin{equation*}
  \cir{W}_0^{k,p}(\Dom,w^{\BDom}_{\gam};X) := \overline{\Cc^\infty(\Dom;X)}^{W^{k,p}(\Dom,w^{\BDom}_{\gam};X)}.
\end{equation*}
Alternative characterisations of this space in terms of traces are available, see \cite{LLRV24, LLRV25, LV18}. These characterisations are based on the mapping properties of the diffeomorphism in Lemma \ref{lem:localisation_weighted_blow-up}. We distinguish three cases: the half-space, special domains and bounded domains.

\subsubsection{Trace characterisations on the half-space}
For $p\in(1,\infty)$, $k\in\NN_0$, $\gam\in (-1, \infty)\setminus\{jp-1:j\in\NN_1\}$ and $X$ a Banach space, we define
\begin{equation}\label{eq:def:W_0}
  W^{k,p}_{0}(\RR_+^d, w_{\gam};X):=\left\{f\in W^{k,p}(\RR_+^d,w_{\gam};X): \operatorname{Tr}( \d^{\alpha}f)=0 \text{ if }k-|\alpha|>\tfrac{\gam+1}{p}\right\}.
\end{equation}
All the traces in the above definition \eqref{eq:def:W_0} are well defined. Indeed, this follows from the embedding $$W^{1,p}(\RR_+, w;X)\hookrightarrow C([0,\infty);X),$$
for $p\in(1,\infty)$, $w$ such that $w^{-\frac{1}{p-1}}\in L^1_{\loc}(\RR_+)$ and $X$ a Banach space, see \cite[Lemma 3.1]{LV18}. For details, we refer to \cite[Section 3.1]{LLRV24}. 

Let $p\in(1,\infty)$ and $X$ a Banach space. To be able to deal with negative weights, we define for $k\in\NN_0$ and $\gam\leq -1$ the space
\begin{equation*}
  W^{k,p}_0(\RRdh, w_{\gam};X):= W^{k,p}(\RRdh, w_{\gam};X),
\end{equation*} 
which is an appropriate definition in view of \cite[Lemma 3.1(2)]{LV18}. For $k\in \NN_1$ and $\gam\in (-1,\infty)\setminus\{jp-1:j\in\NN_1\}$, weighted Sobolev spaces with negative smoothness are defined as
\begin{equation*}
  W^{-k,p}(\RRdh, w_{\gam};X): = \big(W_0^{k,p'}(\RRdh, w_{\gam'}; X')\big)',
\end{equation*}
where $\gam'=-\gam/(p-1)$ and the duality is with respect to the unweighted pairing.\\

From  \cite[Proposition 3.8]{LV18} (see also \cite[Proposition 3.3]{LLRV25}) the following trace characterisation follows.
\begin{proposition}[Trace characterisation on $\RRdh$]\label{prop:tracechar_RRdh}
    Let $p\in(1,\infty)$, $k\in\NN_0$, $\gam\in \RR\setminus\{jp-1:j\in\NN_1\}$ and let $X$ be a Banach space. Then
    \begin{equation*}
  \cir{W}_0^{k,p}(\RRdh,w_{\gam};X) =  W^{k,p}_{0}(\RR_+^d, w_{\gam};X).
\end{equation*}
\end{proposition}
Note that Proposition \ref{prop:tracechar_RRdh} means that $\Cc^{\infty}(\RRdh;X)$ is dense in $W^{k,p}_{0}(\RR_+^d, w_{\gam};X)$ for the given ranges of parameters. In particular, we also have
\begin{equation}\label{eq:densityCcW}
  \Cc^{\infty}(\RRdh;X)\stackrel{\text{dense}}{\subseteq} W^{k,p}(\RRdh, w_{\gam};X)\quad \text{ for }\, \gam>kp-1.
\end{equation}

\subsubsection{Trace characterisations on special domains}
The diffeomorphism $\Phi$ in Lemma \ref{lem:localisation_weighted_blow-up} will allow us to define weighted Sobolev spaces on special domains with vanishing traces. We first note that $\Phi$ gives rise to an isomorphism between weighted Sobolev spaces on special domains and the half-space.
\begin{proposition}\label{prop:isom}
  Let $p\in(1,\infty)$, $\ell\in \NN_1$, $\kappa\in [0,1]$, $k\in\NN_0$, $\gam\in (-1,\infty)\setminus\{jp-1:j\in\NN_1\}$ and let $X$ be a Banach space. Moreover, let $\OO$ be a special $\Cc^{\ell,\kappa}$-domain with $[\OO]_{C^{\ell,\kappa}}\leq 1$. Let $\Phi:\OO\to \RRdh$ be as in Lemma \ref{lem:localisation_weighted_blow-up} and consider the change of coordinates mappings
  \begin{align*}
  \Phi_*&:W^{k,p}(\OO, w_{\gam}^{\d\OO};X)\to W^{k,p}(\RRdh, w_{\gam};X)\qquad \text{ for } \gam>(k-(\ell+\kappa))_+p-1,\\
    \Phi_*&:\cir{W}_0^{k,p}(\OO, w_{\gam}^{\d\OO};X)\to \cir{W}_0^{k,p}(\RRdh, w_{\gam};X),
  \end{align*}
  defined by $\Phi_* f := f\circ \Phi^{-1}$. Then $\Phi_*$ is an isomorphism of Banach spaces for which $(\Phi^{-1})_*$ acts as inverse.
\end{proposition}
\begin{proof}
  The statement of the proposition follows from \cite[Proposition 3.7 \& Remark 3.8]{LLRV25}.
\end{proof}

With the aid of the isomorphism in Proposition \ref{prop:isom}, we can now define traces of functions in weighted Sobolev spaces on special domains with low smoothness compared to the smoothness of the Sobolev space. We note that the standard localisation procedure in, e.g, \cite{DHP03, Ev10, KrBook08} is not sufficient in our case since the regularity of the domain is too low, see \cite[Remark 3.10]{LLRV25}.
\begin{definition}\label{def:spaces_special}
   Let $p\in(1,\infty)$, $\ell\in \NN_1$, $\kappa\in[0,1]$, $k\in\NN_0$ and $X$ a Banach space. Let $\gam\in (-1,\infty)\setminus\{jp-1:j\in\NN_1\}$ be such that $\gam>(k-(\ell+\kappa))_+p-1$ and let $\OO$ be a special $\Cc^{\ell,\kappa}$-domain with $[\OO]_{C^{\ell,\kappa}}\leq 1$. Let $\Phi_*$ be the isomorphism from Proposition \ref{prop:isom}. We define
   \begin{equation*}
     W_{0}^{k,p}(\Dom,w^{\BDom}_{\gam};X):= \cbraces{f \in W^{k,p}(\Dom,w^{\BDom}_{\gam};X): \Tr( \d^\alpha(\Phi_* f))=0 \text{ if } k -\abs{\alpha} >\tfrac{\gam+1}{p}}.
   \end{equation*}
\end{definition}
Note that the traces in the above definition are considered on the half-space and are thus well defined by Proposition \ref{prop:isom}. Moreover, provided that $\gam$ is large enough, Lemma \ref{lem:localisation_weighted_blow-up} ensures that the definition of the space is consistent in the sense that viewing $\OO$ as either a special $\Cc^{\ell,\kappa}$-domain or a special $\Cc^1$-domain yields the same space. \\

We now obtain a trace characterisation for the spaces $\cir{W}^{k,p}_0(\OO, w_{\gam}^{\d\OO};X)$ if $\OO$ is a special domain.
\begin{proposition}[Trace characterisation on special domains]\label{prop:trace_char_dom}
   Let $p\in(1,\infty)$, $\ell\in \NN_1$, $\kappa\in[0,1]$, $k\in\NN_0$ and let $X$ be a Banach space. Let $\gam\in (-1,\infty)\setminus\{jp-1:j\in\NN_1\}$ be such that $\gam>(k-(\ell+\kappa))_+p-1$ and let $\OO$ be a special $\Cc^{\ell,\kappa}$-domain with $[\OO]_{C^{\ell,\kappa}}\leq 1$. Let $\Phi_*$ be the isomorphism from Proposition \ref{prop:isom}. Then
   \begin{equation*}
     \cir{W}^{k,p}_0(\OO, w_{\gam}^{\d\OO};X)= W_{0}^{k,p}(\Dom,w^{\BDom}_{\gam};X).
   \end{equation*}
\end{proposition}
\begin{proof}
  The proof is the same as the proof of \cite[Proposition 3.10]{LLRV25} using Proposition \ref{prop:isom}.
\end{proof}

\subsubsection{Trace characterisations on bounded domains} To define traces on bounded domains, we employ a localisation procedure based on the following decomposition of weighted Sobolev spaces, see also \cite[Section 2.2]{LV18} and \cite[Lemma 3.12]{LLRV25}. We note that the results in the lemma easily follow from a covering of the domain, Remark \ref{rem:special_dom} and a partition of unity (see, e.g., \cite[Section 8.4]{KrBook08}).
\begin{lemma}\label{lem:decomp}
  Let $\ell\in\NN_1$, $\kappa\in[0,1]$, and let $\OO\subseteq\RR^d$ be a bounded $C^{\ell,\kappa}$-domain. Then the following statements hold.
\begin{enumerate}[(i)]
  \item For all $\eps\in (0,\kappa)$ there exists a finite open cover $(V_n)_{n=1}^N$ of $\d \OO$, together with special $\Cc^{\ell,\kappa}$-domains $(\OO_n)_{n=1}^N$ which satisfy $[\OO_n]_{C^{\ell,\kappa-\eps}}<\delta$, such that
      \begin{equation*}
        \mc{O}\cap V_n = \OO_n \cap V_n \qquad \text{and}\qquad \BDom \cap V_n = \partial \OO_n \cap V_n,\quad n\in\{1,\dots, N\}.
        \end{equation*}
        If $\kappa=0$, then for any $\delta>0$, the special $\Cc^\ell$-domains $(\OO_n)_{n=1}^N$ can be chosen such that  $[\OO_n]_{C^{\ell}}<\delta$. 
  \item There exist $\eta_0\in \Cc^{\infty}(\OO)$ and $\eta_n \in \Cc^{\infty}(V_n)$ for all $n\in\{1,\dots,N\}$ such that $\sum_{n=0}^{N}\eta_n^2=1$ on $\OO$ (partition of unity).
      \item For $p\in(1,\infty)$, $k\in\NN_0$, $\gam\in\RR$ and $X$ a Banach space, the space $W^{k,p}(\OO, w_{\gam}^{\d\OO};X)$ has the direct sum decomposition
          \begin{equation}\label{eq:Fk}
            \mc{W}^{k,p}_{\gam}:=W^{k,p}(\RRd;X)\oplus \bigoplus_{n=1}^NW^{k,p}(\OO_n,w_{\gam}^{\d\OO_n};X).
          \end{equation}
          Moreover, the mappings
           \begin{align*}
             \II \colon W^{k,p}(\OO, w_{\gam}^{\d\OO};X)\to \mc{W}^{k,p}_{\gam}\quad \text{ and }\quad \PP\colon \mc{W}^{k,p}_{\gam}\to W^{k,p}(\OO, w_{\gam}^{\d\OO};X)
           \end{align*}
           given by
          \begin{equation}\label{eq:retraction}
            \II f := (\eta_n f)_{n=0}^N\quad \text{ and }\quad \PP(f_n)_{n=0}^N:=\sum_{n=0}^N\eta_nf_n,
          \end{equation}
          are continuous and satisfy $\PP\II = \operatorname{id}$. Thus, $\PP$ is a retraction with coretraction $\II$.
\end{enumerate}
\end{lemma}
In addition to \eqref{eq:Fk}, we define spaces incorporating zero boundary conditions as
\begin{equation}\label{eq:Fk0}
    \mc{W}^{k,p}_{0,\gam}:= W^{k,p}(\RRd;X)\oplus \bigoplus_{n=1}^NW_0^{k,p}(\OO_n,w_{\gam}^{\d\OO_n};X).
\end{equation}

Using the above decomposition, we can define weighted Sobolev spaces on bounded domains with vanishing traces. 
\begin{definition}
   Let $p\in(1,\infty)$, $\ell\in \NN_1$, $\kappa\in[0,1]$, $k\in\NN_0$ and let $X$ be a Banach space. Let $\gam\in (-1,\infty)\setminus\{jp-1:j\in\NN_1\}$ be such that $\gam>(k-(\ell+\kappa))_+p-1$. Moreover, let $\OO$ be a bounded $C^{\ell, \kappa}$-domain, let $(\OO_n)_{n=1}^N$ be special $\Cc^{\ell,\kappa}$-domains and let $\mc{I}$ be the coretraction from Lemma \ref{lem:decomp}. We define
   \begin{equation*}
     W^{k,p}_0(\OO, w_{\gam}^{\d\OO};X):= \big\{f\in W^{k,p}(\OO, w_{\gam}^{\d\OO};X): \II f\in W^{k,p}(\RRd;X)\oplus \bigoplus_{n=1}^N W^{k,p}_0(\OO_n, w_{\gam}^{\d\OO_n};X) \big\}.
   \end{equation*}
\end{definition}
This space is well defined by Lemma \ref{lem:decomp} and Definition \ref{def:spaces_special}. Moreover, the definitions are independent of the chosen covering of $\d\OO$ and the partition of unity in Lemma \ref{lem:decomp}.
Similar to Proposition \ref{prop:trace_char_dom}, we obtain the following trace characterisation for bounded domains. For the proof of this trace characterisation, we refer to \cite[Proposition 3.14]{LLRV25}.
\begin{proposition}[Trace characterisation on bounded domains]
    Let $p\in(1,\infty)$, $\ell\in\NN_1$, $\kappa\in[0,1]$, $k\in\NN_0$ and let $X$ be a Banach space. Let $\gam\in (-1,\infty)\setminus\{jp-1:j\in\NN_1\}$ be such that $\gam>(k-(\ell+\kappa))_+p-1$ and let $\OO$ be a bounded $C^{\ell,\kappa}$-domain. Then
    \begin{equation*}
        \cir{W}_0^{k,p}(\OO, w_{\gam}^{\d\OO};X)= W_0^{k,p}(\OO, w_{\gam}^{\d\OO};X).
    \end{equation*}
\end{proposition}

\subsubsection{Properties of weighted Sobolev spaces} Finally, we gather some properties of weighted Sobolev spaces with vanishing boundary conditions on the half-space. We begin with the following variant of Hardy's inequality, see, e.g., \cite[Corollary 3.4]{LV18}.
\begin{lemma}[Hardy's inequality on $\RRdh$]\label{lem:Hardy}
  Let $p\in(1,\infty)$, $k\in\NN_1$, $\gam\in \RR$ and let $X$ be a Banach space. Then 
        \begin{align*}
     W_0^{k,p}(\RRdh,w_{\gam};X)&\hookrightarrow W^{k-1,p} (\RRdh,w_{\gam-p};X) &&\text{ if }\gam<p-1,\\
     W^{k,p}(\RRdh,w_{\gam};X)&\hookrightarrow W^{k-1,p} (\RRdh,w_{\gam-p};X) &&\text{ if }\gam>p-1,\\
      W_0^{k,p}(\RRdh,w_{\gam};X)&\hookrightarrow W_0^{k-1,p} (\RRdh,w_{\gam-p};X)&&\text{ if }\gam\notin \{jp-1:j\in\NN_1\}.
  \end{align*}
\end{lemma}

We continue with a density result for weighted Sobolev spaces with negative smoothness. Results for spaces with positive smoothness are, for instance, contained in \cite[Section 3.3]{LV18} and \cite[Section 4.2]{Ro25}.

\begin{lemma}\label{lem:densityCcnegative}
  Let $p\in(1,\infty)$, $k\in \NN_1$, $\gam\in (-1,\infty)\setminus\{jp-1:j\in\NN_1\}$ and let $X$ be a reflexive Banach space. Then $\Cc^\infty(\RRdh;X)$ is dense in $W^{-k,p}(\RRdh, w_{\gam};X)$.
\end{lemma}
\begin{proof}
  From \cite[Proposition 3.8]{LV18} we have
  \begin{equation*}
    \Cc^\infty(\RRdh;X')\stackrel{\text{dense}}{\hookrightarrow} W_0^{k,p'}(\RRdh, w_{\gam'};X')\quad \text{ and }\quad \Cc^\infty(\RRdh;X')\stackrel{\text{dense}}{\hookrightarrow}L^{p'}(\RRdh, w_{\gam'};X'),
  \end{equation*}
  where we recall that $\gam'=-\frac{\gam}{p-1}$. Hence, $W_0^{k,p'}(\RRdh, w_{\gam'};X')$ is dense in $L^{p'}(\RRdh, w_{\gam'};X')$ as well. Since $X$ is reflexive, it also holds that $W_0^{k,p'}(\RRdh, w_{\gam'};X')$ is reflexive. Hence, it follows that
  \begin{equation*}
    L^p(\RRdh, w_{\gam};X)\stackrel{\text{dense}}{\hookrightarrow} \big(W^{k,p'}_0(\RRdh,w_{\gam'}; X')\big)'= W^{-k,p}(\RRdh, w_{\gam};X),
  \end{equation*}
  and since $\Cc^\infty(\RRdh;X)$ is dense in $L^p(\RRdh, w_{\gam};X)$, the result follows.
\end{proof}

\section{\texorpdfstring{$R$-}{R-}sectoriality for the elliptic operator on the half-space}\label{sec:Rsect}
In this section, we study the elliptic operator corresponding to the damped plate equation \eqref{eq:plate_eq_intro}. This elliptic operator on $\RRd$, $\RRdh$, and bounded domains is studied in \cite{DS15}, and these results will be the starting point for our analysis on weighted spaces. Therefore, we will first recall several results from \cite{DS15} in a slightly more general setting with Muckenhoupt power weights.

On $\RRd$ the system \eqref{eq:plate_eq_intro} without boundary conditions can be written in the form
\begin{equation}\label{eq:plate_eq_Rd}
\left\{
\begin{aligned}
&\d_t^2 u + \rho B^{\half} \d_t u + Bu =f
&\text{ for }t>0,\\
&u|_{t=0}=(\d_t u)|_{t=0}=0,
\end{aligned}
\right.
\end{equation}
where $B:= \del^2$ on $L^p(\RRd)$ with $D(B):=W^{4,p}(\RRd)$. Equation \eqref{eq:plate_eq_Rd} is related to the quadratic operator pencil $V: W^{4,p}(\RRd)\to L^p(\RRd)$ given by
\begin{equation*}
  V(\lambda):= \lambda^2 + \lambda \rho B^{\half} + B = (\alpha_+\lambda+ B^{\half})(\alpha_-\lambda + B^{\half}),
\end{equation*}
where 
\begin{equation*}
\alpha_{\pm}=\begin{cases}
\frac{\rho}{2}\pm \ii \sqrt{1-\frac{\rho^2}{4}}  & \mbox{for }0< \rho <2,\\
               \frac{\rho}{2}\pm \sqrt{\frac{\rho^2}{4}-1} & \mbox{for } \rho\geq 2.
             \end{cases}
\end{equation*}
The roots can be written as $\alpha_{\pm} = e^{\pm \ii \vartheta}$ for $0<\rho < 2$ and $\alpha_{\pm}>0$ for $\rho\geq 2$, where the angle $\vartheta=\vartheta(\rho)$ is given by
\begin{equation*}
  \vartheta(\rho):=\begin{cases}
                     \arctan\frac{2}{\rho}\sqrt{1-\frac{\rho^2}{4}} & \mbox{for } 0< \rho <2, \\
                     0 & \mbox{for } \rho\geq 2.
                   \end{cases}
\end{equation*}
Note that $\arg \alpha_{\pm}=\pm \vartheta(\rho)$ and $\vartheta(\rho)\uparrow \frac{\pi}{2}$ as $\rho\downarrow 0$.

By the classical theory of quadratic operator pencils and second-order Cauchy problems, the operator $V(\lambda)$ can be inverted to show maximal regularity for \eqref{eq:plate_eq_Rd}. We refer to \cite{CCD08,CS05, FO91, SV10} and, in particular, see \cite[Theorem 3.4]{FO91} and \cite[Theorem 4.1]{CS05}. In \cite{DS15}, the related first-order system is studied which is obtained by setting $v:=(u, \d_t u)^\top$. Then we can write
\begin{equation*}
  \left\{
\begin{aligned}
&\d_t v + A v =\begin{pmatrix}
                 0 \\
                 f
               \end{pmatrix}
&\text{ for }t>0,\\
&v|_{t=0}=0,
\end{aligned}
\right.
\end{equation*}
where, formally, the operator matrix $A$ is given by
\begin{equation}\label{eq:A_matrix}
  A : =  \begin{pmatrix}
         0 & -\id \\
         \del^2 & -\rho\del 
       \end{pmatrix},\quad \rho>0.
\end{equation}
In addition, we occasionally study the closely related operator
\begin{equation}\label{eq:A_matrix_mu}
  A_\eta : =  \begin{pmatrix}
         0 & -\id \\
         \eta+\del^2 & -\rho\del 
       \end{pmatrix},\quad \eta\geq 0,\, \rho>0.
\end{equation}
In the case of unbounded domains, the operator $A_\eta$ with $\eta>0$ behaves slightly better than $A=A_0$, since $\eta+\del^2$ is invertible while $\Delta^2$ is not.\\

Following \cite[Section 2]{DS15}, the operator $A$ on the full space can be defined via its symbol, i.e., $A:=\FF^{-1} A(\xi)\FF$, where the matrix-valued symbol $A(\xi)$ is given by
\begin{equation*}
    A(\xi):= \begin{pmatrix}
        0& -1\\|\xi|^4& \rho|\xi|^2
    \end{pmatrix},\qquad \xi\in\RR^d.
\end{equation*}
For $p\in(1,\infty)$, $\gam\in(-1,p-1)$, we define the Banach spaces
\begin{align*}
  \XX& := W^{2,p}(\RRd, w_\gam)\times L^p(\RRd, w_\gam),\\
  \DD&:= W^{4,p}(\RRd, w_\gam)\times  W^{2,p}(\RRd, w_\gam).
\end{align*}

We record the following version of \cite[Proposition 2.4]{DS15}. The proof of \cite[Proposition 2.4]{DS15} only relies on arguments involving Mihlin's multiplier theorem (for operator-valued symbols). Since Mihlin's multiplier theorem is also available on spaces with Muckenhoupt  weights, the results can be extended. Furthermore, in Proposition \ref{prop:higher_reg_Rd} we will provide an extension of Proposition \ref{prop:DS_Rd} concerning higher-order regularity.
\begin{proposition}[{Weighted version of \cite[Proposition 2.4]{DS15}}]\label{prop:DS_Rd}
    Let $p\in(1,\infty)$, $\gam\in (-1,p-1)$ and $\rho>0$. Let $A$ be the realisation of the operator in \eqref{eq:A_matrix} on $\XX$ with domain $D(A):=\DD$. Then, for all $\lambda>0$, the operator $\lambda + A$ is $R$-sectorial with angle $\om_{R}(\lambda+A) =\vartheta(\rho)$. 
\end{proposition}
\begin{proof}
    The proof is similar as in \cite[Lemma 2.1, Propositions 2.2 \& 2.4]{DS15} using weighted versions of Mihlin's multiplier theorem. For scalar-valued symbols, a weighted version of Mihlin's multiplier theorem can be found in, e.g., \cite[Proposition 3.1]{MV15}. For operator-valued symbols one has to apply \cite[Theorem 5.12(b)]{FHL19} instead of \cite[Corollary 3.3]{GW03} to obtain $R$-boundedness.
\end{proof}

In the rest of this section, we consider the operator $A$ on the half-space. First, in Section \ref{subsec:Rsect_unweight} we establish $R$-sectoriality of $A$ on spaces with Muckenhoupt power weights. Subsequently, in Section \ref{subsec:Rsect_weight}, we lift the $R$-sectoriality to higher-order Sobolev spaces with power weights that do not belong to the class of Muckenhoupt weights.

\subsection{\texorpdfstring{$R$-}{R-}sectoriality on spaces with Muckenhoupt power weights}\label{subsec:Rsect_unweight}
For $p\in(1,\infty)$ and $\gam\in(-1,p-1)$, we define the Banach spaces 
\begin{align*}
  \XX_+ &:= W^{2,p}_0(\RRdh, w_\gam)\times L^p(\RRdh, w_\gam),\\
  \DD_+&:= \big(W^{4,p}(\RRdh,w_\gam)\cap W^{2,p}_0(\RRdh, w_\gam)\big)\times  W^{2,p}_0(\RRdh, w_\gam).
\end{align*}

Similar to \cite[Theorem 4.4]{DS15} we obtain $R$-sectoriality of $A_\eta$ on spaces with Muckenhoupt power weights.
\begin{theorem}[{Weighted version of \cite[Theorem 4.4]{DS15}}]\label{thm:DS}
  Let $p\in(1,\infty)$, $\gam\in (-1,p-1)$, $\eta\geq 0$, $\rho>0$ and $\sigma\in (\vartheta(\rho),\pi)$. Let $A_\eta$ be the realisation of the operator in \eqref{eq:A_matrix_mu} on $\XX_+$ with domain $D(A_\eta):=\DD_+$. Then there exists a $\tilde{\lambda}\geq 0$ such that for all $\lambda > \tilde{\lambda}$, the operator $\lambda + A_\eta$ is $R$-sectorial with angle $\om_{R}(\lambda+A_\eta) \leq \sigma$. 
\end{theorem}
\begin{proof}
It suffices to consider $\eta=0$, since a lower-order perturbation result \cite[Corollary 16.2.5]{HNVW24} yields the case $\eta>0$. For $\eta=0$ the result follows from adapting the proof of \cite[Theorem 4.4]{DS15} to include weights. We briefly sketch how \cite[Theorem 4.4]{DS15} can be extended to the weighted setting. First, note that the representation formulas for the solution in \cite[Lemmas 3.1 \& 3.2]{DS15} do not depend on the underlying space. Then a weighted version of \cite[Proposition 3.3]{DS15} can be proved using a weighted version of \cite[Proposition 4.12]{DHP03} and verifying that the integral operator related to the kernel $(x_1+y_1)^{-1}$ is bounded in $L^p(\RR_+, w_{\gam})$ for $\gam\in (-1,p-1)$. The proof of \cite[Proposition 4.12]{DHP03} directly carries over to the weighted setting, and the boundedness of the integral operator follows from \cite[Theorem~1]{Andersen80}, where the condition on the weight functions in this theorem is satisfied for power weights exactly for $\gam\in (-1,p-1)$. Moreover, \cite[Lemma 4.3]{DS15} also holds on $L^p(\RRdh, w_{\gam})$. This provides all the ingredients for the proof of \cite[Theorem 4.4]{DS15} including weights.
\end{proof}

The zero boundary conditions in the base space $\XX_+$ are \emph{necessary} for the result of Theorem \ref{thm:DS} to be true, see \cite{DD11} and \cite[Proposition 3.4]{DS15}. The existence of the zeroth and first-order trace of functions in $W^{2,p}(\RRdh, w_\gam)$ is problematic here. We aim to obtain $R$-sectoriality for $A$ on spaces \emph{without} the unnatural boundary conditions. The results in \cite{LLRV24, LLRV25} indicate that this can be achieved on higher-order spaces with suitable power weights. The main advantage of these weighted spaces is that fewer traces exist than in the unweighted case. In particular, no traces exist for functions in the weighted space $W^{2,p}(\RRdh, w_{\gam+2p})$ with $\gam\in (-1,p-1)$. This allows us to obtain $R$-sectoriality for $A$ on $W^{2,p}(\RRdh, w_{\gam+2p})\times L^p(\RRdh, w_{\gam+2p})$ in Section \ref{subsec:Rsect_weight}. \\

To be able to deal with spaces with more general power weights outside the Muckenhoupt class, we 
first study $R$-sectoriality of $A$ on spaces with negative smoothness in Theorem \ref{thm:Rsect_noweight}.
This is done by extrapolating the $R$-sectoriality result for $A$ on $W_0^{2,p}(\RRdh, w_{\gam})\times L^p(\RRdh, w_{\gam})$ to the space $L^p(\RRdh, w_\gam)\times W^{-2,p}(\RRdh, w_\gam)$, similar to \cite[Section 3]{DK18_transmission} where the case $p=2$ is considered for bounded domains.

For proving Theorem \ref{thm:Rsect_noweight} we need some preliminaries. We start off with the study of $\del^2$ in the strong and weak setting, where we have the following result that can be proved similarly as in \cite[Theorem 7.4]{DHP03}.

\begin{theorem}\label{thm:del2_strong}
   Let $p\in(1,\infty)$ and $\gam\in (-1,p-1)$. Then $\del^2$ on $L^p(\RRdh, w_\gam)$ with domain $D(\del^2):=W^{4,p}(\RRdh, w_\gam)\cap W^{2,p}_0(\RRdh, w_\gam)$ has a bounded $\Hinf$-calculus with angle $\om_{\Hinf}(\del^2)=0$.
\end{theorem}

Using the interpolation-extrapolation scale (Section \ref{subsec:int_ext}), we can obtain results for $\del^2$ in the weak setting. For this, we first determine the adjoint operator of $\del^2$ on $L^p(\RRdh, w_\gam)$. We recall that we use the notation $\gam'=-\frac{\gam}{p-1}$.

\begin{lemma}\label{lem:adjointdel2}
  Let $p\in(1,\infty)$ and $\gam\in (-1,p-1)$. Let $B_{p,\gam}:=\del^2$ on $L^p(\RRdh, w_\gam)$ with domain $D(B_{p,\gam}):=W^{4,p}(\RRdh, w_\gam)\cap W^{2,p}_0(\RRdh, w_\gam)$. Then the adjoint operator is $B_{p,\gam}'=B_{p', \gam'}$.
\end{lemma}
\begin{proof}
  By integration by parts we obtain
  \begin{equation*}
    \langle \del^2 f,g\rangle_{L^p(\RRdh, w_\gam)\times L^{p'}(\RRdh, w_{\gam'})} = \langle f, \del^2 g\rangle_{L^p(\RRdh, w_\gam)\times L^{p'}(\RRdh, w_{\gam'})}
  \end{equation*}
  for $f\in D(B_{p, \gam})$ and $g\in D(B_{p', \gam'})$. In particular, by definition of the domain of the adjoint (see for instance \cite[Section III.5.5]{Ka96}) we obtain $B_{p,\gam}\subseteq B_{p', \gam'}'$ and $B_{p', \gam'}\subseteq B_{p,\gam}'$. For the other inclusion $B_{p,\gam}'\subseteq B_{p', \gam'}$ it suffices to show that $B_{p', \gam'}$ is surjective and $B_{p, \gam}'$ is injective, see \cite[Exercise IV.1.21(5)]{EN00}. The surjectivity of $B_{p', \gam'}$ follows from Theorem \ref{thm:del2_strong}. By reflexivity and Theorem \ref{thm:del2_strong} it also follows that $B_{p, \gam}''=B_{p, \gam}$ is surjective, hence $B_{p', \gam'}'$ is injective.
\end{proof}

For $\del^2$ in the weak setting, we now obtain the following regularity result.

\begin{proposition}\label{cor:del2_weak}
  Let $p\in(1,\infty)$, $\gam\in (-1,p-1)$, and $\eta>0$. Then $\eta+\del^2$ on $W^{-2,p}(\RRdh, w_\gam)$ with domain $D(\del^2):= W^{2,p}_0(\RRdh, w_\gam)$ has a bounded $\Hinf$-calculus with angle $\om_{\Hinf}(\eta+\del^2)=0$.
\end{proposition}
\begin{proof}
Let $B_{p,\gam}=\del^2$ on $L^p(\RRdh, w_\gam)$ be as in Lemma \ref{lem:adjointdel2} and let $B_{p, \gam}'= B_{p', \gam'}$ be its adjoint. Note that the bounded $\Hinf$-calculus (Theorem \ref{thm:del2_strong}) implies that $B_{p',\gam'}$ on $L^{p'}(\RRdh, w_{\gam'})$ has bounded imaginary powers. Therefore, by \cite[Corollary 15.3.10]{HNVW24} and Lemma \ref{lem:adjointdel2} we obtain
\begin{equation}\label{eq:compl_int_id}
  \begin{aligned}
  D((B_{p,\gam}')^\half) & = \big[L^{p'}(\RRdh, w_{\gam'}), D(B_{p,\gam}')\big]_\half\\
  &= \big[L^{p'}(\RRdh, w_{\gam'}), W^{4,p'}(\RRdh, w_{\gam'})\cap W^{2,p'}_0(\RRdh, w_{\gam'})\big]_\half
  = W^{2,p'}_0(\RRdh, w_{\gam'}),
\end{aligned}
\end{equation}
where the last identification follows from \cite[Theorem 6.4]{Ro25} (see \cite{Se72} for the unweighted case). Using the interpolation-extrapolation scale (see Proposition \ref{prop:ext_adjoint}), \eqref{eq:compl_int_id} and duality, we obtain
\begin{equation}\label{eq:dual_del2}
  E_{-\half, \eta + B_{p,\gam}} = (E_{\half, (\eta + B_{p,\gam})'})'=\big(D((B_{p,\gam}')^\half)\big)' = W^{-2,p}(\RRdh, w_\gam).
\end{equation}
From Theorem \ref{thm:del2_strong}, Propositions \ref{prop:int_ext_properties} and \ref{prop:Ext_compl_int}, and \eqref{eq:dual_del2}, it follows now that $\eta+\del^2$ on $W^{-2,p}(\RRdh, w_\gam)$ has a bounded $\Hinf$-calculus with angle zero. Moreover, \eqref{eq:compl_int_id} (with $p'$ and $\gam'$ replaced by $p$ and $\gam$, respectively) shows that the domain of $\eta+\del^2$ on $W^{-2,p}(\RRdh, w_\gam)$ is given by $W^{2,p}_0(\RRdh, w_\gam)$.
\end{proof}

With the above results for $\del^2$ at hand, we can now study the more sophisticated operator $A_\eta$ from \eqref{eq:A_matrix_mu} for $\eta>0$. Again, we make use of the interpolation-extrapolation scales to obtain properties of $A_\eta$ on spaces with less smoothness than those in Theorem \ref{thm:DS}. We first determine the adjoint of $A_\eta$.

\begin{lemma}\label{lem:adjointA}
  Let $p\in(1,\infty)$, $\gam\in (-1,p-1)$, $\eta>0$ and $\rho>0$. Let $A_\eta$ be the realisation of the operator in \eqref{eq:A_matrix_mu} on $\XX_+$ with domain $D(A_\eta):=\DD_+$. Then the adjoint operator $A'_\eta$ on $W^{-2,p'}(\RRdh, w_{\gam'})\times L^{p'}(\RRdh, w_{\gam'})$ is given by
  \begin{equation*}\label{eq:defadjointA}
  A'_\eta= \begin{pmatrix}
        0 & \eta+\del^2\\
        -\id & -\rho \del
      \end{pmatrix} \quad \text{ with }\quad D(A_\eta')= L^{p'}(\RRdh, w_{\gam'})\times W^{2,p'}_0(\RRdh, w_{\gam'}).
\end{equation*}
\end{lemma}
\begin{proof}
  For notational convenience we write $W_{(0)}^{k,p}(w_\gam):=W^{k,p}_{(0)}(\RRdh, w_\gam)$ for $p\in(1,\infty)$, $\gam\in (-1,p-1)$ and $k\in\ZZ$. Let $v=(v_1,v_2)\in L^{p'}(w_{\gam'})\times W^{2,p'}_0(w_{\gam'})\subseteq W^{-2,p'}(w_{\gam'})\times L^{p'}(w_{\gam'})$ and $u=(u_1,u_2)\in D(A_\eta)= (W^{4,p}(w_\gam)\cap W^{2,p}_0(w_\gam))\times W^{2,p}_0(w_\gam)$. By integration by parts and the definition of the distributional derivative, we obtain
  \begin{align*}
    \langle A_\eta u, v \rangle   =\;& \langle -u_2,  v_1\rangle_{L^p(w_\gam)\times L^{p'}(w_{\gam'})}  + \langle  \eta u_1,  v_2\rangle_{L^p(w_\gam)\times L^{p'}(w_{\gam'})}
    +\langle  \del^2 u_1, v_2 \rangle_{L^p(w_\gam)\times L^{p'}(w_{\gam'})}\\\,& +  \langle -\rho \del u_2, v_2\rangle_{L^p(w_\gam)\times L^{p'}(w_{\gam'})}\\
    =\;& \langle  u_2, -v_1 \rangle_{L^p(w_\gam)\times L^{p'}(w_{\gam'})} + \langle  u_1,  \eta v_2\rangle_{L^p(w_\gam)\times L^{p'}(w_{\gam'})}+ \langle   u_1, \del^2 v_2\rangle_{W^{2,p}_0(w_\gam)\times W^{-2,p'}(w_{\gam'})}\\ \;& +  \langle  u_2, -\rho \del v_2\rangle_{L^p(w_\gam)\times L^{p'}(w_{\gam'})}\\
    =\;&  \langle u, \tilde{A}_\eta v\rangle,
  \end{align*}
  where 
  \begin{equation*}
    \tilde{A}_\eta := \begin{pmatrix}
                   0 & \eta +\del^2 \\
                   -\id & -\rho \del
                 \end{pmatrix}\quad \text{ with }\quad D(\tilde{A}_\eta): = L^{p'}(w_{\gam'})\times W^{2,p'}_0(w_{\gam'}).
  \end{equation*}
  For $v\in D(\tilde{A}_\eta)$ it holds that the mapping $(u\mapsto \langle A_\eta u, v\rangle): D(A_\eta)\to \CC$ is continuous with respect to the norm on $W^{2,p}_0(w_\gam)\times L^p(w_\gam)$. Hence, we have that $\tilde{A}_\eta\subseteq A'_\eta$.

  For the converse inclusion, we need to show that if $v\in D(A'_\eta)$, then $v\in D(\tilde{A}_\eta)$.  First, note that if $v\in W^{-2,p'}(w_{\gam'})\times L^{p'}(w_{\gam'})$ and $u\in D(A_\eta)$, then we have
  \begin{equation}\label{eq:adj_A1}
  \begin{aligned}
          \langle A_\eta u , v\rangle = &\;\langle -u_2, v_1\rangle_{W^{2,p}_0(w_\gam)\times W^{-2,p'}(w_{\gam'})} + \langle (\eta+\del^2)u_1, v_2 \rangle_{L^p(w_\gam)\times L^{p'}(w_{\gam'})}\\ &\;+ \langle-\rho\del u_2, v_2\rangle_{L^p(w_\gam)\times L^{p'}(w_{\gam'})}. 
  \end{aligned}
  \end{equation} 
  Now, let $v\in D(A'_\eta)$. Note that the mapping $(u\mapsto \langle A_\eta u, v\rangle): D(A_\eta)\to \CC$ can be extended to a linear and continuous mapping from $W^{2,p}_0(w_\gam)\times L^p(w_\gam)$ to $\CC$. Moreover, the mapping
  \begin{equation}\label{eq:Adj_A2}
    \ph: W^{4,p}(w_\gam)\cap W^{2,p}_0(w_\gam) \to \CC,\quad u_1\mapsto \ph(u_1):=\langle (\eta+\del^2) u_1, v_2\rangle_{L^p(w_\gam)\times L^{p'}(w_{\gam'})}
  \end{equation}
  is continuous with respect to the norm on $W^{2,p}(w_\gam)$. Indeed, this follows from
  \begin{align*}
    \big|\langle(\eta+\del^2) u_1, v_2\rangle_{L^p(w_\gam)\times L^{p'}(w_{\gam'})}\big| &= |\langle A_\eta (u_1,0), v\rangle |\\
    &\leq C\|(u_1,0)\|_{W^{2,p}(w_\gam)\times L^p(w_\gam)} =  C\|u_1\|_{W^{2,p}(w_\gam)},
  \end{align*}
  for $u_1\in W^{4,p}(w_\gam)\times W^{2,p}_0(w_\gam)$. Furthermore, from Proposition \ref{cor:del2_weak} and $\eta>0$, we obtain that
  \begin{equation}\label{eq:Adj_A3}
    \eta + \del^2:W^{2,p}_0(w_\gam)\to W^{-2,p}(w_\gam)
  \end{equation}
  is an isomorphism. Combining \eqref{eq:Adj_A2} and \eqref{eq:Adj_A3} gives
  \begin{equation*}
    \big(\tilde{u}_1\mapsto \ph\big((\eta+\del^2)^{-1}\tilde{u}_1\big) = \langle \tilde{u}_1, v_2\rangle_{L^p(w_\gam)\times L^{p'}(w_{\gam'})}\big): L^p(w_\gam)\to \CC
  \end{equation*}
  is continuous as mapping from $(L^p(w_\gam), \|\cdot\|_{W^{-2,p}(w_\gam)})$ to $\CC$. Since $L^p(w_\gam)$ is dense in $W^{-2,p}(w_\gam)$, there exists a unique continuous extension $\tilde{\ph}:W^{-2,p}(w_\gam)\to \CC$ which we identify with
  $\tilde{\ph}\in (W^{-2,p}(w_\gam))'= W^{2,p'}_0(w_{\gam'})$
  and it holds that
  \begin{equation*}
    \langle\tilde{u}_1, \tilde{\ph} \rangle_{ W^{-2,p}(w_\gam)\times W^{2,p'}_0(w_{\gam'})} = \langle \tilde{u}_1, v_2\rangle_{W^{-2,p}(w_\gam)\times W^{2,p'}_0(w_\gam')}
  \end{equation*}
  for $\tilde{u}_1\in L^p(w_\gam)$. We conclude that $v_2=\tilde{\ph}\in W^{2,p'}_0(w_{\gam'})$ and this implies that the last term on the right-hand side of \eqref{eq:adj_A1}, given by
  \begin{equation*}
    \big(u_2 \mapsto \langle -\rho \del u_2, v_2\rangle_{L^p(w_\gam)\times L^{p'}(w_{\gam'})}= \langle -\rho \del u_2, v_2\rangle_{W^{-2,p}(w_\gam)\times W^{2,p'}_0(w_{\gam'})}  \big): W^{2,p}_0(w_\gam) \to \CC,  
\end{equation*}
is continuous on $L^p(w_\gam)$. Since \eqref{eq:adj_A1} needs to be continuous, by setting $u_1=0$ it follows that also the first term on the right-hand side of \eqref{eq:adj_A1}, given by
\begin{equation*}
  \big(u_2 \mapsto \langle u_2, v_1 \rangle_{ W^{2,p}_0(w_\gam)\times W^{-2,p'}(w_{\gam'})}\big): W^{2,p}_0(w_\gam)\to \CC
\end{equation*}
can be extended continuously to $L^p(w_\gam)$. This implies that 
 $v_1\in L^{p'}(w_{\gam'})$ and thus we have $v\in D(\tilde{A}_\eta)$. Therefore, $D(A'_\eta)=D(\tilde{A}_\eta)$ and $A'_\eta=\tilde{A}_\eta$, which completes the proof.
\end{proof}

The above lemma allows us to obtain $R$-sectoriality for $A_\eta$ on $L^p(\RRdh, w_\gam)\times W^{-2,p}(\RRdh, w_\gam)$.
\begin{theorem}\label{thm:Rsect_noweight}
  Let $p\in(1,\infty)$, $\gam\in (-1,p-1)$ $\eta\geq 0$, $\rho>0$ and $\sigma\in (\vartheta(\rho), \pi)$. Let $A_\eta$ be the realisation of the operator in \eqref{eq:A_matrix_mu} on $L^p(\RRdh, w_\gam)\times W^{-2,p}(\RRdh, w_\gam)$ with domain $D(A_\eta):= W^{2,p}_0(\RRdh, w_\gam)\times L^p(\RRdh, w_\gam)$. Then there exists a $\tilde{\lambda}> 0$ such that for all $\lambda > \tilde{\lambda}$, the operator $\lambda + A_\eta$ is $R$-sectorial with angle $\om_{R}(\lambda+A_\eta) \leq \sigma$. 
\end{theorem}
\begin{proof} First consider $\eta>0$. Take $\tilde{\lambda}$ from Theorem \ref{thm:DS} and let $\lambda>\tilde{\lambda}$. From Theorem \ref{thm:DS} and Proposition \ref{prop:int_ext_properties} we obtain $R$-sectoriality on the extrapolation space $E_{-1,\lambda+A_\eta}$, which by Proposition \ref{prop:ext_adjoint} and Lemma \ref{lem:adjointA} (using $\eta>0$) is 
  \begin{align*}
    E_{-1,\lambda+A_\eta} &= (E_{1, (\lambda+A_\eta)'})'= \big(D(A'_\eta)\big)'\\
    &= \big(L^{p'}(\RRdh, w_{\gam'})\times W_0^{2,p'}(\RRdh, w_{\gam'})\big)'=L^p(\RRdh, w_\gam)\times W^{-2,p}(\RRdh, w_\gam).
  \end{align*}
  Again, by Proposition \ref{prop:int_ext_properties} the domain is given by $\XX_+ = W^{2,p}_0(\RRdh, w_\gam)\times L^p(\RRdh, w_\gam)$.

  The case $\eta=0$ can be obtained with the lower-order perturbation result \cite[Corollary 16.2.5]{HNVW24} for $R$-sectoriality upon taking $\tilde{\lambda}$ larger.
\end{proof}

\subsection{\texorpdfstring{$R$-}{R-}sectoriality on spaces with more general power weights}\label{subsec:Rsect_weight}
In this section, we study the $R$-sectoriality of $A$ on higher-order spaces with power weights that do not belong to the class of Muckenhoupt weights. For fixed $p\in(1,\infty)$ and $\gam\in (-1,p-1)$, we define for $k\in\NN_0$ the Banach spaces
\begin{equation}\label{eq:XXDD}
  \begin{aligned}
  \XX^{k,p}_\gam(\RRdh) := W^{k,p}(\RRdh,w_{\gam+kp})\times W^{k-2,p}(\RRdh, w_{\gam+kp}), \\
  \DD^{k,p}_\gam(\RRdh):= W^{k+2,p}_0(\RRdh,w_{\gam+kp})\times W^{k,p}(\RRdh, w_{\gam+kp}).
\end{aligned}
\end{equation}
In particular, note that no traces exist of functions in $\XX^{k,p}_\gam(\RRd)$, see \eqref{eq:def:W_0}. Furthermore, only the zeroth and first order trace exist of functions in $\DD^{k,p}_\gam(\RRdh)$, modelling the correct Dirichlet--Neumann boundary conditions of \eqref{eq:plate_eq_intro}.
\begin{theorem}\label{thm:Rsect_weight}
  Let $p\in(1,\infty)$, $\gam\in (-1,p-1)$ $k\in\NN_0$, $\rho>0$ and $\sigma\in (\vartheta(\rho), \pi)$. Let $A$ be the realisation of the operator in \eqref{eq:A_matrix} on $\XX^{k,p}_\gam(\RRdh)$ with domain $D(A):= \DD^{k,p}_\gam(\RRdh)$. Then there exists a $\tilde{\lambda}>0$ such that for all $\lambda_0>\tilde{\lambda}$, the operator $\lambda_0+A$ is $R$-sectorial with angle $\om_R(\lambda_0+A)\leq \sigma$.
\end{theorem}

Before we prove Theorem \ref{thm:Rsect_weight}, we collect a few preliminary results. We start with some estimates on functions in weighted Sobolev spaces.
\begin{lemma}\label{lem:est_Hardy}
  Let $p\in(1,\infty)$, $\gam\in (-1,p-1)$ and $k\in\NN_0$. Then, it holds that
  \begin{align*}
    \|f\|_{W^{k,p}(\RRdh, w_{\gam+kp})} &\leq C \|f\|_{W^{k+1,p}(\RRdh, w_{\gam+(k+1)p})},\quad f\in W^{k+1,p}(\RRdh, w_{\gam+(k+1)p}),\\
    \| f\|_{W^{k-2,p}(\RRdh, w_{\gam+kp})}&\leq C \|f\|_{W^{k-1,p}(\RRdh, w_{\gam+(k+1)p})},\quad f\in W^{k-1,p}(\RRdh, w_{\gam+(k+1)p}).
  \end{align*}
\end{lemma}
\begin{proof}
  We only consider the second estimate for $k\in\{0,1\}$. All the other estimates follow immediately from Hardy's inequality (Lemma \ref{lem:Hardy}). For $k\in\{0,1\}$ it suffices by duality to prove
  \begin{equation*}
    \|g\|_{W_0^{1-k,p'}(\RRdh, w_{\gam'-kp'-p'})}\leq C \|g\|_{W_0^{2-k,p'}(\RRdh, w_{\gam'-kp'})},\quad g\in W_0^{2-k,p'}(\RRdh, w_{\gam'-kp'}),
  \end{equation*}
  which again follows from Hardy's inequality noting that $\gam'= -\frac{\gam}{p-1}\in (-1, p'-1)$.
\end{proof}

Let $\theta\in \RR$ and define the pointwise multiplication operator
\begin{equation*}
  M^\theta:\Cc^{\infty}(\RR^d_+;X)\to \Cc^{\infty}(\RR^d_+;X) \quad \text{ by }\quad M^\theta u(x)=x_1^\theta \cdot u(x),\qquad x\in\RRdh.
\end{equation*}
By duality the operator extends to $M^\theta:\mc{D}'(\RRdh;X)\to \mc{D}'(\RRdh;X)$, and $M^{-\theta}$ acts as inverse for $M^{\theta}$ on $\mc{D}'(\RRdh;X)$. Moreover, we write $M:=M^1$. We have the next estimates involving the multiplication operator $M$ on weighted Sobolev spaces.

\begin{lemma}\label{lem:est_M}
  Let $p\in(1,\infty)$, $\gam\in (-1, p-1)$ and $k\in\NN_0$. Then, it holds that
  \begin{align*}
    \sum_{|\alpha|\leq 1}\|M \d^{\alpha} f\|_{W^{k,p}(\RRdh, w_{\gam+kp})} \eqsim \|f\|_{W^{k+1,p}(\RRdh, w_{\gam+(k+1)p})},\quad f\in W^{k+1,p}(\RRdh, w_{\gam+(k+1)p}),\\
    \sum_{|\alpha|\leq 1}\|M \d^{\alpha} f\|_{W_0^{k+2,p}(\RRdh, w_{\gam+kp})} \eqsim \|f\|_{W_0^{k+3,p}(\RRdh, w_{\gam+(k+1)p})},\quad f\in W_0^{k+3,p}(\RRdh, w_{\gam+(k+1)p}).
  \end{align*}
  Moreover, for $|\alpha|\leq 1$, we have the estimate
  \begin{equation*}
    \|M \d^{\alpha} f\|_{W^{k-2,p}(\RRdh, w_{\gam+kp})}\leq C \|f\|_{W^{k-1,p}(\RRdh, w_{\gam+(k+1)p})},\quad f\in W^{k-1,p}(\RRdh, w_{\gam+(k+1)p}).
  \end{equation*}
\end{lemma}
\begin{proof}
  Note that by \cite[Lemma 3.6]{LLRV24} it follows that
  \begin{equation*}
  M: W_0^{\ell,p}(\RRdh, w_{\gam+(k+1)p}) \to W_0^{\ell,p}(\RRdh, w_{\gam+kp})
  \end{equation*}
  is an isomorphism for all $\ell\in \NN_0$. Moreover, if $\ell=k$, then $M:W^{k,p}(\RRdh, w_{\gam+(k+1)p}) \to W^{k,p}(\RRdh, w_{\gam+kp})$ is an isomorphism since in this case no traces exist. Hence, we have
  \begin{equation*}
    \|f\|_{W^{k+1,p}(\RRdh, w_{\gam+(k+1)p})}\eqsim \sum_{|\alpha|\leq 1}\|\d^{\alpha} f\|_{W^{k,p}(\RRdh, w_{\gam+(k+1)p})}\eqsim \sum_{|\alpha|\leq 1}\|M \d^{\alpha} f\|_{W^{k,p}(\RRdh, w_{\gam+kp})},
  \end{equation*}
  for $f\in W^{k+1,p}(\RRdh, w_{\gam+(k+1)p})$. The second estimate follows similarly.
  
  Using that $M: W^{\ell,p}(\RRdh, w_{\gam+(k+1)p})\to W^{\ell,p}(\RRdh, w_{\gam+kp})$ is bounded for $\ell\in\NN_0$ (\cite[Lemma 3.6]{LLRV24}), we obtain for $k\geq 2$ and $|\alpha|\leq 1$ that
  \begin{equation*}
    \|M\d^{\alpha} f\|_{W^{k-2,p}(\RRdh, w_{\gam+kp})}\leq C \|\d^{\alpha} f\|_{W^{k-2,p}(\RRdh, w_{\gam+(k+1)p})}\leq C\| f\|_{W^{k-1,p}(\RRdh, w_{\gam+(k+1)p})},
  \end{equation*}
  for $f\in W^{k-1,p}(\RRdh, w_{\gam+(k+1)p})$. It remains to prove the above estimate for $k\in\{0,1\}$. By duality, it suffices to prove 
  \begin{equation*}
    \|\d^{\alpha} M g\|_{W_0^{1-k,p'}(\RRdh, w_{\gam'-kp'-p'})}\leq C\|g\|_{W_0^{2-k, p'}(\RRdh, w_{\gam'-kp'})},\quad g\in W_0^{2-k, p'}(\RRdh, w_{\gam'-kp'}),
  \end{equation*}
  where by integration by parts $(-1)^{|\alpha|}\d^{\alpha} M$ is the adjoint operator to $M\d^{\alpha}$. Let $\alpha=(\alpha_1, \tilde{\alpha})\in \NN_0\times \NN_0^{d-1}$ with $|\alpha|=1$. Then, $\d^{\alpha} M = M\d^{\alpha} + \alpha_1$. Moreover, since
  \begin{equation*}
    M: W_0^{\ell,p'}(\RRdh, w_{\mu}) \to W_0^{\ell,p'}(\RRdh, w_{\mu-p'})
  \end{equation*}
  is an isomorphism for all $\ell\in \NN_0$ and $\mu\in \RR\setminus\{jp'-1: j\in\NN_1\}$, it follows that 
  \begin{align*}
    \|\d^{\alpha} M g\|_{W_0^{1-k,p'}(\RRdh, w_{\gam'-kp'-p'})} &\leq \|M \d^{\alpha} g\|_{W_0^{1-k,p'}(\RRdh, w_{\gam'-kp'-p'})} +\alpha_1\|g\|_{W_0^{1-k,p'}(\RRdh, w_{\gam'-kp'-p'})}\\
    &\leq C \big( \|\d^{\alpha} g\|_{W_0^{1-k,p'}(\RRdh, w_{\gam'-kp'})} +\alpha_1\|g\|_{W_0^{1-k,p'}(\RRdh, w_{\gam'-kp'-p'})}\big)\\
    &\leq C  \| g\|_{W_0^{2-k,p'}(\RRdh, w_{\gam'-kp'})}, \quad g\in W_0^{2-k, p'}(\RRdh, w_{\gam'-kp'}),
  \end{align*}
  where we have additionally applied Hardy's inequality (Lemma \ref{lem:Hardy}).
\end{proof}

Furthermore, in the proof of Theorem \ref{thm:Rsect_weight}, we need an elliptic regularity result in the case that the right-hand side of the equation is smooth.

\begin{lemma}\label{lem:smoothRHS}
  Let $\rho>0$, $ \sigma\in(\vartheta(\rho), \pi)$ and $\lambda_0>\tilde{\lambda}$ with $\tilde{\lambda}$ as in Theorem \ref{thm:Rsect_noweight}. Then, for all $f\in \Cc^{\infty}(\RRdh)\times \Cc^\infty(\RRdh)$ and $\lambda\in \lambda_0+\Sigma_{\pi-\sigma}$, there exists a unique $u =(u_1, u_2)^\top \in \SS(\RRdh)\times\SS(\RRdh) $ such that
  \begin{equation}\label{eq:smoothRHS}
    \lambda u + A u =f,\qquad u_1(0,\cdot)=(\d_1 u_1)(0,\cdot)=0.
  \end{equation}
\end{lemma}
\begin{proof}
We fix $\lambda>\tilde \lambda$. Note that, in particular, we have $f=(f_1,f_2)^\top\in L^2(\RRdh)\times W^{-2,2}(\RRdh)$. Then, by Theorem~\ref{thm:Rsect_noweight}, we know that for every $\lambda>\tilde \lambda$, equation \eqref{eq:smoothRHS} has a unique
solution $u=(u_1,u_2)^\top\in W^{2,2}_0(\RRdh)\times L^2(\RRdh)$. By the definition of $A$, this shows that $u_1$ is a solution of the scalar equation
\begin{equation}
    \label{eq:smoothscalar2}
     \Delta^2u_1  -\lambda\rho\Delta u_1 + \lambda^2u_1 = f_2+ (\lambda-\rho\Delta)f_1, \qquad u_1(0,\cdot)=(\d_1 u_1)(0,\cdot)=0.
\end{equation}
Let $L(D):=  \Delta^2 - \lambda\rho \Delta +\lambda^2$. We 
extend  $g:= f_2+ (\lambda-\rho\Delta)f_1\in \Cc^\infty(\RRdh)$ by zero to $\tilde g\in \Cc^\infty(\RRd)$ and consider the equation
\[ L(D)v = \tilde g\]
in $\RRd$. The operator $L(D)$ can be written as $\mathcal F^{-1}m\mathcal F$ with the symbol 
\[ m(\xi) := |\xi|^4+ \lambda\rho |\xi|^2 + \lambda^2,\qquad 
\xi\in\RRd.\]
Using $m(\xi)\neq0$ for all $\xi\in\RRd$, we obtain
$v = P\tilde g := \mathcal F^{-1} \frac 1m \mathcal F \tilde g$. As $P$ is a  pseudodifferential operator
of order $-4$ (see, e.g., \cite[Definition~3.1]{Ab12}), we
obtain $v = P\tilde{g}\in \SS(\RRd)$ by \cite[Theorem~3.6]{Ab12}.

We use the ansatz  $u_1 = v|_{\RRdh}+w$,
which yields for 
$w$  the boundary value problem
\begin{equation}\label{eq-bvp-smooth}
    L(D)w = 0,\quad (w, \d_1 w)|_{\d\RRdh} = -(v, \d_1 v)|_{\d\RRdh} \in\SS(\RR^{d-1})\times \SS(\RR^{d-1}).
\end{equation} 
Such boundary value problems were studied in \cite{HL22}. As the 
principal part of $L(D)$ is given by $\Delta^2$, it is easily
seen that the ellipticity conditions (E) and (LS) in \cite[Section~2.5.2]{HL22} are satisfied for any $\phi\in [0,\pi)$. By \cite[Proposition~4.16]{HL22}, there exists a 
solution $w\in\SS(\RRdh)$ of \eqref{eq-bvp-smooth},
which is given by the application of a Poisson operator to the
right-hand side. 

By construction, $\tilde u_1 := v|_{\RRdh}+w \in \SS(\RRdh)$ solves \eqref{eq:smoothscalar2}. Setting $\tilde u_2=\lambda \tilde u_1-f_1\in\SS(\RRdh)$ 
yields a solution $(\tilde u_1,\tilde u_2)^\top $ of 
\eqref{eq:smoothRHS}. By unique solvability
of 
\eqref{eq:smoothRHS}, we obtain
$u_j= \tilde u_j\in \SS(\RRdh)$ for $j\in \{1,2\}$.
\end{proof}

\begin{remark}
    Alternatively, the result of Lemma \ref{lem:smoothRHS} can be derived from properties of the Boutet de Monvel calculus for pseudodifferential boundary value problems. The boundary value problem $\mathcal A:= (L(D), (\Tr, \Tr \op \d_1))^\top$ belongs to the class $\mathscr B_{\operatorname{cl}}^{4,2}(\overline{\RRdh})$ of classical operators in Boutet de Monvel’s calculus of order $4$ and type $2$ in the sense of \cite[Definition~5.1]{Schrohe01}. Therefore, by \cite[Theorem~5.11]{Schrohe01}, there exists a parametrix of the form $\mathcal B = (\tilde P_+ + \tilde G ,\; \tilde K)\in \mathscr B_{\operatorname{cl}}^{-4,0}(\overline{\RRdh})$, i.e., $S:=\mathcal B\mathcal A-\id $ is a regularizing Boutet de Monvel operator. Using \eqref{eq:smoothscalar2} we write
\begin{equation*}
    u_1 = \mathcal B\mathcal A u_1 - Su_1 = (\tilde P_++\tilde G)\big(f_2+ (\lambda-\rho\Delta)f_1\big) - S_1.
\end{equation*} 
As $g:= f_2+ (\lambda-\rho\Delta)f_1\in \Cc^\infty(\RRdh)$, we
get $(\tilde P_++\tilde G)g\in \SS(\RRdh)$  by \cite[Corollary~5.5]{Schrohe01}. Since $\mc{A}$ has constant coefficients it even belongs to the SG-calculus in the sense of \cite[Definition~2.9]{Schrohe99}, so we can even choose $S$ to belong to the class of smoothing operators in the SG-calculus, see  \cite[Definition 2.18]{Schrohe99}. By the mapping properties of such operators \cite[Theorem~4.5]{Schrohe99}, we get  $Su_1\in \SS(\RRdh)$. Hence, $u_1\in \SS(\RRdh)$ and $u_2=\lambda u_1-f_1\in \SS(\RRdh)$.
\end{remark}

We now turn to the proof of $R$-sectoriality of the operator $A$ as stated in Theorem \ref{thm:Rsect_weight}.
\begin{proof}[Proof of Theorem \ref{thm:Rsect_weight}]
    For notational convenience we write $\XX^{k,p}_\gam:=\XX^{k,p}_\gam(\RRdh)$ and $\DD^{k,p}_\gam:=\DD^{k,p}_\gam(\RRdh)$ for $p\in(1,\infty)$, $\gam\in (-1,p-1)$, and $k\in\ZZ$. Moreover, for vectors we write $u=(u_1, u_2)$ and for sequences of vectors we use the notation $(u^{(n)})_{n=1}^N$ for some $N\in \NN_1$.
    
    \textit{Step 1: elliptic regularity.} Let $\sigma>\vartheta(\rho)$ and let $\tilde{\lambda}$ be as in Theorem \ref{thm:Rsect_noweight}. Let $\lambda_0>\tilde{\lambda}$ and $N\in\NN_1$. We first prove the following elliptic regularity result: for all sequences $(\lambda_n)_{n=1}^N$ in $\lambda_0+\Sigma_{\pi-\sigma}$ and $(f^{(n)})_{n=1}^N\in \eps_N^2(\XX^{k,p}_\gam)$, there exists a unique $(u^{(n)})_{n=1}^N\in \eps_N^2(\DD^{k,p}_{\gam})$ such that 
    \begin{equation}\label{eq:R_sect_equation}
      \lambda_n u^{(n)} + A u^{(n)} =f^{(n)},\quad \Tr u_1^{(n)} = \Tr \d_1 u_1^{(n)}=0,
    \end{equation}
    holds for each $n\in\{1,\dots, N\}$. Moreover, this solution satisfies the estimates 
    \begin{subequations}\label{eq:est_sol}
    \begin{align}
      \big\|\sum_{n=1}^N \eps_n \lambda_n u^{(n)}\big\|_{L^2(\Om; \XX^{k,p}_{\gam})}& \lesssim \big\|\sum_{n=1}^N \eps_n f^{(n)} \big\|_{L^2(\Om; \XX^{k,p}_{\gam})},\label{eq:IH_Rbdd_X}\\
      \big\|\sum_{n=1}^N \eps_n  u^{(n)}\big\|_{L^2(\Om; \DD^{k,p}_{\gam})}& \lesssim\big\|\sum_{n=1}^N \eps_n f^{(n)} \big\|_{L^2(\Om; \XX^{k,p}_{\gam})}.\label{eq:IH_Rbdd_D}
      \end{align}
    \end{subequations}
    The case $k=0$ follows from Theorem \ref{thm:Rsect_noweight}. For $k\in\NN_1$, the uniqueness of $(u^{(n)})_{n=1}^N$ in the above statement follows from the uniqueness in $\XX^{0,p}_\gam$ and Lemma \ref{lem:est_Hardy}. The proof of the existence and \eqref{eq:est_sol} goes by induction on $k\geq 0$. Now assume that \eqref{eq:R_sect_equation} and \eqref{eq:est_sol} hold for some $k\in\NN_0$. It remains to show existence of a solution to \eqref{eq:R_sect_equation} and the estimates \eqref{eq:est_sol} for $k+1$.

    Fix $N\in \NN_1$ and let $n\in\{1,\dots, N\}$. Moreover, let $(f^{(n)})_{n=1}^N\in\eps_N^2( \Cc^{\infty}(\RRdh)\times \Cc^{\infty}(\RRdh))$ and consider the equations \eqref{eq:R_sect_equation}
    with $\lambda_n\in \lambda_0+ \Sigma_{\pi-\sigma}$. For each $f^{(n)}\in \Cc^{\infty}(\RRdh)\times \Cc^{\infty}(\RRdh)$, the unique solution $u^{(n)}$ to \eqref{eq:R_sect_equation} is in $\SS(\RRdh)\times\SS(\RRdh)$ and $u^{(n)}_1(0,\cdot)= (\d_1 u_1^{(n)})(0,\cdot)=0$. This follows from Lemma \ref{lem:smoothRHS}. 
    
    For $|\alpha|\leq 1$ we define $v^{(n)}:= M \d^{\alpha} u^{(n)}$. We claim that for $n\in\{1,\dots, N\}$ and $|\alpha|\leq 1$, $v^{(n)}$ satisfies the equation
    \begin{equation}\label{eq:vnalpha}
      \lambda_n v^{(n)} + A v^{(n)} = M\d^\alpha f^{(n)} + \begin{pmatrix}
                                                                   0 \\
                                                                    4\d_1 \d^\alpha \del u_1^{(n)}-2\rho\d_1\d^{\alpha}u_2^{(n)}
                                                                 \end{pmatrix}
    \end{equation}
    and the boundary conditions $ \Tr v_{1}^{(n)} = \Tr \d_1 v_{1}^{(n)}=0$. Indeed, note that we have the commutation relations
    \begin{align*}
    [ \del, M\d^{\alpha} ]g &= 2\d_1\d^{\alpha} g,\\
    [\del^2, M\d^{\alpha}]g &= 4\d_1\d^{\alpha}\del g,
    \end{align*}
    for $g\in \SS(\RRdh)$. Therefore, we obtain
    \begin{align*}
      \lambda_n M\d^{\alpha} u^{(n)} + A M\d^{\alpha}u^{(n)} =&\;\begin{pmatrix}
                              M\d^{\alpha} (\lambda_n u_1^{(n)}-u_2^{(n)}) \\
                              M\d^{\alpha} (\del^2 u_1^{(n)}+\lambda_n u_2^{(n)}-\rho \del u_2^{(n)})
                            \end{pmatrix}\\&\;+ \begin{pmatrix}
                                                                   0 \\
                                                                    4\d_1 \d^\alpha \del u_1^{(n)}-2\rho\d_1\d^{\alpha}u_2^{(n)}
                                                                 \end{pmatrix}
    \end{align*}
    and appealing to \eqref{eq:R_sect_equation} yields \eqref{eq:vnalpha}. Moreover, note that $\Tr M \d^\alpha u_1^{(n)}=0$ and 
    \begin{equation*}
      \Tr \d_1 M\d^{\alpha} u^{(n)}_1 =\Tr\big( M\d_1\d^{\alpha}u^{(n)}_1 + \d^{\alpha} u^{(n)}_1\big) = \Tr \d^{\alpha} u^{(n)}_1= 0,
    \end{equation*}
    where the last identity follows from $\Tr u_1^{(n)}= \Tr \d_1 u^{(n)}_1 = 0$ (if $\alpha=0$ or $\alpha=(1,0,\dots,0)$) and $\Tr \d^{\tilde{\alpha}} u_1^{(n)}= \d^{\tilde{\alpha}}\Tr u_1^{(n)}=0$ if $\alpha=(0,\tilde{\alpha})\in \NN_0\times \NN_0^{d-1}$ with $|\tilde{\alpha}|=1$. This proves the claim.
    
    From the induction hypothesis \eqref{eq:IH_Rbdd_D} applied to \eqref{eq:vnalpha} and \eqref{eq:R_sect_equation}, we obtain for $|\alpha|\leq 1$ that
    \begin{align*}
      \big\|\sum_{n=1}^N \eps_n  v^{(n)}\big\|_{L^2(\Om; \DD^{k,p}_\gam)} \lesssim &\; \big\|\sum_{n=1}^N \eps_n M\d^{\alpha} f^{(n)} \big\|_{L^2(\Om; \XX^{k,p}_\gam)} + \big\|\sum_{n=1}^N \eps_n u_1^{(n)}\big\|_{L^2(\Om; W^{k+2,p}(\RRdh, w_{\gam+kp}))} \\
       &\; + \big\|\sum_{n=1}^N \eps_n u_2^{(n)}\big\|_{L^2(\Om; W^{k,p}(\RRdh, w_{\gam+kp}))}\\
       \lesssim &\;  \big\|\sum_{n=1}^N \eps_n  M\d^{\alpha}   f^{(n)} \big\|_{L^2(\Om; \XX^{k,p}_{\gam})} +  \big\|\sum_{n=1}^N \eps_n f^{(n)} \big\|_{L^2(\Om; \XX^{k,p}_\gam)}\\
       \lesssim &\; \big\|\sum_{n=1}^N \eps_n f^{(n)} \big\|_{L^2(\Om; \XX^{k+1,p}_{\gam})},
    \end{align*}
    where in the last estimate we have used Lemmas \ref{lem:est_Hardy} and \ref{lem:est_M}. Using Lemma \ref{lem:est_M} once more gives
    \begin{equation}\label{eq:IH_D}
      \big\|\sum_{n=1}^N \eps_n u^{(n)}\big\|_{L^2(\Om; \DD^{k+1,p}_\gam)} \lesssim \sum_{|\alpha|\leq 1}\big\|\sum_{n=1}^N \eps_n v^{(n)}\big\|_{L^2(\Om; \DD^{k,p}_\gam)} \lesssim \big\|\sum_{n=1}^N \eps_n f^{(n)} \big\|_{L^2(\Om; \XX^{k+1,p}_\gam)}.
    \end{equation}
    The equation $\lambda_n u^{(n)} + Au^{(n)} =f^{(n)}$ and \eqref{eq:IH_D} yield
    \begin{align*}
      \big\|\sum_{n=1}^N \eps_n \lambda_n u^{(n)}\big\|_{L^2(\Om; \XX^{k+1,p}_\gam)}& \lesssim \big\|\sum_{n=1}^N \eps_n A u^{(n)}\big\|_{L^2(\Om; \XX^{k+1,p}_\gam)}  + \big\|\sum_{n=1}^N \eps_n f^{(n)} \big\|_{L^2(\Om; \XX^{k+1,p}_\gam)}\\
      &\lesssim  \big\|\sum_{n=1}^N \eps_n u^{(n)}\big\|_{L^2(\Om; \DD^{k+1,p}_\gam)}  + \big\|\sum_{n=1}^N \eps_n f^{(n)} \big\|_{L^2(\Om; \XX^{k+1,p}_\gam)}\\
      &\lesssim \big\|\sum_{n=1}^N \eps_n f^{(n)} \big\|_{L^2(\Om; \XX^{k+1,p}_\gam)},
    \end{align*}
    which proves the desired estimates for $(f^{(n)})_{n=1}^N\in \eps_N^2(\Cc^{\infty}(\RRdh)\times \Cc^{\infty}(\RRdh))$. 
    
    The general case follows with by density argument. First, note that $\Cc^{\infty}(\RRdh)\times \Cc^{\infty}(\RRdh)$ is dense in $\XX^{k+1,p}_\gam$ by Lemma \ref{lem:densityCcnegative} and \eqref{eq:densityCcW} for negative and positive smoothness, respectively. Let $(f^{(n)})_{n=1}^N\in \eps_N^2(\XX^{k+1,p}_\gam)$, then for every $n\in\{1,\dots, N\}$ there exists a sequence $(f^{(n,m)})_{m\geq 1}\subseteq \Cc^\infty(\RRdh)\times \Cc^\infty(\RRdh)$ such that $f^{(n,m)}\to f^{(n)}$ in $\XX^{k+1,p}_\gam$ as $m\to\infty$. Moreover, $(f^{(n,m)})_{n=1}^N\to (f^{(n)})_{n=1}^N$ in $\eps_N^2(\XX^{k+1,p}_\gam)$ as $m\to\infty$. Every $(f^{(n,m)})_{n=1}^N$ defines a $(u^{(n,m)})_{n=1}^N\in \eps_N^2(\DD^{k+1,p}_\gam)$ such that $\lambda_nu^{(n,m)}+Au^{(n,m)}=f^{(n,m)}$ for all $n\in\{1,\dots, N\}$ and $m\geq 1$. In addition, by \eqref{eq:IH_D} we have 
    \begin{equation*}
      \big\|\sum_{n=1}^N \eps_n(u^{(n,m_1)}-u^{(n,m_2)})\big\|_{L^2(\Om; \DD^{k+1,p}_\gam)}\lesssim \big\|\sum_{n=1}^N \eps_n(f^{(n,m_1)}-f^{(n,m_2)})\big\|_{L^2(\Om; \XX^{k+1,p}_\gam)}\to 0
    \end{equation*} 
    as $m_1,m_2\to\infty$. Completeness of $\eps_N^{2}(\DD^{k+1,p}_\gam)$ implies that there exists a $(u^{(n)})_{n=1}^N\in \eps_N^2(\DD^{k+1,p}_\gam)$ such that $(u^{(n,m)})_{n=1}^N\to (u^{(n)})_{n=1}^N $ in $\eps_N^2(\DD^{k+1,p}_\gam)$ as $m\to\infty$. Moreover, we have
    \begin{equation*}
      \big\|\sum_{n=1}^N \eps_n (\lambda_n + A)(u^{(n,m)}-u^{(n)})\big\|_{L^2(\Om;\XX^{k+1,p}_\gam)}\lesssim \big\|\sum_{n=1}^N \eps_n (u^{(n,m)}-u^{(n)})\big\|_{L^2(\Om;\DD^{k+1,p}_\gam)}\to 0,
    \end{equation*}
    as $m\to\infty$. This implies that 
    \begin{equation*}
      \big((\lambda_n+A)u^{(n)}\big)_{n=1}^N =\lim_{m\to\infty}\big((\lambda_n+A)u^{(n,m)}\big)_{n=1}^N=\lim_{m\to\infty}(f^{(n,m)})_{n=1}^{N}=(f^{(n)})_{n=1}^N.
    \end{equation*}
     Thus, for any $(f^{(n)})_{n=1}^N\in \eps_N^2(\XX^{k+1,p}_\gam)$ there exists a $(u^{(n)})_{n=1}^N\in \eps_N^2(\DD^{k+1,p}_\gam)$ such that $\lambda_n u^{(n)} + Au^{(n)} =f^{(n)}$ for every $n\in\{1,\dots, N\}$ and the required estimates hold. This finishes the induction.
    
    \textit{Step 2: $R$-sectoriality. }To show that $\lambda_0+A$ is $R$-sectorial, we prove that for $\sigma>\vartheta(\rho)$ the set
    \begin{equation*}
      \{(\lambda-\lambda_0)(\lambda+A)^{-1}: \lambda\in \lambda_0+ \Sigma_{\pi-\sigma}\}
    \end{equation*}
    is $R$-bounded. Note that Step 1 implies that $\lambda_0+ A:\DD^{k,p}_{\gam}\to \XX^{k,p}_{\gam}$ is a closed and surjective operator. Hence, it has dense range, and \cite[Propositions 10.1.7(3) and 10.1.9]{HNVW17} imply dense domain and  injectivity.    
    
    Fix $N\in\NN_1$ and let $n\in \{1,\dots, N\}$. By the Kahane contraction principle \cite[Theorem 6.1.13]{HNVW17} and Step 1, we obtain for $\lambda_n\in \lambda_0+\Sigma_{\pi-\sigma}$
    \begin{align*}
      \big\|\sum_{n=1}^N \eps_n (\lambda_n-\lambda_0)&(\lambda_n + A)^{-1}f^{(n)}\big\|_{L^2(\Om; \XX^{k,p}_{\gam})}\\
      & \lesssim  \max_{1\leq n\leq N}\big|\tfrac{\lambda_n-\lambda_0}{\lambda_n}\big|\big\|\sum_{n=1}^N \eps_n \lambda_n(\lambda_n + A)^{-1}f^{(n)}\big\|_{L^2(\Om; \XX^{k,p}_\gam)}\\
      &\lesssim \big\|\sum_{n=1}^N \eps_n f^{(n)}\big\|_{L^2(\Om; \XX^{k,p}_\gam)},
    \end{align*}
    where we have used that
    \begin{equation*}
      |\lambda_n-\lambda_0|\leq \begin{cases}
                                  |\lambda_n| & \mbox{if } \sigma\in [\frac{\pi}{2},\pi) \\
                                  \frac{|\lambda_n|}{\sin(\sigma)} & \mbox{if } \sigma \in (\vartheta(\rho), \frac{\pi}{2}]
                                \end{cases},\qquad \text{ for }\lambda_n\in \lambda_0+\Sigma_{\pi-\sigma}.
    \end{equation*}
    This shows that $\lambda_0+A$ is $R$-sectorial with angle $\om_{R}(\lambda_0+A)\leq \sigma$. 
\end{proof}

\section{\texorpdfstring{$R$-}{R-}sectoriality for the elliptic operator on domains}\label{sec:Rsect_domains}

In this section, we apply perturbation and localisation techniques to transfer $R$-sectoriality of $A$ on the half-space to bounded domains. Localisation techniques are quite standard in the literature, see, e.g., \cite{DHP03, Ev10, KrBook08}. In the weighted setting, one can use a more advanced localisation procedure based on the diffeomorphism in Lemma \ref{lem:localisation_weighted_blow-up}. In this way, we can deal with domains that have low regularity and cannot be treated by applying standard localisation techniques.\\

We will use a similar strategy as in \cite{LLRV25}, where maximal regularity on weighted Sobolev spaces for the heat equation on bounded $C^1$-domains is derived. For the convenience of the reader, we will include the details of the perturbation and localisation argument, since only the Laplace operator is considered in \cite{LLRV25}. We proceed in three steps.
\begin{enumerate}
  \item Use $R$-sectoriality for $\lambda_0+A$ on $\RRdh$ (Theorem \ref{thm:Rsect_weight}) and known perturbation theory for $R$-sectoriality to obtain $R$-sectoriality for $\lambda_0 + A$ on special domains with $\lambda_0$ large enough.
  \item Perform a localisation procedure to transfer $R$-sectoriality for $\lambda_0+ A$ on special domains to bounded domains.
  \item Use discreteness of the spectrum of $A$ on bounded domains to obtain $R$-sectoriality for the unshifted operator $A$.
\end{enumerate}
These three steps are carried out in the subsequent Sections \ref{subsec:pert_spec}, \ref{subsec:pert_bdd} and \ref{subsec:pert_bdd_noshift}.

\subsection{Perturbation of $R$-sectoriality to special domains}\label{subsec:pert_spec}
Let $p\in (1,\infty)$, $k\in \NN_0$, and $\gam\in (-1,p-1)$. We define the following spaces on domains
\begin{equation}\label{eq:XXDD_dom}
  \begin{aligned}
  \XX^{k,p}_\gam(\OO) := W^{k,p}(\OO,w^{\d\OO}_{\gam+kp})\times W^{k-2,p}(\OO, w^{\d\OO}_{\gam+kp}), \\
  \DD^{k,p}_\gam(\OO):= W^{k+2,p}_0(\OO,w^{\d\OO}_{\gam+kp})\times W^{k,p}(\OO, w^{\d\OO}_{\gam+kp}),
\end{aligned}
\end{equation}
where the space with zero boundary conditions is as in Section \ref{subsec:Sob_BC} if $\OO$ is a special or bounded domain. Note that this definition for $\OO=\RRdh$ coincides with \eqref{eq:XXDD}.\\

We start with the perturbation of the $R$-sectoriality of $A$ to special domains.
\begin{theorem}\label{thm:Rsect_special_dom}
Let $p\in(1,\infty)$, $\kappa\in (0,1)$,  $\gam\in ((1-\kappa)p-1,p-1)$, $k\geq 2$ an integer, $\rho>0$, $\sigma\in (\vartheta(\rho),\pi)$, and let $\OO$ be a special $\Cc^{1,\kappa}$-domain  with $[\OO]_{C^{1,\kappa}}\leq 1$. Let $A$ be the realisation of the operator in \eqref{eq:A_matrix} on $\XX^{k,p}_\gam(\OO)$ with domain $D(A):= \DD^{k,p}_\gam(\OO)$. Then there exist  $\tilde{\lambda}>0$ and $\delta\in(0,1)$ such that, if $[\OO]_{C^{1,\kappa}}<\delta$, then for all $\lambda_0> \tilde{\lambda}$, the operator $\lambda_0+A$ is $R$-sectorial with angle $\om_R(\lambda_0+A)\leq \sigma$.
\end{theorem}

\begin{remark}
    The conditions that $\OO$ is a special $\Cc^{1,\kappa}$-domain and $\gam>(1-\kappa)p-1$ are only required to obtain the domain characterisation $D(A)=\DD^{k,p}_{\gam}(\OO)$, see Proposition \ref{prop:trace_char_dom}. In fact, the rest of the proof as given in this section also works for special $\Cc^1$-domains and $\gam\in (-1,p-1)$. Therefore, Theorem \ref{thm:Rsect_special_dom} remains valid for special $\Cc^1$-domains and $A$ with domain 
    \begin{equation*}
        D(A) := \cir{W}^{k+2,p}_0(\OO,w^{\d\OO}_{\gam+kp})\times W^{k,p}(\OO, w^{\d\OO}_{\gam+kp}).
    \end{equation*}
    Indeed, using these spaces, $\Phi_*$ is an isomorphism without any additional conditions on the smoothness of the domain or weight exponent, see Proposition \ref{prop:isom}.
\end{remark}

To relate the operator $A$ in \eqref{eq:A_matrix} on special domains and the half-space, let $h_1, h_2$ and $\Phi$ be as in Lemma \ref{lem:localisation_weighted_blow-up} defining a diffeomorphism between a special $\Cc^{1,\kappa}$-domain and the half-space. Recall that $\Phi_* f:=f\circ \Phi^{-1}$ for $f\in L^1_{\loc}(\OO)$ and define $A^{\Phi}: W^{4,1}_{\loc}(\RRdh)\times W^{2,1}_{\loc}(\RRdh) \to W^{2,1}_{\loc}(\RRdh) \times L^1_{\loc}(\RRdh)$ by 
\begin{equation}\label{eq:coordinatechangeA}
A^{\Phi} := \Phi_* \circ A \circ (\Phi^{-1})_{*} = A + B,
\end{equation}
where $A$ is the operator on $\RRdh$ and $B$ is a perturbation term. Since $A$ consists of the Laplacian and the bi-Laplacian, we introduce the coordinate changes  $\del^\Phi:W^{2,1}_{\loc}(\RRdh)\to L^1_\loc(\RRdh)$ and $(\del^2)^\Phi: W^{4,1}_{\loc}(\RRdh)\to L^1_\loc(\RRdh)$, which are given by
\begin{equation}\label{eq:coordinatechangeLaplace}
  \del^\Phi: = \Phi_* \circ \del \circ (\Phi^{-1})_{*}\quad \text{ and }\quad (\del^2)^\Phi:= \Phi_* \circ \del^2 \circ (\Phi^{-1})_{*}.
\end{equation}

More generally, if $C$ is a linear differential operator of order $N\in\N_1$, we can define the coordinate change $C^\Phi:= \Phi_*\circ C\circ (\Phi^{-1})_*$. We will compute the perturbation terms corresponding to this coordinate change if $C$ has constant coefficients. To this end, we introduce the following class of differential operators.

\begin{definition}\label{def:specform}
    For $N\in\N_1$, we call a differential operator $P$ in $\R^d_+$ an \emph{operator of special form of order $N$} if it is a linear combination of terms of the form 
$p_1(y)\cdot\ldots\cdot p_\ell (y)\partial^\mu$ with $\mu\in \N_0^d$, 
where each $p_i(y)$ is of the form
\[ p_i(y) = \partial^{\nu_i}h_1(\Phi^{-1}(y))\quad \text{ or }\quad p_i(y) = \partial^{\nu_i}h_2(y), \qquad i \in \{1,\dots,\ell\},\;y\in \RRdh,\]
with $\nu_i\in \NN_0^d\setminus\{0\}$ for $i\in \{1, \dots, \ell\}$. Here,  $\ell\in \{1, \dots, N\}$ and $|\mu|\in \{1, \dots, N\}$ are such that $|\mu|+\sum_{i=1}^{\ell}|\nu_i| = \ell+N$.
\end{definition}

The following lemma provides the necessary properties of differential operators of special form and the coordinate change $C^\Phi$.

\begin{lemma}\label{lem:pushforward} The following statements hold.
\begin{enumerate}[(i)]
    \item\label{it:lem:pushforward1} If $C$ is a linear differential operator of order $N\in\N_1$ with constant coefficients, then $P:=C^\Phi-C$ is of special form of order $N$, 
    where additionally all coefficients $p_i$ have the form $p_i=(\partial^{\nu_i} h_1) \circ\Phi^{-1}$.
    \item\label{it:lem:pushforward2} Let $\tilde{P}$ be a differential operator of special form of order $N\in \NN_1$. Then $\d^\delta \tilde{P}$ is a differential operator of special form of order $N+|\delta|$ for all $\delta\in \NN_0^d$.
\end{enumerate}
\end{lemma}
\begin{proof}
\textit{Proof of \ref{it:lem:pushforward1}.}
     By Lemma~\ref{lem:localisation_weighted_blow-up}, the Jacobian of $\Phi$ is given by 
    \[ (D\Phi)(x) = \begin{pmatrix}
        1 -\partial_1 h_1(x) & -\partial_2 h_1(x) & \ldots & -\partial_d h_1(x) \\
        0 & 1 & \ldots & 0 \\
        \vdots & \vdots & \ddots & \vdots\\
        0 & 0 & \ldots  & 1
    \end{pmatrix}.\]
    From this and the chain rule, we see that 
    \begin{equation}\label{eq:push-forward-u} \partial_k (u\circ \Phi) = ( \partial_k u)\circ\Phi
    -(\partial_k h_1)\big( ( \partial_1 u)\circ\Phi\big),\qquad k\in \{1,\dots,d\},
    \end{equation}
    for sufficiently smooth $u\colon\RR^d_+\to \C$. We show by induction on $N$ that for every differential operator $C$ of order $N$ with constant coefficients, the difference $(C(u\circ\Phi))(x) - (Cu)(\Phi(x))$ is a linear combination of terms of the form
    \begin{equation}\label{eq:special_form}
        \Big(\prod_{i=1}^\ell (\partial^{\nu_i} h_1)(x)\Big) (\partial^\mu u)(\Phi(x)),\qquad x\in \OO,
    \end{equation} 
    where $\ell, \nu_i,\mu$ are as in Definition \ref{def:specform}. From this, the statement in \ref{it:lem:pushforward1} follows by taking $x=\Phi^{-1}(y)$.

    For $N=1$, this follows immediately from \eqref{eq:push-forward-u}. Now, let $C$ be a differential operator of order  $N+1$. By linearity, we may assume that $C=\partial^\alpha$ with $\alpha\in\NN_0^d$, $|\alpha|=N+1$. Writing $\partial^\alpha = \partial_k\partial^{\alpha'}$ with $k\in \{1,\dots,d\}$ and $\alpha'\in\NN_0^d$, $|\alpha'|=N$, we get by induction and \eqref{eq:push-forward-u} that
    \begin{equation}\label{eq:commutator}
           \begin{aligned}  \partial^\alpha(u\circ\Phi) -(\partial^\alpha u)\circ \Phi &= \partial_k \partial^{\alpha'}(u\circ\Phi)  -(\partial^\alpha u)\circ\Phi \\
     &  =  \partial_k \big( (\partial^{\alpha'} u)\circ\Phi + B u\big) -(\partial^\alpha u)\circ \Phi  \\
     & = (\partial_k\partial^{\alpha'} u)\circ\Phi - (\partial_k h_1) ( (\partial_1\partial^{\alpha'} u)\circ\Phi) + \partial_k B u -(\partial^\alpha u)\circ \Phi \\
     & = -(\partial_k h_1) ((\partial_1\partial^{\alpha'} u)\circ\Phi )+ \partial_k B u, 
     \end{aligned}
    \end{equation}
     where $Bu$ is a linear combination of terms  of the form \eqref{eq:special_form}. The first term in 
     \eqref{eq:commutator} is obviously of special form \eqref{eq:special_form} of order  $N+1$. For the second term, we use the product rule to see that 
     $\partial_k Bu$ is a linear combination of terms of the form
     \begin{align*}  \partial_k \Big(\prod_{i=1}^\ell \partial^{\nu_i}h_1 \Big) \big((\partial^\mu u) \circ \Phi\big)
     & = \sum_{j=1}^\ell \Big(\prod_{i\neq j} \partial^{\nu_i}h_1\Big) (\partial_k\partial^{\nu_i }h_1) \big((\partial^\mu u)\circ\Phi\big)\\
    &  + \Big(\prod_{i=1}^\ell \partial^{\nu_i}h_1\Big) \big((\partial_k\partial^\mu u)\circ\Phi\big) -\partial_k h_1 \Big(\prod_{i=1}^\ell \partial^{\nu_i}h_1\Big) \big((\partial_1\partial^\mu u)\circ\Phi\big),
     \end{align*}
    where we have again used \eqref{eq:push-forward-u}. From this we see that all terms are of special form \eqref{eq:special_form} of order $N+1$, which finishes the proof of \ref{it:lem:pushforward1}. 
    
    \textit{Proof of \ref{it:lem:pushforward2}.} In the same way as above, we see that 
    \begin{equation}\label{eq:gradient-transform} 
    \partial_k(h_1\circ \Phi^{-1}) =  (\partial_k h_1)\circ\Phi^{-1} +(\partial_k h_2)\big( (\partial_1 h_1)\circ\Phi^{-1}\big),\qquad k\in \{1,\dots,d\}.
    \end{equation}
    To show \ref{it:lem:pushforward2}, it is sufficient to consider $|\delta|=1$, i.e., $\partial^\delta = \partial_k$ for some $k\in\{1,\dots,d\}$. By the product rule, we have 
    \[ \partial_k \Big(\prod_{i=1}^\ell \tilde{p}_i\Big) \partial^\mu =  \Big(\prod_{i=1}^\ell \tilde{p}_i\Big) \partial_k\partial^\mu + \sum_{i=1}^\ell \Big( \prod_{j\neq i} \tilde{p}_j \Big)(\partial_k \tilde{p}_i) \partial^\mu.\]
    Here, each term on the right-hand side is an operator of special form of order $N+1$. This is obvious if $\tilde{p}_i = \partial^{\nu_i}h_2$ and follows from \eqref{eq:gradient-transform} if $\tilde{p}_i = (\partial^{\nu_i} h_1)\circ\Phi^{-1}$.
\end{proof}

As a special case of Lemma \ref{lem:pushforward} we obtain the perturbation terms arising from the change of coordinates corresponding to the (bi-)Laplacian and $A$ in \eqref{eq:coordinatechangeA} and \eqref{eq:coordinatechangeLaplace}.
\begin{corollary}\label{cor:pert_terms_Psi}
  Let $\del^\Phi$ and $(\del^2)^\Phi$ be as in \eqref{eq:coordinatechangeLaplace}. Then for $j\in\{1,2\}$ it holds that
  \begin{equation*}
    (\del^j)^\Phi = \del^j + B_j,
  \end{equation*}
  where $B_j$ is an operator of special form of order $2j$. 
  Moreover, for $A^\Phi$ as in \eqref{eq:coordinatechangeA}, we have 
\begin{equation*}
    A^\Phi = A + B, \quad \text{ with }\quad B:=\begin{pmatrix}
         0 & 0\\ B_2& -\rho B_1
    \end{pmatrix}.
\end{equation*}
\end{corollary}

\begin{proof}
The statements on $(\del^j)^\Phi$ follow immediately from 
Lemma~\ref{lem:pushforward}\ref{it:lem:pushforward1}. The perturbation term $B$ for $A^\Phi$ simply follows from the structure of $A$ in \eqref{eq:A_matrix}.
\end{proof}

To apply perturbation theorems for $R$-sectoriality to $B$, we provide the necessary estimates on these perturbations. 
\begin{lemma}\label{lem:est_pert}
Let $p\in(1,\infty)$, $\kappa\in (0,1)$, $\gam\in ((1-\kappa)p-1,p-1)$, $k\geq 2$ an integer, $\rho>0$, and let $\OO$ be a special $\Cc^{1,\kappa}$-domain with $[\OO]_{C^{1,\kappa}}\leq 1$. 
Then the perturbation term $B$ from Corollary \ref{cor:pert_terms_Psi} satisfies the estimate
\begin{equation*} 
  \|B u \|_{\XX^{k,p}_\gam(\RRdh)}\leq C \cdot [\OO]_{C^{1,\kappa}}\cdot\|u \|_{\DD^{k,p}_\gam(\RRdh)},\quad u\in \DD^{k,p}_\gam(\RRdh),
\end{equation*}
where the constant $C>0$ only depends on $p,k,\kappa, \gam, \rho$ and $d$.
\end{lemma}
\begin{proof}
Let $u=(u_1,u_2)\in \DD^{k,p}_\gam(\RRdh)= W^{k+2,p}_0(\RRdh,w_{\gam+kp})\times W^{k,p}(\RRdh, w_{\gam+kp})$. 
By Corollary \ref{cor:pert_terms_Psi}, it suffices to prove the estimates
    \begin{align*}
        \|B_2 u_1\|_{W^{k-2,p}(\RRdh, w_{\gam+kp})}&\leq C\cdot [\OO]_{C^{1,\kappa}}\cdot \|u_1\|_{W^{k+2,p}(\RRdh, w_{\gam+kp})},\quad &u_1&\in W_0^{k+2,p}(\RRdh, w_{\gam+kp}),\\
        \|B_1 u_2\|_{W^{k-2,p}(\RRdh, w_{\gam+kp})}&\leq C\cdot [\OO]_{C^{1,\kappa}}\cdot \|u_2\|_{W^{k,p}(\RRdh, w_{\gam+kp})},\quad &u_2&\in W^{k,p}(\RRdh, w_{\gam+kp}).
    \end{align*}

Let $j\in\{1,2\}$ and set $v_j:=u_{3-j} \in W^{k-2+2j,p}_0(\RRdh, w_{\gam+kp})$, where it should be noted that for $j=1$ we have $v_j=u_2\in W^{k,p}(\RRdh, w_{\gam+kp})=W^{k,p}_0(\RRdh, w_{\gam+kp})$ by \eqref{eq:def:W_0}. Thus, we have to show that
\begin{equation}\label{eq:B_jestimate}
    \|B_j v_j \|_{W^{k-2,p}(\RRdh, w_{\gam+kp})}\lesssim [\OO]_{C^{1,\kappa}}\cdot \|v_j\|_{W^{k-2+2j,p}(\RRdh, w_{\gam+kp})},
\end{equation}
for $v_j \in W^{k-2+2j,p}_0(\RRdh, w_{\gam+kp})$.

Let $\alpha\in \NN_0^d$ with $|\alpha|\leq k-2$. By Corollary~\ref{cor:pert_terms_Psi}, we know that $B_j$ is an operator of special form of order $2j$. Moreover, by Lemma~\ref{lem:pushforward}\ref{it:lem:pushforward2}, the operator $\partial^\alpha B_j$ is of special form of order $2j+|\alpha|$. More precisely, $\partial^\alpha B_jv_j$ is a linear combination of terms of the form 
$\prod_{i=1}^\ell b_i \partial^\mu u$, where $b_i = (\partial^{\nu_i}h_1)\circ \Phi^{-1}$ or $b_i = \partial^{\nu_i}h_2$ with $\ell\in \{1,\dots,2j+|\alpha|\}$, $\nu_i\in\NN_0^d\setminus\{0\}$, $|\mu|\in\{1,\dots,2j+|\alpha|\}$ satisfying 
\begin{equation}\label{eq:cond_para}
    |\mu|+\sum_{i=1}^\ell |\nu_i| = \ell+2j+|\alpha|.
\end{equation}
To estimate a factor of the form $b_i=(\partial^{\nu_i}h_1)\circ \Phi^{-1}$, we apply Lemma \ref{lem:localisation_weighted_blow-up}\ref{it:lem:localisation_weighted_blow-up;dist_preserving} and \ref{it:lem:localisation_weighted_blow-up;est} to obtain the (non-optimal) estimate
\begin{equation}\label{eq:estimate_h1}
  |(\d^{\nu_i} h_1)(\Phi^{-1}(y))| \lesssim \frac{[\OO]_{C^{1,\kappa}}}{\mrm{dist}(\Phi^{-1}(y),\BDom)^{ |\nu_i|-1}}
\lesssim \frac{[\OO]_{C^{1,\kappa}}}{y_1^{|\nu_i|-1}},\qquad y \in \R^d_+.
\end{equation} 
For factors of the form $b_i=\partial^{\nu_i} h_2$, we use the same lemma to get 
\begin{equation}\label{eq:estimate_h2}
  |(\d^{\nu_i} h_2)(y)| 
\lesssim \frac{[\OO]_{C^{1,\kappa}}}{y_1^{|\nu_i|-1}},\qquad y \in \R^d_+.
\end{equation} 
For a term of the form $\prod_{i=1}^\ell b_i \partial^\mu v_j$, we obtain, using \eqref{eq:estimate_h1}, \eqref{eq:estimate_h2}, the fact that $[\OO]_{C^{1,\kappa}}\leq 1$, and Hardy's inequality (Lemma \ref{lem:Hardy} applied $\sum_{i=1}^{\ell}(|\nu_i|-1)$ times), 
\begin{align*}
    \Big\|\prod_{i=1}^\ell b_i \d^\mu v_j\Big\|_{L^p(\RRdh, w_{\gam+kp})}&\lesssim [\OO]_{C^{1,\kappa}}\cdot  \|\d^\mu v_j\|_{L^p(\RRdh, w_{\gam+kp - \sum_{i=1}^{\ell}(|\nu_i|-1)p})}\\
    &\lesssim  [\OO]_{C^{1,\kappa}}\cdot\|v_j\|_{W^{ |\mu| + \sum_{i=1}^{\ell}(|\nu_i|-1),p}(\RRdh, w_{\gam+kp})}\\
    &\lesssim  [\OO]_{C^{1,\kappa}}\cdot\|v_j\|_{W^{k-2+2j,p}(\RRdh, w_{\gam+kp})},
\end{align*}
where in the last estimate we have used \eqref{eq:cond_para} and $|\alpha|\leq k-2$ to obtain
\begin{equation*}
     |\mu| + \sum_{i=1}^{\ell}(|\nu_i|-1) =|\mu| -\ell  + \sum_{i=1}^{\ell}|\nu_i|\leq k-2 + 2j.
\end{equation*}
As $\partial^\alpha B_j v_j$ is a linear combination of such terms, this proves the desired estimate \eqref{eq:B_jestimate} and finishes the proof.
\end{proof}

With the estimates on the perturbation terms at hand, we can now prove Theorem \ref{thm:Rsect_special_dom}.
\begin{proof}[Proof of Theorem \ref{thm:Rsect_special_dom}]
  Let $\OO$ be a special $\Cc^{1,\kappa}$-domain and let $h_1, h_2$ and $\Phi$ be as in Lemma \ref{lem:localisation_weighted_blow-up}.
Recall that we introduced $A^{\Phi}: W^{4,1}_{\loc}(\RRdh)\times W^{2,1}_{\loc}(\RRdh) \to W^{2,1}_{\loc}(\RRdh) \times L^1_{\loc}(\RRdh)$ given by 
\begin{equation*}
A^{\Phi} := \Phi_* \circ A \circ (\Phi^{-1})_{*} = A + B,
\end{equation*}
with $B$ as in Corollary \ref{cor:pert_terms_Psi}.
   Consider the realisation of $A^\Phi$ in $\XX^{k,p}_\gam(\RRdh)$ with domain $D(A^\Phi):=\DD^{k,p}_\gam(\RRdh)$. Due to the isomorphisms in Proposition \ref{prop:isom}, the trace characterisation in Proposition \ref{prop:trace_char_dom} and standard properties of $R$-sectoriality, the desired statements in Theorem \ref{thm:Rsect_special_dom} for $A$ on $\XX^{k,p}_\gam(\OO)$ are equivalent to the corresponding statements for $A^\Phi$ on $\XX^{k,p}_\gam(\RRdh)$. To prove that $\lambda_0+A^\Phi$ on $\XX^{k,p}_\gam(\RRdh)$ is $R$-sectorial for $\lambda_0$ large enough and $[\OO]_{C^{1,\kappa}}$ small enough, we apply the perturbation result \cite[Theorem 16.2.4]{HNVW24}. Indeed, by Lemma \ref{lem:est_pert} and Theorem \ref{thm:Rsect_weight}, we have for all $u\in \DD^{k,p}_\gam(\RRdh)$ that
   \begin{equation*}
     \|B u \|_{\XX^{k,p}_\gam(\RRdh)}\lesssim  [\OO]_{C^{1,\kappa}}\cdot\|u\|_{\DD^{k,p}_\gam(\RRdh)}\eqsim  [\OO]_{C^{1,\kappa}}\cdot\|(\lambda_0+A)u\|_{\XX^{k,p}_\gam(\RRdh)}.
   \end{equation*}
   Therefore, applying \cite[Theorem 16.2.4]{HNVW24} gives the desired result.
\end{proof}

\subsection{Localisation of $R$-sectoriality to bounded domains}\label{subsec:pert_bdd}
To obtain $R$-sectoriality on bounded domains, it remains to apply a localisation procedure based on lower-order perturbations. 
\begin{theorem}\label{thm:Rsect_dom}
    Let $p\in(1,\infty)$, $\kappa\in (0,1)$, $\gam\in ((1-\kappa)p-1,p-1)$, $k\geq 2$ an integer, $\rho>0$, $\sigma\in (\vartheta(\rho), \pi)$ and let $\OO$ be a bounded $C^{1,\kappa}$-domain. Let $A$ be the realisation of the operator in \eqref{eq:A_matrix} on $\XX^{k,p}_\gam(\OO)$ with domain $D(A):=\DD^{k,p}_\gam(\OO)$. Then there exists a $\tilde{\lambda}>0$ such that for all $\lambda_0>\tilde{\lambda}$, the operator $\lambda_0+A$ is $R$-sectorial with angle $\om_R(\lambda_0+A)\leq \sigma$.
\end{theorem}

To transfer $R$-sectoriality from special domains to bounded domains, we use the following abstract lemma based on lower-order perturbations. A version of the lemma below for the $\Hinf$-calculus can be found in \cite[Lemma 6.11]{LV18}.
\begin{lemma}\label{lem:abstract_localisation_Rsect}
   Let $A$ be a linear operator on a Banach space $Y$, and let $\tilde{A}$ be an $R$-sectorial operator on a Banach space $\tilde{Y}$. Assume that there exist bounded linear mappings $\II\colon Y\to \tilde{Y}$ and $\PP\colon \tilde{Y}\to Y$ satisfying
  \begin{enumerate}[(i)]
    \item \label{it:lemLV6.11_1} $\PP \II = \id$,
    \item \label{it:lemLV6.11_2} $\II D(A)\subseteq D(\tilde{A})$ and $\PP D(\tilde{A})\subseteq D(A)$,
    \item \label{it:lemLV6.11_3} $(\II A - \tilde{A}\II)\PP\colon D(\tilde{A})\to \tilde{Y}$ and $\II(A\PP-\PP\tilde{A})\colon D(\tilde{A})\to \tilde{Y}$ extend to bounded linear operators $[\tilde{Y},D(\tilde{A})]_{\theta}\to \tilde{Y}$ for some $\theta\in (0,1)$.
  \end{enumerate}
  Then $A$ is a closed and densely defined operator and for every $\sigma>\om_{R}(\tilde{A})$ there exists a $\lambda>0$ such that $\lambda+A$ is $R$-sectorial with $\om_{R}(\lambda+A)\leq \sigma$.
\end{lemma}
\begin{proof}
  The proof is similar to the proof of \cite[Lemma 6.11]{LV18} if one uses the perturbation theorem \cite[Corollary 16.2.5]{HNVW24} for $R$-sectoriality.
\end{proof}

To apply the localisation procedure from Lemma \ref{lem:abstract_localisation_Rsect}, we use the decomposition of weighted Sobolev spaces as in Lemma \ref{lem:decomp}. This requires the $R$-sectoriality of $A$ on special domains (Theorem \ref{thm:Rsect_special_dom}) and on the full space $\RRd$. From Proposition \ref{prop:DS_Rd} and a lifting argument, we obtain higher-order regularity for $A$ on $\RRd$.
\begin{proposition}\label{prop:higher_reg_Rd}
    Let $p\in(1,\infty)$, $\gam\in (-1,p-1)$, $k\geq 2$ an integer, $\rho>0$, $\sigma\in(\vartheta(\rho),\pi)$ and let $A$ be the realisation of the operator in \eqref{eq:A_matrix} on $W^{k,p}(\RRd, w_\gam)\times W^{k-2,p}(\RRd, w_\gam)$ with domain $D(A):=W^{k+2,p}(\RRd, w_\gam)\times W^{k,p}(\RRd, w_\gam)$. Then there exists a $\tilde{\lambda}>0$ such that for all $\lambda_0>\tilde{\lambda}$, the operator $\lambda_0+ A$ is $R$-sectorial with angle $\om_R(\lambda_0+A)\leq \sigma$.
\end{proposition}
\begin{proof}
Recall that the Bessel potential operator is defined by
\begin{equation*}
    J_s f:=(1-\del)^{\frac{s}{2}} f=\FF^{-1}((1+|\cdot|^2)^{\frac{s}{2}}\FF f),\qquad f\in \SS'(\RRd),\; s\in \RR. 
\end{equation*}
Then, $J_{\ell-\ell_0}$ is an isometry from $W^{\ell,p}(\RRd, w_\gam)$ to $W^{\ell_0,p}(\RRd, w_\gam)$ for all $p\in(1,\infty)$ and $\ell, \ell_0\in\NN_0$ with $\ell\geq \ell_0$. Indeed, for the Bessel potential scale this follows from the definition and by \cite[Proposition 3.2]{MV15} it also holds for Sobolev spaces. 

 Let $f\in \SS(\RRd)\times \SS(\RR^d)$, $\lambda_0>0$ and $\lambda\in \lambda_0+ \Sigma_{\pi-\sigma}$. Consider the equation $(\lambda+ A)u = f$ which has the solution representation
    \begin{equation*}
        u(x) = \mc{F}^{-1}\big(\xi \mapsto A^{-1}(\xi, \lambda) (\mc{F}f)(\xi)\big)(x)
    \end{equation*}
where, following the notation of \cite[Section 2]{DS15}, we have that
\begin{equation*}
    A(\xi,\lambda): =\begin{pmatrix}
        \lambda & -1\\|\xi|^4 & \lambda + \rho|\xi|^2
    \end{pmatrix}
\end{equation*}
is invertible and its inverse (see \cite[Equation (2.3)]{DS15}) is given by
\begin{equation*}
    A^{-1}(\xi, \lambda)=\frac{1}{(\alpha_+\lambda + |\xi|^2)(\alpha_-\lambda+|\xi|^2)}\begin{pmatrix}
        \lambda+ \rho |\xi|^2 & 1 \\ -|\xi|^4& \lambda
    \end{pmatrix}.
\end{equation*} 
Now, the solution to $(\lambda+ A)\tilde{u}= J_{k-2} f \in \SS(\RRd)\times \SS(\RRd)$ is given by
\begin{align*}
    \tilde{u}(x)&=\mc{F}^{-1}\big(\xi \mapsto A^{-1}(\xi, \lambda) (\mc{F}J_{k-2} f)(\xi)\big)(x)\\
    &= \mc{F}^{-1}\big(\xi\mapsto (1+|\xi|^2)^{\frac{k-2}{2}}A^{-1}(\xi, \lambda)(\mc{F}f)(\xi)\big)(x) = (J_{k-2} u)(x).
\end{align*}
    This proves that $R(\lambda, -A)J_{k-2} f = J_{k-2} R(\lambda, -A)f$, which is equivalent to
    \begin{equation}\label{eq:lift_RsectR^d}
        R(z, \lambda_0+A)J_{k-2} f = J_{k-2} R(z, \lambda_0+A)f, \quad \text{ for }z\in \CC\setminus\overline{\Sigma}_{\sigma}.
    \end{equation}

    We can now lift the $R$-sectoriality to higher-order spaces. Fix $N\in\NN_1$, let $(\lambda_n)_{n=1}^N$ be a sequence in $\CC\setminus\overline{\Sigma}_{\sigma}$ and let $(f^{(n)})_{n=1}^N\in \eps_N^2(\SS(\RRd)\times \SS(\RRd))$. By Proposition \ref{prop:DS_Rd} and the definition of $R$-sectoriality, we have the estimate
    \begin{equation}\label{eq:lift_k=2}
    \begin{aligned}
                \big\|\sum_{n=1}^N \eps_n \lambda_n R(\lambda_n, \lambda_0 &+ A)f^{(n)} \big\|_{L^2(\Om ; W^{2,p}(\RRd, w_\gam)\times L^p(\RRd, w_\gam))}\\
                &\;\lesssim \big\|\sum_{n=1}^N \eps_n f^{(n)} \big\|_{L^2(\Om ; W^{2,p}(\RRd, w_\gam)\times L^p(\RRd, w_\gam))}.
    \end{aligned}
    \end{equation}
    From \eqref{eq:lift_RsectR^d} and \eqref{eq:lift_k=2} it follows that
    \begin{align*}
        \big\|\sum_{n=1}^N &\eps_n \lambda_n R(\lambda_n, \lambda_0 + A)f^{(n)} \big\|_{L^2(\Om ; W^{k,p}(\RRd,w_\gam)\times W^{k-2,p}(\RRd,w_\gam))}\\
        &\eqsim \big\|\sum_{n=1}^N \eps_n \lambda_n R(\lambda_n, \lambda_0 + A)J_{k-2}f^{(n)} \big\|_{L^2(\Om ; W^{2,p}(\RRd,w_\gam)\times L^p(\RRd, w_\gam))}\\
        &\lesssim \big\|\sum_{n=1}^N \eps_n J_{k-2}f^{(n)} \big\|_{L^2(\Om ; W^{2,p}(\RRd,w_\gam)\times L^p(\RRd,w_\gam))}\\
        &\eqsim \big\|\sum_{n=1}^N \eps_n f^{(n)} \big\|_{L^2(\Om ; W^{k,p}(\RRd,w_\gam)\times W^{k-2,p}(\RRd, w_\gam))}.
    \end{align*}
    By a density argument, this proves the required $R$-sectoriality estimate.
\end{proof}

We can turn to the proof of our main result concerning $R$-sectoriality on domains.
\begin{proof}[Proof of Theorem \ref{thm:Rsect_dom}]
The result will be proved using Lemma \ref{lem:abstract_localisation_Rsect} together with Theorem \ref{thm:Rsect_special_dom} and Proposition \ref{prop:higher_reg_Rd}.

Let $\eps\in (0,\kappa)$ be such that $\gam>(1-(\kappa-\eps))p-1$. 
    Take $(V_n)_{n=1}^N,(\OO_n)_{n=1}^N, (\eta_n)_{n=0}^N$ from Lemma \ref{lem:decomp} such that for all $n\in\{1,\dots, N\}$ we have $[\OO_n]_{C^{1,\kappa-\eps}}< \delta$ where $\delta\in(0,1)$ is small enough such that Theorem \ref{thm:Rsect_special_dom} (applied with $\kappa$ replaced by $\kappa-\eps$) applies for every $\OO_n$. Note that this is possible due to Remark \ref{rem:special_dom}. We define the following operators, where the spaces are as defined in \eqref{eq:Fk} and \eqref{eq:Fk0}.
\begin{enumerate}[(i)]
    \item $\tilde{A}:= \bigoplus_{n=0}^N\tilde{A}_n$ on $\mc{W}^{k,p}_{\gam+kp}\times \mc{W}^{k-2,p}_{\gam+kp}$ with $D(\tilde{A})= \mc{W}^{k+2,p}_{0, \gam+kp}\times \mc{W}^{k,p}_{\gam+kp}$, where
    \begin{enumerate}
        \item $\tilde{A}_0$ on $W^{k,p}(\RRd)\times W^{k-2,p}(\RRd)$ with $D(\tilde{A}_0)= W^{k+2,p}(\RRd)\times W^{k,p}(\RRd)$ is given by $\tilde{A}_0\tilde{u} = A\tilde{u}$,
        \item $\tilde{A}_n$ on $W^{k,p}(\OO_n, w_{\gam+kp}^{\d\OO_n})\times W^{k-2,p}(\OO_n, w_{\gam+kp}^{\d\OO_n})$ with $D(\tilde{A}_n)= W^{k+2,p}_0(\OO_n, w_{\gam+kp}^{\d\OO_n})\times W^{k,p}(\OO_n, w_{\gam+kp}^{\d\OO_n})$ is given by $\tilde{A}_n \tilde{u} =A\tilde{u}$ for all $n\in\{1,\dots, N\}$.
    \end{enumerate}
    \item $B: D(A)=\DD^{k,p}_\gam(\OO)\to \mc{W}^{k,p}_{\gam+kp}\times \mc{W}^{k-2,p}_{\gam+kp} $ given by
    \begin{equation*}
        B u = B \begin{pmatrix}
            u_1\\u_2
        \end{pmatrix}=
        \begin{pmatrix}
            0\\\big([\del^2, \eta_n]u_1 -\rho [\del, \eta_n]u_2\big)
        \end{pmatrix}_{n=0}^N.
    \end{equation*}
    \item $C: D(\tilde{A})\to \XX^{k,p}_\gam(\OO)$ given by
    \begin{equation*}
        C \tilde{u} = C \begin{pmatrix}
            \tilde{u}_1\\\tilde{u}_2
        \end{pmatrix}=
        \begin{pmatrix}
            0\\\sum_{n=0}^N\big([\del^2, \eta_n]\tilde{u}_{1,n} -\rho [\del,\eta_n]\tilde{u}_{2,n}\big)
        \end{pmatrix}.
    \end{equation*}
\end{enumerate}
For $\lambda_0>0$ large enough, $\lambda_0+\tilde{A}_n$ for all $n\in\{0,\dots, N\}$ is $R$-sectorial with $\om_R(\lambda_0+\tilde{A}_n)\leq \sigma$ by Proposition \ref{prop:higher_reg_Rd} and Theorem \ref{thm:Rsect_special_dom}. Hence, for $\lambda_0>0$ large enough, also $\lambda_0+ \tilde{A}$ is $R$-sectorial with $\om_R(\lambda_0+\tilde{A})\leq \sigma$.

Let $\mc{P}$ and $\mc{I}$ be as in \eqref{eq:retraction}. It is straightforward to verify that the conditions \ref{it:lemLV6.11_1} and \ref{it:lemLV6.11_2} from Lemma \ref{lem:abstract_localisation_Rsect} hold. It remains to check condition \ref{it:lemLV6.11_3} in Lemma \ref{lem:abstract_localisation_Rsect}. We first characterise the complex interpolation space
\begin{equation*}
    [\mc{W}^{k,p}_{\gam+kp}\times \mc{W}^{k-2,p}_{\gam+kp}, D(\tilde{A})]_\half = [\mc{W}^{k,p}_{\gam+kp},\mc{W}^{k+2,p}_{0, \gam+kp} ]_\half \times [\mc{W}^{k-2,p}_{\gam+kp},\mc{W}^{k,p}_{\gam+kp} ]_\half.
\end{equation*}
Recall that $(\OO_n)_{n=1}^N$ are special $\Cc^{1,\kappa}$-domains. Hence, using Propositions \ref{prop:isom} and \ref{prop:trace_char_dom} to reduce to $\RRdh$ and applying \cite[Proposition 6.2]{Ro25}, it follows that
\begin{equation}\label{eq:int_W0}
    [W^{k,p}(\OO_n, w_{\gam+kp}^{\d\OO_n}), W^{k+2,p}_0(\OO_n, w_{\gam+kp}^{\d\OO_n})]_\half = W^{k+1,p}_0(\OO_n, w_{\gam+kp}^{\d\OO_n}).
\end{equation}
From \eqref{eq:int_W0} and interpolation on $\RRd$ (see, e.g., \cite[Theorems 5.6.9 \& 5.6.11]{HNVW16}), we obtain
\begin{align*}
    [\mc{W}^{k,p}_{\gam+kp},&\mc{W}^{k+2,p}_{0, \gam+kp} ]_\half\\
    & = [W^{k,p}(\RRd), W^{k+2,p}(\RRd)]_\half \oplus \bigoplus_{n=1}^N [W^{k,p}(\OO_n, w_{\gam+kp}^{\d\OO_n}), W^{k+2,p}_0(\OO_n, w_{\gam+kp}^{\d\OO_n})]_\half\\
    &=W^{k+1,p}(\RRd)\oplus\bigoplus_{n=1}^N W^{k+1,p}_0(\OO_n, w_{\gam+kp}^{\d\OO_n}) = \mc{W}^{k+1,p}_{0,\gam+kp}. 
\end{align*}
Similarly, since $\Phi_*: W^{k_0,p}(\OO_n, w_{\gam_0}^{\d\OO_n})\to W^{k_0,p}(\RRdh, w_{\gam_0})$ is an isomorphism for all $k_0\in\NN_0$ and all $\gam_0\in((k_0-1)p-1, \infty)\setminus\{jp-1:j\in \NN_1\}$ (see Proposition \ref{prop:isom}), it follows from \cite[Theorem 6.4]{Ro25} that
\begin{equation*}
    [W^{k-2,p}(\OO_n, w_{\gam+kp}^{\d\OO_n}), W^{k,p}(\OO_n, w_{\gam+kp}^{\d\OO_n})]_\half = W^{k-1,p}(\OO_n, w_{\gam+kp}^{\d\OO_n}).
\end{equation*}
Hence, we obtain $[\mc{W}^{k-2,p}_{\gam+kp},\mc{W}^{k,p}_{\gam+kp} ]_\half = \mc{W}^{k-1,p}_{\gam+kp}$ and thus we have proved that
\begin{equation}\label{eq:intp_space_mcW}
     [\mc{W}^{k,p}_{\gam+kp}\times \mc{W}^{k-2,p}_{\gam+kp}, D(\tilde{A})]_\half = \mc{W}^{k+1,p}_{0,\gam+kp}\times \mc{W}^{k-1,p}_{\gam+kp}.
\end{equation}
To continue, note that
\begin{equation*}
    \mc{I}A u - \tilde{A}\mc{I} u = - Bu,\quad u\in D(A), \quad \text{ and } \quad A\mc{P}\tilde{u}-\mc{P}\tilde{A}\tilde{u}=C\tilde{u},\quad \tilde{u}\in D(\tilde{A}).
\end{equation*}
Furthermore, the commutators $[\del, \eta_n]$ and $[\del^2, \eta_n]$ are first and third-order partial differential operators with smooth and compactly supported coefficients, respectively. This and \eqref{eq:intp_space_mcW} yield that the operators
\begin{align*}
    \mc{I}A -\tilde{A}\mc{I}&: W^{k+1,p}_0(\OO, w_{\gam+kp}^{\d\OO})\times W^{k-1,p}(\OO, w_{\gam+kp}^{\d\OO})  \to \mc{W}^{k,p}_{\gam+kp}\times \mc{W}^{k-2,p}_{\gam+kp},\\
    \mc{P}&:[\mc{W}^{k,p}_{\gam+kp}\times \mc{W}^{k-2,p}_{\gam+kp}, D(\tilde{A})]_\half \to W^{k+1,p}_0(\OO, w_{\gam+kp}^{\d\OO})\times W^{k-1,p}(\OO, w_{\gam+kp}^{\d\OO})
\end{align*}
are bounded. Similarly, it can be shown that the operators
\begin{align*}
    A\mc{P}-\mc{P}\tilde{A}&:[\mc{W}^{k,p}_{\gam+kp}\times \mc{W}^{k-2,p}_{\gam+kp}, D(\tilde{A})]_\half \to W^{k,p}(\OO, w_{\gam+kp}^{\d\OO})\times W^{k-2,p}(\OO, w_{\gam+kp}^{\d\OO}),\\
    \mc{I}&: W^{k,p}(\OO, w_{\gam+kp}^{\d\OO})\times W^{k-2,p}(\OO, w_{\gam+kp}^{\d\OO})\to \mc{W}^{k,p}_{\gam+kp}\times \mc{W}^{k-2,p}_{\gam+kp} 
\end{align*}
    are bounded. This shows that $(\II A - \tilde{A}\II)\PP$ and $\II(A\PP-\PP\tilde{A})$ extend to bounded operators from
 $[\mc{W}^{k,p}_{\gam+kp}\times \mc{W}^{k-2,p}_{\gam+kp}, D(\tilde{A})]_\half$ to $\mc{W}^{k,p}_{\gam+kp}\times \mc{W}^{k-2,p}_{\gam+kp} $. Applying Lemma \ref{lem:abstract_localisation_Rsect} gives that for all $\sigma\in(\vartheta(\rho),\pi)$ there exists a $\tilde{\lambda}>0$ such that for all $\lambda_0>\tilde{\lambda}$ the operator $\lambda_0+A$ on $\XX^{k,p}_\gam(\OO)$ is $R$-sectorial with $\om_{R}(\lambda_0+A)\leq \sigma$. 
\end{proof}

\subsection{$R$-sectoriality for the unshifted operator} \label{subsec:pert_bdd_noshift}

In this section, we prove the following improved version of Theorem \ref{thm:Rsect_dom}, which, in particular, yields $R$-sectoriality for the operator $A$ without a shift.
\begin{theorem}\label{thm:Rsect_dom_noshift}
     Let $p\in(1,\infty)$, $\kappa\in (0,1)$, $\gam\in((1-\kappa)p-1,p-1)$, $k\geq 2$ an integer, $\rho>0$,  and let $\OO$ be a bounded $C^{1,\kappa}$-domain. Let $A$ be the realisation of the operator in \eqref{eq:A_matrix} on $\XX^{k,p}_\gam(\OO)$ with domain $D(A):= \DD^{k,p}_\gam(\OO)$. Then the following assertions hold.
    \begin{enumerate}[(i)]
        \item\label{it:thm:Rsect_dom_noshift1} The spectrum $\sigma(A)$ is a discrete set of eigenvalues which is contained in the open right half-plane. Furthermore, $\sigma(A)$ is independent of $p\in(1,\infty)$, $\gam\in ((1-\kappa)p-1,p-1)$, and $k\geq 2$.
        \item\label{it:thm:Rsect_dom_noshift2} There exists a $\tilde{\lambda}>0$ and a $\sigma < \frac\pi 2$ such that for all $\lambda_0>-\tilde{\lambda}$, the operator $\lambda_0+A$ is $R$-sectorial with $\om_{R}(\lambda_0+A)\leq \sigma$.
    \end{enumerate}
\end{theorem}
Essentially, the $R$-sectoriality result in Theorem \ref{thm:Rsect_dom_noshift}\ref{it:thm:Rsect_dom_noshift2} follows from Theorem \ref{thm:Rsect_dom} and Theorem \ref{thm:Rsect_dom_noshift}\ref{it:thm:Rsect_dom_noshift1}. However, since we are dealing with weighted spaces and domains with low regularity, we cannot argue immediately as in the proof of \cite[Theorem 5.1]{DS15}.
To prove Theorem \ref{thm:Rsect_dom_noshift}\ref{it:thm:Rsect_dom_noshift1}, we consider the operators $A$ on $\XX^{k,p}_\gam(\OO)$ and on $W^{2,2}_0(\OO)\times L^2(\OO)$. Then it suffices to show that the spectra of these two operators are equal and that the spectrum of $A$ on $W^{2,2}_0(\OO)\times L^2(\OO)$ has the desired properties. \\

We start with the analysis of the operator in the unweighted $L^2$-setting. For sufficiently smooth
domains, the following result is an easy consequence of classical elliptic regularity. As the
domain $\OO$ is of class $C^{1,\kappa}$ only, we have to argue slightly differently. 

\begin{lemma}\label{lem:spA}
    Let $\rho>0$ and let $\OO$ be a bounded $C^1$-domain. Let $\mc{A}$ be the realisation of the operator in \eqref{eq:A_matrix} on $W^{2,2}_0(\OO)\times L^2(\OO)$ with domain 
    \begin{equation*}
        D(\mc{A}):= \big\{u\in W^{2,2}_0(\OO)\times L^2(\OO): \mc{A} u \in W^{2,2}_0(\OO)\times L^2(\OO)\big\}.
    \end{equation*}
    Then the following assertions hold.
    \begin{enumerate}[(i)]
        \item\label{it:lem:spA1} The operator $\mc A$ is densely defined and closed.
        \item\label{it:lem:spA2} Let $s\in (0,\frac12)$. Then we have the higher regularity $u_1\in W^{4,2}_{\loc}(\OO)$
        for each $u=(u_1,u_2)^\top\in D(\mc A)$ as well as the continuous embedding
        \begin{equation}
    \label{eq:closed2} D(\mc A)  \subseteq \big(  H^{2+s,2}(\OO)\cap  W^{2,2}_0(\OO)\big)\times W^{2,2}_0(\OO), 
    \end{equation}
        where $H^{2+s,2}(\OO)$ stands for the Bessel potential space.
        \item\label{it:lem:spA3} The spectrum $\sigma(\mc{A})$ is a discrete set of eigenvalues which is contained in the open right half-plane.
    \end{enumerate}
\end{lemma}

\begin{proof} \textit{Proof of \ref{it:lem:spA1}.}
As $\Cc^{\infty}(\OO) \times \Cc^{\infty}(\OO) \subseteq D(\mc A)$, the operator $\mc A$ is densely defined. To show closedness, let $(u^{(k)})_{k\in\NN}=((u_1^{(k)}, u_2^{(k)})^\top)_{k\in\NN}\subseteq D(\mc A)$ be a sequence with $u^{(k)}\to 
u$ and $\mc A u^{(k)}\to v$ as $k\to \infty$ in $W^{2,2}_0(\OO)\times L^2(\OO)$. By the definition of $\mc A$, this implies
$u^{(k)}_2\to - v_1$ in $W^{2,2}_0(\OO)$. As we also have $u^{(k)}_2\to u_2$ in $L^2(\OO)$, the uniqueness
of the limit in $L^2(\OO)$ yields $u_2=-v_1\in W^{2,2}_0(\OO)$. 

Similarly, we have the convergence $\Delta^2 u_1^{(k)} -\rho \Delta u_2^{(k)} \to v_2$ in $L^2(\OO)$ which
implies 
\begin{equation}
    \label{eq:closed1}
    \Delta^2 u_1^{(k)} \to v_2 + \rho\Delta u_2 \quad\text{in } L^2(\OO),
\end{equation}
using that $u^{(k)}_2\to - v_1=u_2$ in $W^{2,2}_0(\OO)$. By the continuous embedding $L^2(\OO)\subseteq
W^{-2,2}(\OO)$, the convergence in \eqref{eq:closed1} also holds in $W^{-2,2}(\OO)$. On the other 
hand, from $u_1^{(k)}\to u_1$ in $W^{2,2}_0(\OO)$ and the continuity of $\Delta^2\colon W^{2,2}_0(\OO)\to W^{-2,2}(\OO)$, we also get
\[\Delta^2 u_1^{(k)} \to \Delta^2 u_1 \quad\text{in }W^{-2,2}(\OO).\]
Again by the uniqueness of the limit, we see that $\Delta^2 u_1 = v_2+\rho\Delta u_2$. Therefore,
$u=(u_1,u_2)^\top\in D(\mc A)$ and $\mc A u = (-u_2, \Delta^2 u_1-\rho\Delta u_2)^\top = (v_1,v_2)^\top$,
which shows that the operator $\mc A$ is closed.

\textit{Proof of \ref{it:lem:spA2}.} Let $u=(u_1,u_2)^\top\in D(\mc A)$. Then we obtain $u_1\in  W^{4,2}_\loc(\OO)$ by a similar argument as in the proof of \cite[Lemma 11.36]{Ne22}. Setting $f=(f_1,f_2)^\top := \mc A u$, we get
\[ \Delta^2 u_1 = \rho\Delta u_2 + f_2 = -\rho\Delta f_1+f_2 \in L^2(\OO).\]
Let $s\in (0,\frac12)$. By \cite[Theorem~6]{Savare98}, the function $u_1$ belongs to $H^{2+s,2}(\OO)$, and we have the estimate
\[ \|u_1\|_{H^{s+2,2}(\OO)} \le C \big\| -\rho\Delta f_1+f_2\big\|_{H^{s-2,2}(\OO)} \le C
\big\| -\rho\Delta f_1+f_2\big\|_{L^2(\OO)}\le C \|f\|_{W^{2,2}(\OO)\times L^2(\OO)}.\]
From this and $\|u_2\|_{W^{2,2}(\OO)} = \|f_1\|_{W^{2,2}(\OO)}$, we obtain the continuous embedding
\eqref{eq:closed2}, where we endow $D(\mc A)$ with the graph norm. 

\textit{Proof of \ref{it:lem:spA3}.} By the Rellich--Kondrachov theorem, the space on the right-hand side of \eqref{eq:closed2} is
compactly embedded into $W^{2,2}_0(\OO)\times L^2(\OO)$. Therefore, the operator $\mc A$ has compact
resolvent and, consequently, its spectrum is discrete and consists of eigenvalues of finite
algebraic multiplicity. Therefore, to show the statement on the spectrum in \ref{it:lem:spA3}, it is 
sufficient to show that for each $\lambda$ with $\Re\lambda\ge 0$, the operator $\mc A+\lambda$ is injective.

Let $\lambda\in\C$ with $\Re\lambda\ge 0$. On $W^{2,2}_0(\OO)\times L^2(\OO)$, define the inner product
    \begin{equation*}
        \langle v,w\rangle:= (v_1, w_1)_{L^2(\OO)} + (\del v_1, \del w_1)_{L^2(\OO)} +(v_2, w_2)_{L^2(\OO)},\quad v,w \in W^{2,2}_0(\OO)\times L^2(\OO).
    \end{equation*}
    For $u =(u_1, u_2)^\top\in D(\mc{A})$ with $(\mc A+\lambda)u=0$, we obtain by integration by parts and using the boundary conditions, 
    \begin{equation*}
        0 = \Re \,\langle (\mc A+\lambda) u, u \rangle =(\Re \lambda)  \big( \|\Delta u_1\|_{L^2(\OO)}^2 +
        \|u_2\|_{L^2(\OO)}^2 \big)+ \rho \|\nabla u_2\|_{L^2(\OO;\C^d)}^2.
    \end{equation*}
As $\Re\lambda\ge 0$, this yields $\nabla u_2=0$. Hence, since $u_2\in W^{2,2}_0(\OO)$, we obtain $u_2=0$.
Furthermore, the second component of the equation $(\mc{A}+\lambda)u =0 $ implies that, $\del^2 u_1 = \rho\del u_2-\lambda u_2 =0$. Since $u_1\in W^{2,2}_0(\OO)$, this implies, using integration by parts again,
that $0= (\Delta^2 u_1,u_1)_{L^2(\OO)} = \|\Delta u_1\|_{L^2(\OO)}^2$ and therefore $u_1=0$. 
So we see that $u=0$, which shows that the operator $\mc A+\lambda$ is injective. 
\end{proof}

In order to show that the spectrum of $A$ on $\XX^{k,p}_\gam(\OO)$ is independent of $p$, $\gam$ and $k$, we need the following lemma about the consistency of resolvents.
\begin{lemma}\label{lem:consis}
    Let $p\in(1,\infty)$, $\kappa\in(0,1)$, $\gam\in ((1-\kappa)p-1,p-1)$, $k\geq 2$ an integer and let $\OO$ be a bounded $C^{1,\kappa}$-domain. Let 
    \begin{equation*}
        A_{p,k,\gam}:= A \quad \text{ on }\XX^{k,p}_\gam(\OO)\text{ with }D(A_{p,k,\gam}):= \DD^{k,p}_\gam(\OO)
    \end{equation*}
    be as in \eqref{eq:A_matrix}. Let $\mc{A}$ be the operator from Lemma~\ref{lem:spA}, i.e., the realisation of the operator in \eqref{eq:A_matrix} on $W^{2,2}_0(\OO)\times L^2(\OO)$ with domain 
    \begin{equation*}
        D(\mc{A}):= \big\{u\in W^{2,2}_0(\OO)\times L^2(\OO): \mc{A} u \in W^{2,2}_0(\OO)\times L^2(\OO)\big\}.
    \end{equation*}
    Then there exists a $\tilde{\lambda}>0$ such that for all $\lambda> \tilde{\lambda}$ the resolvents $R(\lambda, -A_{p,k,\gam})$ and $R(\lambda, -\mc{A})$ are consistent.
\end{lemma}

\begin{proof}
Let $f\in \XX^{k,p}_\gam(\OO) \cap \big(W^{2,2}_0(\OO)\times L^2(\OO)\big)$. We have to
 show that, for sufficiently large $\lambda$,  the resolvents $R(\lambda, -A_{p,k,\gam})$ and $R(\lambda, -\mc{A})$ exist and are equal when applied to $f$. For this, we 
 first show an embedding result.

 Take $1<q<\min\{p,2\}$ and $\mu\in (2q-1, 3q-1)$ such that 
    \begin{equation}\label{eq:cond_kappa}
        \mu>  \frac{q(\gam+2p+1)}{p}-1> (3-\kappa)q-1. 
    \end{equation}
    Using that $\gam\in ((1-\kappa)p-1,p-1)$, we can always find a $\mu$ such that \eqref{eq:cond_kappa} holds. First, we claim that 
    \begin{equation}\label{eq:claim}
        L^p(\OO, w_{\gam+2p}^{\d\OO})\hookrightarrow L^q(\OO, w_{\mu}^{\d\OO}).
    \end{equation} 
    Indeed, for $u\in L^p(\OO, w_{\gam+2p}^{\d\OO})$ we have by H{\"o}lder's inequality that
    \begin{equation*}
        \int_{\OO}|u(x)|^q w_{\mu}^{\d\OO}(x)\dd x \leq \Big(\int_{\OO} |u(x)|^pw_{\gam+2p}^{\d\OO}(x)\dd x\Big)^{\frac{q}{p}} \Big(\int_\OO w_{\frac{\mu p - q(\gam+2p)}{p-q}}^{\d\OO}(x)\dd x\Big)^{\frac{p-q}{q}}<\infty.
    \end{equation*}
    The latter integral can be written as an integral over $\RRdh$ (using the localisation from Lemma \ref{lem:decomp} and the diffeomorphism from Lemma \ref{lem:localisation_weighted_blow-up}), hence the integral is finite since \eqref{eq:cond_kappa} implies $(\mu p - q(\gam+2p))/ (p-q)>-1$. This proves the claim.
    
    We introduce for $r\in (1,\infty)$ and $\nu>-1$ the space
    \begin{equation*}
        Z_{r,\nu}:=\big\{u\in W^{2,r}_0(\OO, w_{\nu}^{\d\OO})\times L^r(\OO, w_{\nu}^{\d\OO}): \mc{A} u \in W^{2,r}_0(\OO, w_{\nu}^{\d\OO})\times L^r(\OO, w_{\nu}^{\d\OO})\big\}.
    \end{equation*}
    Note that $D(\mc{A})= Z_{2,0}$. For $\lambda$ large enough, consider the equation
    \begin{equation}\label{eq:resol_eq}
        \lambda  \tilde{w} + A \tilde{w} = f ,\qquad f\in \XX^{k,p}_\gam(\OO) \cap \big(W^{2,2}_0(\OO)\times L^2(\OO)\big).
    \end{equation}
     Then by Theorem \ref{thm:Rsect_dom}, there exists a unique solution $\tilde{w}_0\in \DD^{k,p}_\gam(\OO)$. 
    Hardy's inequality (for Lipschitz domains, see e.g., \cite[Theorem 8.8]{Ku85}) and \eqref{eq:claim} yield 
    \begin{align*}
        \tilde{w}_0\in\DD^{k,p}_\gam(\OO) &\hookrightarrow \DD^{2,p}_\gam(\OO) = W^{4,p}_0(\OO, w_{\gam+2p}^{\d\OO})\times W^{2,p}(\OO, w_{\gam+2p}^{\d\OO})\\
        &\hookrightarrow W^{4,q}_0(\OO, w_{\mu}^{\d\OO})\times W^{2,q}(\OO, w_{\mu}^{\d\OO}). 
    \end{align*}
   Furthermore, for each $\lambda\in\CC$ with $\Re\lambda\ge  0$, we have $\lambda\in\rho(-\mc A)$ by Lemma \ref{lem:spA}\ref{it:lem:spA3}. As $\mc A$ is closed by Lemma  \ref{lem:spA}\ref{it:lem:spA1}, this is equivalent to the bijectivity of the operator 
     $\lambda+\mc A$. Therefore, there also exists a unique solution $\tilde{w}_1\in Z_{2,0}$ of \eqref{eq:resol_eq}. Using $\mu>0$, $q<2$ and the elliptic regularity (Theorem \ref{thm:Rsect_dom} using \eqref{eq:cond_kappa}), we find
    \begin{equation*}
      \tilde{w}_1\in   Z_{2,0}\hookrightarrow Z_{q, \mu}= W^{4,q}_0(\OO, w_{\mu}^{\d\OO})\times W^{2,q}(\OO, w_{\mu}^{\d\OO}).
    \end{equation*}
    Note that \eqref{eq:resol_eq} with right-hand side $f\in L^q(\OO, w_{\mu}^{\d\OO})$ has a unique solution $\tilde{w}$ by Theorem \ref{thm:Rsect_dom} (again using \eqref{eq:cond_kappa}). It follows that $\tilde{w}_0=\tilde{w}_1$, which proves that the resolvents of $-A_{p,k,\gam}$ and $-\mc{A}$ are consistent.
    \end{proof}

Using the above lemmas about the spectrum of $\mc{A}$ and the consistency, we can now prove Theorem \ref{thm:Rsect_dom_noshift}. 
\begin{proof}[Proof of Theorem \ref{thm:Rsect_dom_noshift}]
\textit{Proof of \ref{it:thm:Rsect_dom_noshift1}.} 
    Since the embedding $W^{1,p}(\OO, w_{\gam_0}^{\d\OO})\hookrightarrow L^p(\OO, w_{\gam_0}^{\d\OO})$ is compact for all $p\in(1,\infty)$ and $\gam_0>-1$ (see \cite[Theorem 8.8]{GO89}), it follows that the embedding $\DD^{k,p}_\gam(\OO)\hookrightarrow \XX^{k,p}_\gam(\OO)$ is compact as well. Hence, the resolvent operator $R(\lambda, A)$ with $\lambda\in \rho(A)$ is compact. By the Riesz--Schauder theorem for compact operators, the resolvent operator $R(\lambda, A)$ has a discrete countable spectrum $\{\sigma_j: j\in\NN_0\}$, where $\sigma_j\neq 0$ are eigenvalues. Furthermore, zero is contained in the spectrum of $R(\lambda, A)$ and is the only accumulation point of the spectrum. Therefore, the spectral mapping theorem implies that $\sigma(A)$ is a discrete set of eigenvalues. 

    To continue, we prove that the spectrum $\sigma(A)$ is independent of $p$, $k$ and $\gam$. Let $A_{p,k, \gam}$ and $\mc{A}$ be as in Lemma \ref{lem:consis}. It suffices to prove that $\sigma(A_{p,k,\gam})=\sigma(\mc{A})$. Note that $\sigma(\mc{A})$ is a discrete set of eigenvalues contained in the open right half-plane by Lemma \ref{lem:spA}. By Lemma \ref{lem:consis} and analytic continuation it holds that $R(z, A_{p,k,\gam})$ and $R(z, \mc{A})$ are consistent for all $z\in \rho(A_{p,k,\gam})\cap \rho(\mc{A})$. If $\mu\in \rho(\mc{A})$, then since $\sigma(A_{p,k, \gam})$ is discrete and countable, there exists an $r>0$ such that $\overline{B(\mu, r)}\setminus\{\mu\}\subseteq \rho(A_{p,k,\gam})\cap \rho(\mc{A})$. Therefore, by consistency of the resolvents, we find
    \begin{equation*}
        \int_{\d B(\mu,r)}R(z, A_{p,k, \gam})\dd z = \int_{\d B(\mu,r)} R(z, \mc{A})\dd z =0,
    \end{equation*}
    which implies that $\mu\in \rho(A_{p,k, \gam})$. The other inclusion follows similarly. This proves that $\sigma(A_{p,k,\gam})=\sigma(\mc{A})$ and thus $\sigma(A_{p,k,\gam})$ is independent of $p$, $k$ and $\gam$.

    Using the properties of $\sigma(\mc{A})$ from Lemma \ref{lem:spA}, now yields that $\sigma(A_{p,k,\gam})$ is a discrete set of eigenvalues contained in the open right half-plane.

    \textit{Proof of \ref{it:thm:Rsect_dom_noshift2}.} By Theorem \ref{thm:Rsect_dom}, it holds for $\lambda_0$ large enough that $\lambda_0+ A$ is $R$-sectorial with $\om_R(\lambda_0+A)<\frac\pi 2$.  
    From this and the fact that the spectrum of $A$ is discrete, we see that 
    for each $r>0$ only finitely many eigenvalues $\lambda$ with $\Re\lambda<r$ exist. As the imaginary axis belongs to $\rho(A)$ by \ref{it:thm:Rsect_dom_noshift1}, this yields 
    $\min_{\lambda\in \sigma(A)}\Re \lambda >0$. We set $\tilde\lambda := \frac12\min_{\lambda\in \sigma(A)}\Re \lambda $. Using the analyticity of the resolvent, it follows that there exists some $\sigma<\frac\pi 2$ such that for $\lambda_0> -\tilde{\lambda}$,  the operator $\lambda_0+ A$ is $R$-sectorial with $\om_R(\lambda_0+A)\leq \sigma$.
\end{proof}

\section{Maximal regularity for the damped plate equation}\label{sec:MR}
We prove our main results on maximal regularity (see Section \ref{subsec:prelim-Rsect}) for the damped plate equation on the half-space and on bounded domains.

\subsection{Damped plate equation with homogeneous boundary conditions}\label{sec:plateeq_hom}
We first study the damped plate equation with homogeneous boundary conditions, i.e.,
\begin{equation}\label{eq:plate_eq_MR_hom}
\left\{
  \begin{aligned}
  &\d_t^2 u +\del^2 u-\rho\del \d_t u = f\qquad &(t,x)&\in (0,T)\times \OO,\\
   &u=\d_{\vec{n}}u= 0\qquad &(t,x)&\in (0,T)\times \d\OO,\\
   &u|_{t=0}=(\d_tu)|_{t=0}=0\qquad &x&\in \OO.
\end{aligned}\right.
\end{equation}
Throughout, $\rho>0$ is fixed, we write $I:=(0,T)$ for some $T\in(0,\infty]$ and $\vec{n}$ denotes the  inward pointing unit normal vector at $\d\OO$. Furthermore, for $q\in (1,\infty)$ and $\mu\in (-1,q-1)$, we denote by $v_\mu(t):=|t|^\mu$ a temporal weight.\\

We recall that the spaces $\XX^{k,p}_\gam(\OO)$ and $\DD^{k,p}_\gam(\OO)$ are as defined in \eqref{eq:XXDD_dom}. Furthermore, we introduce the following short-hand notation.  Let $p,q\in (1,\infty)$, $\gam\in (-1,p-1)$, $\mu\in (-1,q-1)$, $k\geq 2$ an integer and $T\in (0, \infty]$. We define the space for the data $f$ as
\begin{align*}
    \F^{k-2}_{q,p}(I, v_\mu; \OO, w_{\gam+kp}^{\d\OO})&:= L^q(I, v_\mu; W^{k-2,p}(\OO, w_{\gam+kp}^{\d\OO})).
\end{align*}
For the solution $u$ we define the spaces
\begin{align*}
    \U^{k-2}_{q,p}(I, v_\mu; \OO, w^{\d\OO}_{\gam+kp})&:=  W^{2,q}(I, v_\mu; W^{k-2,p}(\OO, w_{\gam+kp}^{\d\OO}))\cap L^q(I, v_\mu; W^{k+2,p}(\OO, w_{\gam+kp}^{\d\OO})),\\
    \prescript{}{0}{\U}^{k-2}_{q,p}(I, v_\mu; \OO, w_{\gam+kp}^{\d\OO})&:= \big\{u\in \U^{k-2}_{q,p}(I, v_\mu; \OO, w_{\gam+kp}^{\d\OO}): u|_{t=0}=(\d_t u)|_{t=0}=0\big\}.
\end{align*}
Similarly, we define the spaces with zero boundary conditions at $\d\OO$ as
\begin{align*}
    \U^{k-2}_{q,p,0}(I, v_\mu; \OO, w^{\d\OO}_{\gam+kp})&:=  W^{2,q}(I, v_\mu; W^{k-2,p}(\OO, w_{\gam+kp}^{\d\OO}))\cap L^q(I, v_\mu; W^{k+2,p}_0(\OO, w_{\gam+kp}^{\d\OO})),\\
    \prescript{}{0}{\U}^{k-2}_{q,p,0}(I, v_\mu; \OO, w_{\gam+kp}^{\d\OO})&:= \big\{u\in \U^{k-2}_{q,p,0}(I, v_\mu; \OO, w_{\gam+kp}^{\d\OO}): u|_{t=0}=(\d_t u)|_{t=0}=0\big\}.
\end{align*}

Using $R$-sectoriality of the operator matrix $A$ obtained in Sections \ref{sec:Rsect} and \ref{sec:Rsect_domains} for the half-space and bounded domains, respectively, we can establish maximal regularity for the damped plate equation. 
\begin{theorem}\label{thm:MR_homBC}
  Let $p,q\in(1,\infty)$, $\gam\in (-1,p-1)$, $\mu\in (-1, q-1)$, $k\geq 2$ an integer and $\rho>0$. Assume that either
  \begin{enumerate}[(i)]
      \item $\OO=\RRdh$ and $T\in(0,\infty)$, or,
      \item $\kappa\in (0,1)$, $\gam>(1-\kappa)p-1$, $\OO$ is a bounded $C^{1,\kappa}$-domain and $T\in (0,\infty]$.
  \end{enumerate}
 Then the following assertions hold.  
  \begin{enumerate}[(i)]
    \item\label{it:thm:MR_homBC1} The realisation of the operator $A$ in \eqref{eq:A_matrix} on $\XX^{k,p}_\gam(\OO)$ with domain $\DD^{k,p}_\gam(\OO)$ has maximal $L^q(v_\mu)$-regularity on $I$. Furthermore, $-A$ generates an analytic $C_0$-semigroup on $\XX^{k,p}_{\gam}(\OO)$. Moreover, if $\OO$ is a bounded domain, then this semigroup is uniformly exponentially stable.
    \item\label{it:thm:MR_homBC2} For all $f\in \F^{k-2}_{q,p}(I, v_\mu; \OO, w_{\gam+kp}^{\d\OO})$  there exists a unique solution $$u\in  \prescript{}{0}{\U}^{k-2}_{q,p,0}(I, v_\mu; \OO, w_{\gam+kp}^{\d\OO})$$ to \eqref{eq:plate_eq_MR_hom}.
     Moreover, this solution satisfies
        \begin{equation*}
          \|u\|_{\U^{k-2}_{q,p}(I, v_\mu; \OO, w^{\d\OO}_{\gam+kp})} \leq C \|f\|_{\F^{k-2}_{q,p}(I, v_\mu; \OO, w_{\gam+kp}^{\d\OO})},
        \end{equation*}
        where the constant $C>0$ only depends on $p,q, \gam, \mu, k, \rho, d$ and $T$. If $\OO$ is a bounded domain, then the constant $C$ is independent of $T$.
  \end{enumerate}
\end{theorem}
\begin{proof}
Theorem \ref{thm:Rsect_weight} (for $\OO=\RRdh$) and Theorem \ref{thm:Rsect_dom_noshift}\ref{it:thm:Rsect_dom_noshift2} (for $\OO$ a bounded domain) together with \cite[Theorems 17.3.1(2), 17.2.26(1) \& 17.2.39]{HNVW24} (using that $\CC$ is a $\UMD$ Banach space and $\vartheta(\rho)<\frac{\pi}{2}$)
   gives maximal $L^q(v_\mu)$-regularity of $A$ on $I$. Furthermore, \cite[Theorem 17.2.15]{HNVW24} implies that $-A$ generates an analytic $C_0$-semigroup. If $\OO$ is a bounded domain, then $0\in \rho(A)$ by Theorem \ref{thm:Rsect_dom_noshift}\ref{it:thm:Rsect_dom_noshift1}, hence \cite[Proposition 17.2.8 \& Corollary 17.2.25]{HNVW24} imply that this semigroup is uniformly exponentially stable. This proves \ref{it:thm:MR_homBC1}. For \ref{it:thm:MR_homBC2} it follows from \ref{it:thm:MR_homBC1} and \cite[Propositions 17.2.7 \& 17.2.8]{HNVW24} that there exists a unique solution $\tilde{u}=(\tilde{u}_1,\tilde{u}_2)^\top\in W^{1,q}(I, v_\mu; \XX^{k,p}_\gam(\OO))\cap L^q(I, v_\mu; \DD^{k,p}_\gam(\OO))$ to the first-order Cauchy problem
   \begin{equation}\label{eq:Cauchy_homBC}
     \d_t \tilde{u} + A\tilde{u} =\begin{pmatrix}
                    0 \\
                    f 
                  \end{pmatrix},\qquad \tilde{u}|_{t=0}=0. 
   \end{equation}
   Furthermore, this solution satisfies
   \begin{equation}\label{eq:est_v_homBC}
     \|\tilde{u}\|_{W^{1,q}(I, v_\mu; \XX^{k,p}_\gam(\OO))}+ \|\tilde{u}\|_{L^q(I, v_\mu; \DD^{k,p}_\gam(\OO))}\leq  C \|f\|_{L^q(I, v_\mu; W^{k-2,p}(\OO, w_{\gam+kp}^{\d\OO}))}.
   \end{equation}
   Next, by setting $u:=\tilde{u}_1$ we obtain that $\d_t u = \tilde{u}_2$ by the first component in \eqref{eq:Cauchy_homBC}. Hence, $u$ solves \eqref{eq:plate_eq_MR_hom}. Moreover, it holds that $u\in \prescript{}{0}{\U}^{k-2}_{q,p,0}(I, v_\mu; \OO, w_{\gam+kp}^{\d\OO})$ and 
   \begin{align*}
     \|u\|_{\U^{k-2}_{q,p}(I, v_\mu; \OO, w^{\d\OO}_{\gam+kp})}  \lesssim &\; \sum_{j=0}^1\|\d_t^ju\|_{W^{1,q}(I, v_\mu;W^{k-2,p}(\OO, w_{\gam+kp}^{\d\OO}))} 
      +\|u\|_{L^q(I, v_\mu; W^{k+2,p}(\OO, w_{\gam+kp}^{\d\OO}))}\\
      \lesssim &\; \|f\|_{\F^{k-2}_{q,p}(I, v_\mu; \OO, w_{\gam+kp}^{\d\OO})},
   \end{align*}
   where we have used that $(\tilde{u}_1, \tilde{u}_2)=(u, \d_t u)$ and \eqref{eq:est_v_homBC}.
   
   Conversely, if $u\in\prescript{}{0}{\U}^{k-2}_{q,p,0}(I, v_\mu; \OO, w_{\gam+kp}^{\d\OO})$ solves \eqref{eq:plate_eq_MR_hom}, then $\tilde{u}: =(u, \d_t u)^\top $ belongs to $W^{1,q}(I, v_\mu; \XX^{k,p}_\gam(\OO))\cap L^q(I, v_\mu; \DD^{k,p}_\gam(\OO))$ and solves \eqref{eq:Cauchy_homBC}. Indeed, this follows from the embedding 
   \begin{equation}\label{eq:emb_U}
       \U^{k-2}_{q,p,0}(I, v_\mu; \OO, w_{\gam+kp}^{\d\OO})\hookrightarrow W^{1,q}(I, v_\mu; W^{k,p}(\OO, w_{\gam+kp}^{\d\OO})).
   \end{equation}
   To prove \eqref{eq:emb_U}, note that by \cite[Proposition 5.5]{LMV17}, \cite[Theorem 3.18]{LV18} (which also holds on $I$ by a restriction argument) and \cite[Proposition 6.2]{Ro25} (using $k\geq 2$ and a localisation argument based on Proposition \ref{prop:isom} and Lemma \ref{lem:decomp}), we obtain
   \begin{equation*}
        \begin{aligned}
     W^{2,q}\big(I, v_\mu; &W^{k-2,p}(\OO, w_{\gam+kp}^{\d\OO})\big)\cap L^q\big(I, v_\mu; W^{k+2,p}_0(\OO, w_{\gam+kp}^{\d\OO})\big)\\
      & \hookrightarrow W^{1,q}\big(I, v_\mu; [W^{k-2,p}(\OO, w_{\gam+kp}^{\d\OO}), W^{k+2,p}_0(\OO, w_{\gam+kp}^{\d\OO})]_{\half}\big)\\
      & =W^{1,q}\big(I, v_\mu; W^{k,p}(\OO, w_{\gam+kp}^{\d\OO})\big). 
   \end{aligned}
   \end{equation*}
 Thus, the solution $u$ to \eqref{eq:plate_eq_MR_hom} is unique and this completes the proof of \ref{it:thm:MR_homBC2}. 
 \end{proof}

\subsection{Damped plate equation with inhomogeneous boundary conditions}
We turn to maximal regularity for the damped plate equation with inhomogeneous boundary conditions. Consider
\begin{equation}\label{eq:plate_eq_MR_inhom}
\left\{
  \begin{aligned}
  &\d_t^2 u +\del^2 u-\rho\del \d_t u = f\qquad &(t,x)&\in (0,T)\times \OO,\\
   &u= g_0\qquad &(t,x)&\in (0,T)\times \d\OO,\\
   &\d_{\vec{n}}u= g_1\qquad &(t,x)&\in (0,T)\times \d\OO,\\
   &u|_{t=0}=(\d_tu)|_{t=0}=0\qquad &x&\in \OO.
\end{aligned}\right.
\end{equation}
Again, $\rho>0$ is fixed, we write $I:=(0,T)$ for some $T\in(0,\infty]$ and $\vec{n}$ denotes the inward pointing unit normal vector at $\d\OO$. The functions $g_0$ and $g_1$ denote the boundary data. As is known from the standard unweighted theory for boundary value problems \cite{DHP07}, the optimal space for the boundary data is an intersection space of Triebel--Lizorkin spaces in time and Besov spaces in space. Recently, in \cite{DR25} the trace theory has been extended to the weighted setting with rough domains. The theory for the temporal trace in the weighted setting is more delicate than in the classical unweighted setting, see, e.g., \cite[Section 7]{LV18}. Therefore, we will not consider non-zero initial data.\\

For the definition of weighted Triebel--Lizorkin spaces, we refer to, e.g., \cite[Section 3.1]{MV12}. Similar as in \cite{DR25} (and, e.g., \cite[Definition 3.6.1]{Tr78} in the case of smooth domains), we define spaces on $\d\OO$ as follows. Let $\OO$ be a bounded $C^1$-domain and fix $(V_n)_{n=1}^N$, $(\OO_n)_{n=1}^N$ and  $(\eta_n)_{n=0}^N$ from Lemma \ref{lem:decomp}. For $n\in\{1,\dots, N\}$, let $\Phi_n:\OO_n \to \RRdh$ be the Dahlberg--Kenig--Stein pullback from Lemma \ref{lem:localisation_weighted_blow-up}. Recall that for any $f$ defined on $\overline{\OO_n}$, the change of coordinates mapping is $(\Phi_n)_*f:= f\circ \Phi^{-1}_n$. Furthermore, for any $g$ defined on $\d\OO_n$, we write $(\Phi_n (0,\cdot))_* g = g(\Phi^{-1}_n(0, \cdot))$ which defines a function on $\RR^{d-1}$.

Let $p\in (1,\infty)$ and $s>0$. Let $L^p(\d\OO)$ be the Lebesgue space with respect to the surface measure. 
We define 
    \begin{align*}
        B^{s}_{p,q}(\d\OO):=
        \big\{g\in L^p(\d\OO): (\Phi_n(0,\cdot))_* \eta_n^2 g \in B^s_{p,q}(\RR^{d-1})\; \text{for all } n\in \{1,\dots, N\}\big\},
    \end{align*}
equipped with the norm
\begin{align*}
    \|g\|_{B^s_{p,q}(\d\OO)}&:=\sum_{n=1}^N\|(\Phi_n(0,\cdot))_*\eta_n^2 g\|_{B^s_{p,q}(\RR^{d-1})},\qquad g\in B^s_{p,q}(\d\OO).
\end{align*}
Note that this definition is independent of the partition of unity and the diffeomorphisms up to equivalent norms. \\

In addition to the notation in Section \ref{sec:plateeq_hom}, we introduce the following short-hand notation for the spaces for the boundary data. Let $j\in \{0,1\}$, $p,q\in (1,\infty)$, $\gam\in (-1,p-1)$, $\mu\in (-1,q-1)$, $k\geq 2$ an integer and $t\in (0, \infty]$. We define
\begin{align*}
    \G^{j,\gam}_{q,p}(I, v_\mu; \d\OO) &:= F_{q,p}^{1-\frac{j}{2}-\frac{\gam+1}{2p}}(I, v_\mu; L^p(\d\OO))\cap L^q(I, v_\mu; B^{2-j-\frac{\gam+1}{p}}_{p,p}(\d\OO)),\\
    \prescript{}{0}{\G}^{j,\gam}_{q,p}(I, v_\mu; \d\OO)&:=\begin{cases}
        \G^{j,\gam}_{q,p}(I, v_\mu;\d\OO)& \mbox{ if }1-\frac{j}{2}-\frac{\gam+1}{2p}< \frac{\mu+1}{q},\\
        \big\{g\in \G^{j,\gam}_{q,p}(I, v_\mu; \d\OO): g|_{t=0}=0\big\}&\mbox{ if }1-\frac{j}{2}-\frac{\gam+1}{2p}> \frac{\mu+1}{q}.
    \end{cases}
\end{align*}
The temporal traces in the above spaces are well defined, see, e.g., \cite[Theorem 1.2]{Ro25}. Furthermore, it should be noted that the first-order temporal trace in $\G^{j,\gam}_{q,p}(I, v_\mu;\d\OO)$ only exists if
\begin{equation*}
    1-\frac{j}{2}-\frac{\gam+1}{2p}> 1+\frac{\mu+1}{q}.
\end{equation*}
Using that $j\in\{0,1\}$, $\gam\in(-1,p-1)$ and $\mu\in (-1, q-1)$, we see that $\frac{j}{2}+\frac{\gam+1}{2p}> 0 > -\frac{\mu+1}{q}$. Hence, for the given parameter ranges the first-order temporal trace never exists and, therefore, no compatibility conditions on $\d_t g_0$ and $\d_t g_1$ are needed. This is in contrast to the unweighted setting, see \cite[Theorem 4.6 \& 5.2]{DS15} and, in particular, \cite[Equation (4.10)]{DS15}.\\

For $\OO$ a bounded $C^1$-domain and $u\in C^1(\overline{\OO})$, we define the spatial trace as
\begin{align*}
    \overline{\Tr}^{\d\OO}_1:=(u|_{\d\OO}, \vec{n}\cdot (\grad u)|_{\d\OO}),
\end{align*}
where $\vec{n}$ denotes the inner unit normal vector at $\d\OO$. 

The following theorem gives the optimal space for the boundary data in the weighted setting. It should be noted that, due to the spatial weights, the trace space is \emph{independent} of the smoothness parameter $k$. Also note that more smoothness of the domain is required: $C^{1,\kappa}$, while for homogeneous boundary conditions we could allow for $C^1$-domains.
\begin{theorem}[Spatial trace operator]\label{thm:trace}
    Let $p,q\in (1,\infty)$, $\gam\in (-1,p-1)$, $\mu\in (-1, q-1)$, $k\geq 2$ an integer and $T\in (0, \infty]$. Furthermore, assume that either
    \begin{enumerate}[(i)]
        \item  $\OO=\RRdh$, or, 
        \item $\kappa\in (0,1)$, $\gam > (1-\kappa)p-1$ and $\OO$ is a bounded $C^{1,\kappa}$-domain.
    \end{enumerate}
    Then the trace operator
    \begin{equation*}
        \overline{\Tr}^{\d\OO}_1:\prescript{}{0}{\U}^{k-2}_{q,p}(I, v_\mu; \OO, w^{\d\OO}_{\gam+kp})\to\prod_{j=0}^1 \prescript{}{0}{\G}^{j,\gam}_{q,p}(I, v_\mu; \d\OO)
     \end{equation*}
    is continuous and surjective.
\end{theorem}
\begin{proof}
    If $I=\RR$, then the result follows from \cite[Theorems 4.1 \& 5.3]{DR25} applied with $k, k_1, \ell$ and $\gam$ replaced by $k-2$, $4$, $2$ and $\gam+2p$, respectively. By using the zero extension operator from $\RR_+$ to $\RR$, we can argue similarly as in Step 2 of the proof of \cite[Theorem 4.1]{DR25} to obtain the result for $I=\RR_+$. The result for $I=(0,T)$ with $T<\infty$ follows by using an extension operator from $I$ to $\RR$. Such an extension operator can be obtained as in \cite[Proposition 2.5]{AV22}.
\end{proof}

We can now prove the main theorem of this paper concerning well-posedness and regularity of the plate equation \eqref{eq:plate_eq_MR_inhom} with inhomogeneous boundary conditions.
\begin{theorem}\label{thm:main_inhomgenous}
  Let $p,q\in(1,\infty)$, $\gam\in (-1,p-1)$, $\mu\in (-1,q-1)$, $k\geq 2$ an integer and $\rho>0$. Suppose that $1-\frac{\gam+1}{2p}\neq \frac{\mu+1}{q}$ and $\half-\frac{\gam+1}{2p}\neq \frac{\mu+1}{q}$. Moreover, assume that either
  \begin{enumerate}[(i)]
      \item  $\OO=\RRdh$ and $T\in(0,\infty)$, or,
      \item $\kappa\in (0,1)$, $\gam> (1-\kappa)p-1$, $\OO$ is a bounded $C^{1,\kappa}$-domain and $T\in(0,\infty]$.
  \end{enumerate}
  Then for all 
  \begin{equation*}
      f\in \F^{k-2}_{q,p}(I, v_\mu; \OO, w_{\gam+kp}^{\d\OO}), \quad g_0\in\prescript{}{0}{\G}_{q,p}^{0,\gam}(I, v_\mu; \d\OO)\quad \text{ and }\quad g_1\in \prescript{}{0}{\G}^{1,\gam}_{q,p}(I, v_\mu;\d\OO),
  \end{equation*}
  there exists a unique solution $u\in \prescript{}{0}{\U}^{k-2}_{q,p}(I, v_\mu; \OO, w_{\gam+kp}^{\d\OO})$
  to \eqref{eq:plate_eq_MR_inhom}. Moreover, this solution satisfies
  \begin{equation*}
    \|u\|_{\U^{k-2}_{q,p}(I, v_\mu; \OO, w_{\gam+kp}^{\d\OO})}\leq C\big( \|f\|_{\F^{k-2}_{q,p}(I, v_\mu; \OO, w_{\gam+kp}^{\d\OO})}+ \|g_0\|_{\G_{q,p}^{0,\gam}(I, v_\mu; \d\OO)}+ \|g_1\|_{\G_{q,p}^{1,\gam}(I, v_\mu; \d\OO)}\big),
  \end{equation*}
  where the constant $C>0$ only depends on $p,q,\gam, \mu, k, \rho, d$ and $T$. If $\OO$ is a bounded domain, then the constant $C$ is independent of $T$.
\end{theorem}
\begin{proof}
From Theorem \ref{thm:trace} we obtain a continuous extension operator
\begin{equation}\label{eq:ext}
    \overline{\ext}^{\d\OO}_1: \prod_{j=0}^1 \prescript{}{0}{\G}^{j,\gam}_{q,p}(I, v_\mu; \d\OO)\to \prescript{}{0}{\U}^{k-2}_{q,p}(I, v_\mu; \OO, w^{\d\OO}_{\gam+kp}).
\end{equation}

  First, consider the damped plate equation
  \begin{equation}\label{eq:plate_eq_MR_hom_proof}
\left\{
  \begin{aligned}
  &\d_t^2 \tilde{u} +\del^2 \tilde{u}-\rho\del \d_t\tilde{u} = \tilde{f}\qquad &(t,x)&\in (0,T)\times \OO,\\
   &\tilde{u}=\d_{\vec{n}}\tilde{u}= 0\qquad &(t,x)&\in (0,T)\times \d\OO,\\
   &\tilde{u}|_{t=0}=(\d_t\tilde{u})|_{t=0}=0\qquad &x&\in \OO.
\end{aligned}\right.
\end{equation}
where
   \begin{equation*}
     \tilde{f}:=f-\big(\d_t^2+\del^2 -\rho\del \d_t \big)\overline{\ext}^{\d\OO}_1(g_0,g_1).
   \end{equation*}
  By Theorem \ref{thm:MR_homBC} there exists a solution
  \begin{equation*}
      \tilde{u}\in W^{2,q}(I, v_\mu; W^{k-2,p}(\OO, w^{\d\OO}_{\gam+kp}))\cap L^q(I, v_\mu; W^{k+2,p}_0(\OO, w^{\d\OO}_{\gam+kp}))\hookrightarrow \U^{k-2}_{q,p}(I, v_\mu; \OO, w_{\gam+kp}^{\d\OO})
  \end{equation*}
  to \eqref{eq:plate_eq_MR_hom_proof} which satisfies the estimate
  \begin{equation}\label{eq:est_tildeu}
  \begin{aligned}
      \|\tilde{u}&\|_{\U^{k-2}_{q,p}(I, v_\mu; \OO, w_{\gam+kp}^{\d\OO})}\\
      &\lesssim \|f\|_{\F^{k-2}_{q,p}(I, v_\mu; \OO, w_{\gam+kp}^{\d\OO})} + \big\| \big(\d_t^2+\del^2 -\rho\del \d_t \big)\overline{\ext}^{\d\OO}_1(g_0,g_1)\big\|_{\F^{k-2}_{q,p}(I, v_\mu; \OO, w_{\gam+kp}^{\d\OO})}\\
      &\lesssim \|f\|_{\F^{k-2}_{q,p}(I, v_\mu; \OO, w_{\gam+kp}^{\d\OO})} + \| \overline{\ext}^{\d\OO}_1(g_0,g_1)\|_{\U^{k-2}_{q,p}(I, v_\mu; \OO, w_{\gam+kp}^{\d\OO})},
  \end{aligned}
  \end{equation}
  where for the last estimate we have also used the embedding 
  \begin{equation*}
      \U^{k-2}_{q,p}(I, v_\mu; \OO, w_{\gam+kp}^{\d\OO}) \hookrightarrow W^{1,q}(I, v_\mu; W^{k,p}(\OO, w_{\gam+kp}^{\d\OO}))
  \end{equation*}
 This embedding can be proved similarly to \eqref{eq:emb_U}, using the localisation procedure in Proposition \ref{prop:isom} and Lemma \ref{lem:decomp}. Note that we may apply Proposition \ref{prop:isom} for spaces without zero boundary conditions, since $\gam+kp > (k+1-\kappa)p-1$.
  
  We now set $u:=\tilde{u}+\overline{\ext}^{\d\OO}_1(g_0,g_1)$. Then, $u$ satisfies 
  \begin{equation*}
    \overline{\Tr}^{\d\OO}_1 u =\overline{\Tr}^{\d\OO}_1 \tilde{u} + \overline{\Tr}^{\d\OO}_1\overline{\ext}^{\d\OO}_1 (g_0,g_1) = (g_0,g_1),
  \end{equation*} 
  by \eqref{eq:plate_eq_MR_hom_proof} and properties of the extension operator. Furthermore, by \eqref{eq:ext} and \eqref{eq:plate_eq_MR_hom_proof}, we have 
  \begin{align*}
        u|_{t=0}& = \tilde{u}|_{t=0} + \overline{\ext}^{\d\OO}_1 (g_0,g_1)|_{t=0}=0,\\
         (\d_t u)|_{t=0} &= (\d_t\tilde{u})|_{t=0} + (\d_t \overline{\ext}^{\d\OO}_1 (g_0,g_1))|_{t=0}=0.
      \end{align*}
So, by construction of $\tilde{u}$, we have that $u$ satisfies \eqref{eq:plate_eq_MR_inhom} and we have the estimate
\begin{align*}
  \|u\|_{\U^{k-2}_{q,p}(I, v_\mu; \OO, w_{\gam+kp}^{\d\OO}) } & \lesssim \|\tilde{u}\|_{\U^{k-2}_{q,p}(I, v_\mu; \OO, w_{\gam+kp}^{\d\OO}) } + \| \overline{\ext}^{\d\OO}_1(g_0,g_1)\|_{\U^{k-2}_{q,p}(I, v_\mu; \OO, w_{\gam+kp}^{\d\OO}) }\\
  &\lesssim \|f\|_{\F^{k-2}_{q,p}(I, v_\mu; \OO, w_{\gam+kp}^{\d\OO})}+\| \overline{\ext}^{\d\OO}_1(g_0,g_1)\|_{\U^{k-2}_{q,p}(I, v_\mu; \OO, w_{\gam+kp}^{\d\OO})} \\
   & \lesssim \|f\|_{\F^{k-2}_{q,p}(I, v_\mu; \OO, w_{\gam+kp}^{\d\OO})} + \|g_0\|_{\G^{0,\gam}_{q,p}(I, v_\mu; \d\OO)}+ \|g_1\|_{\G^{1,\gam}_{q,p}(I, v_\mu;\d\OO)},
\end{align*}
  using \eqref{eq:est_tildeu} and \eqref{eq:ext}. Finally, the uniqueness of the solution $u$ to \eqref{eq:plate_eq_MR_inhom} follows from Theorem \ref{thm:MR_homBC}. 
\end{proof}

\bibliographystyle{plain}
\bibliography{references_dampedplate}

\begin{thebibliography}{10}

\bibitem{Ab12}
H.~Abels.
\newblock {\em Pseudodifferential and singular integral operators}.
\newblock De Gruyter Graduate Lectures. De Gruyter, Berlin, 2012.

\bibitem{ADN59}
S.~Agmon, A.~Douglis, and L.~Nirenberg.
\newblock Estimates near the boundary for solutions of elliptic partial
  differential equations satisfying general boundary conditions. {I}.
\newblock {\em Comm. Pure Appl. Math.}, 12:623--727, 1959.

\bibitem{ADN64}
S.~Agmon, A.~Douglis, and L.~Nirenberg.
\newblock Estimates near the boundary for solutions of elliptic partial
  differential equations satisfying general boundary conditions. {II}.
\newblock {\em Comm. Pure Appl. Math.}, 17:35--92, 1964.

\bibitem{AV22}
A.~Agresti and M.C. Veraar.
\newblock {Nonlinear parabolic stochastic evolution equations in critical
  spaces Part I. Stochastic maximal regularity and local existence}.
\newblock {\em Nonlinearity}, 35(8):4100–4210, 2022.

\bibitem{Am95}
H.~Amann.
\newblock {\em {L}inear and quasilinear parabolic problems. {V}ol. {I}:
  abstract linear theory}, volume~89 of {\em Monographs in Mathematics}.
\newblock Birkh\"auser Boston Inc., Boston, MA, 1995.

\bibitem{Andersen80}
K.F. Andersen.
\newblock Weighted inequalities for the {S}tieltjes transformation and
  {H}ilbert's double series.
\newblock {\em Proc. Roy. Soc. Edinburgh Sect. A}, 86(1-2):75--84, 1980.

\bibitem{CCD08}
A.N. Carvalho, J.W. Cholewa, and T.~Dlotko.
\newblock Strongly damped wave problems: bootstrapping and regularity of
  solutions.
\newblock {\em J. Differ. Equ.}, 244(9):2310--2333, 2008.

\bibitem{Chen-Triggiani89}
S.P. Chen and R.~Triggiani.
\newblock Proof of extensions of two conjectures on structural damping for
  elastic systems.
\newblock {\em Pacific J. Math.}, 136(1):15--55, 1989.

\bibitem{CS05}
R.~Chill and S.~Srivastava.
\newblock {$L^p$}-maximal regularity for second order {C}auchy problems.
\newblock {\em Math. Z.}, 251(4):751--781, 2005.

\bibitem{DD11}
R.~Denk and M.~Dreher.
\newblock Resolvent estimates for elliptic systems in function spaces of higher
  regularity.
\newblock {\em Electron. J. Differ. Equ.}, 2011(109):1--12, 2011.

\bibitem{DHP03}
R.~Denk, M.~Hieber, and J.~Pr\"uss.
\newblock {$\mathcal{R}$}-boundedness, {F}ourier multipliers and problems of
  elliptic and parabolic type.
\newblock {\em Mem. Amer. Math. Soc.}, 166(788), 2003.

\bibitem{DHP07}
R.~Denk, M.~Hieber, and J.~Pr\"uss.
\newblock Optimal {$L^p$}-{$L^q$}-estimates for parabolic boundary value
  problems with inhomogeneous data.
\newblock {\em Math. Z.}, 257(1):193--224, 2007.

\bibitem{DK13}
R.~Denk and M.~Kaip.
\newblock {\em General parabolic mixed order systems in $L_p$ and
  applications}, volume 239 of {\em Operator Theory: Advances and
  Applications}.
\newblock Birkh\"auser/Springer, 2013.

\bibitem{DK18_transmission}
R.~Denk and F.~Kammerlander.
\newblock Exponential stability for a coupled system of damped-undamped plate
  equations.
\newblock {\em IMA J. Appl. Math.}, 83(2):302--322, 2018.

\bibitem{DR25}
R.~Denk and F.B. Roodenburg.
\newblock {Trace theory for parabolic boundary value problems with rough
  boundary conditions}.
\newblock arXiv:2512.15382, 2025.

\bibitem{Denk-Saal-Seiler08}
R.~Denk, J.~Saal, and J.~Seiler.
\newblock Inhomogeneous symbols, the {N}ewton polygon, and maximal
  {$L^p$}-regularity.
\newblock {\em Russ. J. Math. Phys.}, 15(2):171--191, 2008.

\bibitem{DS15}
R.~Denk and R.~Schnaubelt.
\newblock A structurally damped plate equation with {D}irichlet-{N}eumann
  boundary conditions.
\newblock {\em J. Differ. Equ.}, 259(4):1323--1353, 2015.

\bibitem{Denk-Volevich02}
R.~Denk and L.R. Volevich.
\newblock Elliptic boundary value problems with large parameter for mixed order
  systems.
\newblock In {\em Partial differential equations}, volume 206 of {\em Amer.
  Math. Soc. Transl. Ser. 2}, pages 29--64. Amer. Math. Soc., Providence, RI,
  2002.

\bibitem{EN00}
K.-J. Engel and R.~Nagel.
\newblock {\em One-parameter semigroups for linear evolution equations}, volume
  194 of {\em Graduate Texts in Mathematics}.
\newblock Springer-Verlag, New York, 2000.

\bibitem{Ev10}
L.C. Evans.
\newblock {\em Partial differential equations}, volume~19 of {\em Graduate
  Studies in Mathematics}.
\newblock American Mathematical Society, second edition, 2010.

\bibitem{FHL19}
S.~Fackler, T.P. Hyt{\"o}nen, and N.~Lindemulder.
\newblock {Weighted estimates for operator-valued Fourier multipliers}.
\newblock {\em Collect. Math.}, 71:511 -- 548, 2019.

\bibitem{FO91}
A.~Favini and E.~Obrecht.
\newblock Conditions for parabolicity of second order abstract differential
  equations.
\newblock {\em Differ. Integral Equ.}, 4(5):1005--1022, 1991.

\bibitem{Gindikin-Volevich92}
S.~Gindikin and L.R. Volevich.
\newblock {\em The method of {N}ewton's polyhedron in the theory of partial
  differential equations}, volume~86 of {\em Mathematics and its Applications
  (Soviet Series)}.
\newblock Kluwer Academic Publishers Group, Dordrecht, 1992.
\newblock Translated from the Russian manuscript by V. M. Volosov.

\bibitem{GW03}
M.~Girardi and L.~Weis.
\newblock Criteria for {R}-boundedness of operator families.
\newblock In {\em Evolution equations}, volume 234 of {\em Lecture Notes in
  Pure and Appl. Math.}, pages 203--221. Dekker, New York, 2003.

\bibitem{Gr14_classical_3rd}
L.~Grafakos.
\newblock {\em Classical {F}ourier analysis}, volume 249 of {\em Graduate Texts
  in Mathematics}.
\newblock Springer, New York, third edition, 2014.

\bibitem{GO89}
P.~Gurka and B.~Opic.
\newblock {Continuous and compact imbeddings of weighted Sobolev spaces. II}.
\newblock {\em Czechoslov. Math. J.}, 39(1):78--94, 1989.

\bibitem{Ha06}
M.H.A. Haase.
\newblock {\em The functional calculus for sectorial operators}, volume 169 of
  {\em Operator Theory: Advances and Applications}.
\newblock Birkh\"auser Verlag, Basel, 2006.

\bibitem{HL22}
F.~Hummel and N.~Lindemulder.
\newblock Elliptic and parabolic boundary value problems in weighted function
  spaces.
\newblock {\em Potential Anal.}, 57(4):601--669, 2022.

\bibitem{HNVW16}
T.P. Hyt\"onen, J.M.A.M.~van Neerven, M.C. Veraar, and L.~Weis.
\newblock {\em Analysis in {B}anach spaces. {V}olume {I}: {M}artingales and
  {L}ittlewood-{P}aley Theory}, volume~63 of {\em Ergebnisse der Mathematik und
  ihrer Grenzgebiete}.
\newblock Springer, 2016.

\bibitem{HNVW17}
T.P. Hyt\"onen, J.M.A.M.~van Neerven, M.C. Veraar, and L.~Weis.
\newblock {\em Analysis in {B}anach spaces. {V}olume {II}: Probabilistic
  methods and operator theory}, volume~67 of {\em Ergebnisse der Mathematik und
  ihrer Grenzgebiete}.
\newblock Springer, 2017.

\bibitem{HNVW24}
T.P. Hyt\"onen, J.M.A.M.~van Neerven, M.C. Veraar, and L.~Weis.
\newblock {\em Analysis in {B}anach spaces. {V}olume {III}: {H}armonic
  {A}nalysis and {S}pectral {T}heory}, volume~76 of {\em Ergebnisse der
  Mathematik und ihrer Grenzgebiete}.
\newblock Springer, 2024.

\bibitem{Ka96}
T.~Kato.
\newblock {\em Perturbation theory for linear operators}.
\newblock Classics in Mathematics. Springer-Verlag, Berlin, 1996.

\bibitem{KimD07}
D.~Kim.
\newblock {Trace theorems for Sobolev-Slobodeckij spaces with or without
  weights}.
\newblock {\em J. Funct. Spaces Appl.}, 5(3):243--268, 2007.

\bibitem{KK04}
K.-H. Kim and N.V. Krylov.
\newblock On the {S}obolev space theory of parabolic and elliptic equations in
  {$C^1$} domains.
\newblock {\em SIAM J. Math. Anal.}, 36(2):618--642, 2004.

\bibitem{KrBook08}
N.V. Krylov.
\newblock {\em {Lectures on elliptic and parabolic equations in Sobolev
  spaces}}, volume~{96} of {\em Graduate Studies in Mathematics}.
\newblock American Mathematical Society, 2008.

\bibitem{Ku85}
A.~Kufner.
\newblock {\em Weighted Sobolev spaces}.
\newblock John Wiley \& Sons, 1985.

\bibitem{KO1984}
A.~Kufner and B.~Opic.
\newblock {How to define reasonably weighted Sobolev spaces}.
\newblock {\em Comment. Math. Univ. Carol.}, 25(3):537--554, 1984.

\bibitem{KW04}
P.C. Kunstmann and L.~Weis.
\newblock Maximal {$L_p$}-regularity for parabolic equations, {F}ourier
  multiplier theorems and {$H^\infty$}-functional calculus.
\newblock In {\em Functional analytic methods for evolution equations}, volume
  1855 of {\em Lecture Notes in Math.}, pages 65--311. Springer, Berlin, 2004.

\bibitem{Lagnese89}
J.E. Lagnese.
\newblock {\em Boundary stabilization of thin plates}, volume~10 of {\em SIAM
  Studies in Applied Mathematics}.
\newblock Society for Industrial and Applied Mathematics (SIAM), Philadelphia,
  PA, 1989.

\bibitem{Lasiecka-Triggiani98}
I.~Lasiecka and R.~Triggiani.
\newblock Analyticity, and lack thereof, of thermo-elastic semigroups.
\newblock In {\em Control and partial differential equations
  ({M}arseille-{L}uminy, 1997)}, volume~4 of {\em ESAIM Proc.}, pages 199--222.
  Soc. Math. Appl. Indust., Paris, 1998.

\bibitem{Leis86}
R.~Leis.
\newblock {\em Initial-boundary value problems in mathematical physics}.
\newblock B.G. Teubner, Stuttgart; John Wiley \& Sons, Ltd., Chichester, 1986.

\bibitem{Li85}
G.M. Lieberman.
\newblock Regularized distance and its applications.
\newblock {\em Pacific J. Math.}, 117(2):329--352, 1985.

\bibitem{LLRV24}
N.~Lindemulder, E.~Lorist, F.B. Roodenburg, and M.C. Veraar.
\newblock Functional calculus on weighted {S}obolev spaces for the {L}aplacian
  on the half-space.
\newblock {\em J. Funct. Anal.}, 289(8):110985, 2025.

\bibitem{LLRV25}
N.~Lindemulder, E.~Lorist, F.B. Roodenburg, and M.C. Veraar.
\newblock Functional calculus on weighted {S}obolev spaces for the {L}aplacian
  on rough domains.
\newblock {\em J. Differ. Equ.}, 454:113884, 2026.

\bibitem{LMV17}
N.~Lindemulder, M.~Meyries, and M.C. Veraar.
\newblock Complex interpolation with {D}irichlet boundary conditions on the
  half line.
\newblock {\em Math. Nachr.}, 291(16):2435--2456, 2017.

\bibitem{LV18}
N.~Lindemulder and M.C. Veraar.
\newblock The heat equation with rough boundary conditions and holomorphic
  functional calculus.
\newblock {\em J. Differ. Equ.}, 269(7):5832--5899, 2020.

\bibitem{MV12}
M.~Meyries and M.C. Veraar.
\newblock Sharp embedding results for spaces of smooth functions with power
  weights.
\newblock {\em Studia Math.}, 208(3):257--293, 2012.

\bibitem{MV15}
M.~Meyries and M.C. Veraar.
\newblock Pointwise multiplication on vector-valued function spaces with power
  weights.
\newblock {\em J. Fourier Anal. Appl.}, 21(1):95--136, 2015.

\bibitem{Ne22}
J.M.A.M.~van Neerven.
\newblock {\em Functional Analysis}, volume 201 of {\em Cambridge Studies in
  Advanced Mathematics}.
\newblock Cambridge University Press, 2022.

\bibitem{Neuttiens-Sauer25}
G.~Neuttiens and J.~Sauer.
\newblock The time-periodic {C}ahn-{H}illiard-{G}urtin system on the half space
  as a mixed-order system with general boundary conditions.
\newblock arXiv:2512.23582, 2025.

\bibitem{Ro25}
F.B. Roodenburg.
\newblock {Complex interpolation of power-weighted Sobolev spaces with boundary
  conditions}.
\newblock To appear in Stud. Math. arXiv:2503.14636, 2025.

\bibitem{Russell84}
D.L. Russell.
\newblock Mathematical models for the elastic beam and their control-theoretic
  implications.
\newblock In {\em Semigroups, theory and applications, {V}ol. {II} ({T}rieste,
  1984)}, volume 152 of {\em Pitman Res. Notes Math. Ser.}, pages 177--216.
  Longman Sci. Tech., Harlow, 1986.

\bibitem{Russell92}
D.L. Russell.
\newblock {\em On Mathematical models for the elastic beam with
  frequency-proportional damping}, chapter~4, pages 125--169.
\newblock Society for Industrial and Applied Mathematics, 1992.

\bibitem{Savare98}
G.~Savar\'e.
\newblock Regularity results for elliptic equations in {L}ipschitz domains.
\newblock {\em J. Funct. Anal.}, 152(1):176--201, 1998.

\bibitem{SV10}
R.~Schnaubelt and M.C. Veraar.
\newblock Structurally damped plate and wave equations with random point force
  in arbitrary space dimensions.
\newblock {\em Differ. Integral Equ.}, 23(9-10):957--988, 2010.

\bibitem{Schrohe99}
E.~Schrohe.
\newblock Fr\'echet algebra techniques for boundary value problems on
  noncompact manifolds: {F}redholm criteria and functional calculus via
  spectral invariance.
\newblock {\em Math. Nachr.}, 199:145--185, 1999.

\bibitem{Schrohe01}
E.~Schrohe.
\newblock A short introduction to {B}outet de {M}onvel's calculus.
\newblock In {\em Approaches to singular analysis ({B}erlin, 1999)}, volume 125
  of {\em Oper. Theory Adv. Appl.}, pages 85--116. Birkh\"auser, Basel, 2001.

\bibitem{Se72}
R.T. Seeley.
\newblock Interpolation in {$L^p$} with boundary conditions.
\newblock {\em Studia Math.}, 44(1):47--60, 1972.

\bibitem{Tr78}
H.~Triebel.
\newblock {\em Interpolation theory, function spaces, differential operators},
  volume~18 of {\em North-Holland Mathematical Library}.
\newblock North-Holland Publishing Co., Amsterdam-New York, 1978.

\end{thebibliography}
\end{document}